\documentclass[11pt,reqno]{amsart}
\usepackage{mathrsfs}
\usepackage{url}
\usepackage{mathtools}
\usepackage{latexsym,epsfig,amssymb,amsmath,amsthm,color,url,bm}
\usepackage[inline,shortlabels]{enumitem}
\usepackage{hyperref}
\usepackage[foot]{amsaddr}
\usepackage{amsmath,amsbsy}
\usepackage{cleveref}
\usepackage{mwe}
\RequirePackage[numbers]{natbib}
\usepackage{mathptmx}
\usepackage[text={16cm,24cm}]{geometry}
\usepackage{xcolor}

\allowdisplaybreaks 
 \numberwithin{equation}{section}
\newtheorem{theorem}{Theorem}[section]
\newtheorem{prop}[theorem]{Proposition}
\newtheorem{lemma}[theorem]{Lemma}
\newtheorem{corollary}[theorem]{Corollary}

\newcommand\Item[1][]{%
  \ifx\relax#1\relax  \item \else \item[#1] \fi
  \abovedisplayskip=0pt\abovedisplayshortskip=0pt~\vspace*{-\baselineskip}}

\theoremstyle{definition}

\theoremstyle{definition}

\DeclareMathOperator{\Prob}{\mathbf{P}}
\DeclareMathOperator{\E}{\mathbf{E}}

\DeclareMathOperator{\dist}{dist}

\DeclareMathOperator{\indeg}{indeg}

\title[]{A model of opinion dynamics evolving via a preferential attachment mechanism involving multiple extractions}
\date{}
\author{Sooraj M,\ Moumanti Podder and Archi Roy}
\address{Sooraj M,\ Indian Institute of Science Education and Research (IISER) Pune, Dr.\ Homi Bhabha Road, Pashan, Pune 411008, Maharashtra, India.}
\address{Moumanti Podder, Indian Institute of Science Education and Research (IISER) Pune, Dr.\ Homi Bhabha Road, Pashan, Pune 411008, Maharashtra, India.}
\address{Archi Roy, Indian Institute of Management, Kozhikode 673570, Kerala, INDIA.}
\email{sooraj.m@students.iiserpune.ac.in}
\email{moumanti@iiserpune.ac.in}
\email{archiroy@iimk.ac.in}

\begin{document}
\bibliographystyle{plainnat}

\begin{abstract}
We study a model of opinion dynamics / social learning / peer-review-based market economics on an evolving network, wherein 
\begin{enumerate*}
\item each of the first $N$ agents adopts one of two available opinions arbitrarily, and
\item the $(n+1)$-st agent, for $n\geqslant N$, upon arrival, draws a sample of size $k_{n}$, with replacement, from the past agents, such that the $i$-th agent (for $i\leqslant n$) is included in the sample with probability proportional to the number of times they were \emph{previously} sampled and agreed with.
\end{enumerate*}
The $(n+1)$-st agent then decides which opinion to adopt 
\begin{enumerate*}
\item based on the proportion of sampled agents conforming to each of the two opinions, and
\item according to a stochastic update rule that involves a memory parameter and a rather general reinforcement function.
\end{enumerate*}
We study both 
\begin{enumerate*}
\item the scenario where $k_{n}=k$ remains fixed with $n$, and
\item the scenario where $k_{n}$ grows at a suitable rate with $n$.
\end{enumerate*}
This model can be represented as an evolving preferential attachment network wherein each vertex is endowed with one of two possible states, and all edges are directed. It can also be framed as a variant of the celebrated elephant random walk. We study the asymptotics of this stochastic process -- in particular, the almost sure convergence, and in case of fixed sample sizes, second order fluctuations, of the relative dominance of each opinion, the influence capital and overall network activity.
\end{abstract}

\subjclass[2020]{}


\maketitle
\section{Introduction}\label{sec:intro}

Opinion dynamics on networks provide a mathematical framework for understanding how local interactions among individuals give rise to large-scale collective behavior, with applications spanning a wide range of fields, including ecology, physics, and economics (see \cite{starnini2025opinion,shirzadi2025opinion,caldarelli2026physics} for a comprehensive review). In this paper, we consider an opinion dynamics model where sequentially arriving agents choose among a pair of competing \textit{opinions} (or \textit{types)} based on the the feedback of their past adopters. Specifically, we assume that agents interact over an evolving network denoted by the sequence $\{G_n\}$, in which each agent corresponds to a vertex and each interaction is represented by a \emph{directed} edge. The network is constructed according to the following mechanism. We begin with an initial population of $N\geq 1$ agents, each of whom chooses one of the two opinions arbitrarily. These agents are represented by $N$ isolated vertices in the initial graph $G_N$, and thereafter the network evolves by the sequential arrival of new agents. At the beginning of epoch $(n+1)$, for $n\geqslant N$, a new agent enters the scene and samples $k_n$ (where $k_n\leqslant n$) of the existing agents with replacement according to a linear preferential attachment mechanism \citep{barabasi1999emergence,albert2002statistical}, whereby the probability of sampling an agent is proportional to its current in-degree. We allow for $k_n$ to be either fixed for all sampling events, or to grow as the size of the network itself grows. Based on the proportion of sampled agents conforming to a particular opinion (out of the two available opinions), together with a \textit{memory parameter} and a \textit{reinforcement function}, the newcomer chooses one of the two opinions, and joins those among the sampled agents whose opinions agree with its own via outgoing edges (counting repetitions, i.e.\ a sampled agent with the same opinion as the newcomer receives as many incoming edges from the newcomer as the number of times they appeared in the sample). A more detailed description of the mathematical model is presented in \S\ref{sec:model}. 

In this work, we are specifically interested in the asymptotic behavior of three aggregate statistics: the \textit{relative dominance} of an opinion, defined as the proportion of agents in the network holding that opinion; its \textit{influence capital}, measured by the sum of the in-degrees of all agents holding that opinion; and the overall \textit{network activity}, represented by the total in-degree of all agents in the network. Each newly arriving agent updates these quantities through its sampling of $k_n$ past adopters (with replacement) and the pre-decided opinion-adoption mechanism. Consequently, an increment in each of these statistics is determined by a retrospective sample of size $k_n$ chosen from the entire history of the evolving network, making their evolution a long-memory self-interacting process. This structure is closely analogous to the \textit{elephant random walk} (ERW) introduced in \cite{schutz2004elephants}. In the classical ERW, a walker on the integer lattice possesses complete memory of its trajectory and, at each step, recalls one past step uniformly at random, following which it either repeats or reverses the recalled step with prespecified probabilities. Over the recent years, several extensions to this model have been proposed (see \cite{parra2026coordinatewiseelephantrandomwalk, nakano2025elephantrandomwalkpolynomially, bercu2025multidimensionalelephantrandomwalk, Roy_2025}, among recent developments). In a setup closely related to ours, \cite{baur2020class} replaced uniform sampling with sampling emulating linear preferential attachment -- however, here too, as in the classical ERW, a \emph{single} step is sampled from the past prior to deciding every new step. Our setting differs in a couple of crucial aspects:
\begin{enumerate*}
\item the $(n+1)$-st agent to arrive at the scene draws a sample of size $k_n\geqslant 1$, with replacement, from among the first $n$ agents, where $k_{n}$ may not only exceed $1$, but may also vary with $n$, and 
\item the stochastic rule via which the opinion to be adopted by the $(n+1)$-st agent is decided, based on the findings from the sample drawn, involves a rather general reinforcement function.
\end{enumerate*}
The first of these two aspects places our model in the framework of ERWs allowing multiple extractions, recently studied in \cite{franchini2025elephant, podder2026elephant, m2026elephantrandomwalkattributed}. Extending the framework of ERWs taking place on a fixed, unchanging integer lattice, involving multiple extractions but allowing only uniform sampling schemes, to a growing, directed network (that suitably captures the non-uniform, history-reliant sampling scheme previously described) presents a substantial amount of new analytical challenges that we tackle in the subsequent sections. 


\subsection{Organization of the paper} 
We begin by discussing the various motivations propelling an investigation of the models studied in this paper in \S\ref{sec:motivation}, following which the models are formally introduced in \S\ref{sec:model}. Two different models have been addressed in this paper, 
\begin{enumerate*}
\item the first of which allows each incoming agent (apart from the first $N$ agents) to draw a sample of a \emph{fixed} size, $k$, with replacement, from the set of past agents,
\item while in the second, the $(n+1)$-st agent, for $n\geqslant N$, draws a sample, with replacement, of size $k_{n}$ from the set of past agents, where $k_{n}$ is assumed to grow at a suitable rate with $n$.
\end{enumerate*}
The former has been described in \S\ref{subsec:model_fixed_sample_size}, while the latter appears in \S\ref{subsec:model_growing_sample_size}. The main results of this paper have been enumerated in \S\ref{sec:main_results}, with \S\ref{subsec:main_results_fixed_sample_size} dedicated to the statements of the results pertaining to the model in \S\ref{subsec:model_fixed_sample_size}, and \S\ref{subsec:main_results_growing_sample_size} containing the statements of the results pertaining to the model in \S\ref{subsec:model_growing_sample_size}. In \S\ref{sec:results_from_literature} we have included definitions from the literature on stability theory for dynamical systems (driven by differential equations or, more generally, differential inclusions) that are relevant for stating and / or proving the main results of this paper. We have also included the statements of three results from the literature (in particular, \cite{borkar2008stochastic} and \cite{benaim2005stochastic}) that are crucial for establishing our main results. The proofs of the results stated in \S\ref{subsec:main_results_fixed_sample_size} have been detailed in \S\ref{sec:proofs_fixed_sample_size}, while the results stated in \S\ref{subsec:main_results_growing_sample_size} have been proved in \S\ref{sec:proofs_growing_sample_size}.

\section{Motivation and related models}\label{sec:motivation}
Our model exhibits a natural co-evolution between network structure and opinion dynamics: new agents arrive and form opinions by preferentially sampling previous agents, and the resulting distribution of opinions across the network determines which agents gain influence, which in turn dictates how future opinions are formed. Such models in which the underlying network and the stochastic process unfolding on it co-evolve have been extensively studied in the literature (see \cite{gross2007adaptive} for an early survey). The model we consider is specifically similar to the one proposed in \cite{antunovic2016coexistence}, where a newcomer joining the (undirected) network of agents at time $t=n$ adopts an opinion with probability proportional to the prevalence of that opinion in a sample of neighbors chosen according to a preferential attachment mechanism. The model has since been extended in several subsequent works (see, for example, \cite{haslegrave2018non,jordan2018preferential,haslegrave2025competing}), and the central focus of these works has been to understand the asymptotic coexistence of the opinions in different regimes of adoption probabilities. More recently, \cite{bhamidi2026network} studied a related model of preferential attachment trees and developed a scaling theory for its growth dynamics. As mentioned in \S\ref{sec:intro}, a key contrast with previous works in the same vein lies in the choice of the statistics we study, as well as in the fact that we operate on a directed network rather than an undirected one. 

The proposed model is sufficiently flexible to capture a broad range of opinion dynamics. For example, it may be viewed as a variant of the voter model in a growing electorate, where the population expands over time rather than remaining fixed. Each new voter enters the network sequentially, samples from a collection of `influential' or `popular' voters, forms an opinion based on the opinions in the sampled group, and then connects to those sampled voters who share its adopted opinion. Unlike the classical voter model (\citep{holley1975ergodic}), where individuals repeatedly update their opinions by being swayed by their neighbors, here, each voter, upon adopting a certain opinion, retains that opinion forever. Moreover, the sampled voters are chosen according to a preferential attachment mechanism, so that voters who are, in some sense, more persuasive -- those who have accumulated \emph{many} incoming connections already -- are more likely to be consulted by future arrivals. Research on voter models on networks has grown substantially over the past decade (we refer the readers to \cite{rivera2018voting} for a comprehensive review). Broadly, the literature divides into two streams: voter models on static networks, where the underlying graph of connections is fixed and only opinions evolve, and voter models on co-evolving networks (also called dynamic or adaptive networks), where the network topology and opinions change simultaneously. Some important extensions of the classical voter model on static networks include 
\begin{enumerate*}
\item the noisy and heterogeneous voter models (\citep{granovsky1995noisy,masuda2010heterogeneous,baxter2011voter}) in which voters are also allowed to change their opinions spontaneously, independent of their neighbors, 
\item the $q$-voter model (\citep{castellano2009nonlinear,vieira2018threshold}) in which a voter updates its opinion according to the collective opinion (which is often the opinion of the majority, as considered in  \citep{chen2018phase,kanoria2011majority,zehmakan2020opinion,berkowitz2022central,haslegrave2017majority}) of a uniformly sampled group of neighbours, as opposed to copying a single neighbour, 
\item voter models with quenched disorder, in which some individuals act as contrarians by adopting the opinion that is opposite of that of a randomly selected neighbour (\citep{masuda2013voter}), or as zealots whose opinions remain fixed throughout the dynamics (\citep{mobilia2007role}), and 
\item discordant voter models (\citep{cooper2018discordant,avena2022discordant,capannoli2025evolution}) in which, upon interacting with a disagreeing neighbour, a voter may either adopt the neighbour's opinion or influence the neighbour to adopt their own opinion.
\end{enumerate*}
Several other extensions can be found in \cite{gastner2018consensus, gastner2019voter, ramirez2024ordering, avena2024meeting, llabres2026partisan}. Among co-evolving voter models, a popular class is rewiring models, in which a voter interacting with a disagreeing neighbour either adopts the neighbour's opinion or replaces the disagreeing connection with one to a like-minded voter. This rewiring mechanism was introduced by \cite{holme2006nonequilibrium} and subsequently studied in \cite{durrett2012graph,basu2017evolving,basak2015evolving}, with Erd\H{o}s–R\'{e}nyi random graph as the initial choice of networks. A few related models include
\begin{enumerate*}
\item \cite{eichhorn2026offended}, where, upon each interaction, either the two voters reach agreement or the edge between them is permanently deleted instead of being rewired,
\item a degree-preserving rewiring by \cite{avena2025voter}, where, during each epoch, a pair of edges is broken and the voters at the corresponding endpoints are randomly re-matched, and
\item \cite{malik2016transitivity} where voters are preferentially rewired to neighbours of neighbours.
\end{enumerate*}
More recently, an interesting model was proposed by \cite{astoquillca2026ergodicity}, where interactions along positively-signed edges cause the two endpoints to align their opinions, while interactions along negatively-signed edges drive the endpoints towards opposite opinions, and the edge signs themselves evolve independently over time via a dynamical percolation process. Several other interesting co-evolving voter models can be found in \cite{papanikolaou2022consensus, timpanaro2024emergence, kravitzch2023analysis, klamser2017zealotry, horstmeyer2020adaptive, choi2025analysis, min2023coevolutionary, baldassarri2024opinion, avila2026temporal}, among others. Most of these works focus on quantities such as time to reach consensus, the phase transition between consensus and fragmentation, and temporal conductance. An important feature of our model is the formation of connection through preferential attachment rather than through uniform sampling. In the existing literature on voter models, non-uniform sampling has been considered by \cite{baronchelli2011voter, lucke2026accurate} in the context of static networks with fixed edge-weights, where the voter adopts the opinion of a neighbour with probability proportional to the weight of the edge connecting them. The latter also allowed for connection with multiple neighbours in the network. A similar setup was considered in \cite{moinet2018generalized}, where each voter in the network was endowed with an attractiveness factor which modulated their probabilities of forming connections. In \cite{fernley2025phase,fernley2024discursive,fernley2023voter}, the voter model on scale-free networks was studied, including the Norros-Reittu model (\citep{norros2006conditionally}) in which a random number of edges, following a Poisson distribution, is formed by each new vertex, unlike the classical linear preferential attachment model (\citep{barabasi1999emergence,albert2002statistical}).

Our proposed model can also be viewed as a social learning model \citep{ellison1993rules}, in which economic agents in a social network repeatedly make decisions under uncertainty by learning from both their own experience (often termed as `private information') and the experiences of the neighbours. In the classical form of the social learning model, each vertex of the network represents an agent seeking to infer an unknown state of the world, which may be categorical (e.g.\ the optimal choice between two products) or continuous (e.g.\ the size of the U.S. national debt). In our model, each arriving agent (with no private information), rather than experimenting independently, observes the experiences of a sample of previous agents and infers which of two competing choices is more likely to be profitable. Agents who have been frequently sampled and imitated by later arrivals accumulate greater influence, making them more likely to be consulted by future agents and naturally giving rise to a preferential attachment mechanism. The social learning literature broadly distinguishes between Bayesian models, in which agents update beliefs optimally via Bayes' rule, and non-Bayesian models, in which beliefs evolve through simpler heuristic rules. Our model falls into the latter category. A seminal contribution to this tradition is \cite{golub2010naive}, where agents receive independent noisy private signals and iteratively update by taking weighted averages of neighbours' opinions. Similarly, \cite{jadbabaie2013information} proposes a linear updating rule combining private information with neighbours' beliefs. Further extensions include
\begin{enumerate*}
\item \cite{candogan2020social}, which incorporates the influence of the information platform through which agents receive private signals,
\item \cite{hunter2022optimizing}, which introduces stubborn agents whose beliefs never change, and
\item \cite{chitiva2024continuous}, which studies competing lobbies that strategically bias agents' opinions.
\end{enumerate*}
Additional contributions include \cite{parasnis2020non, feldman2014reaching, molavi2018theory, liu2011non}. On the Bayesian side, \begin{enumerate*}
\item \cite{mossel2016efficient} considers agents who iteratively combine noisy private signals with neighbours' estimates via Bayes' rule,
\item \cite{bowen2023learning} considers agents who strategically share only selected portions of their private information, resulting in heterogeneous information environments, and 
\item \cite{williams2024echo} studies learning under preference uncertainty, where agents discount divergent opinions, giving rise to echo chamber dynamics. 
\end{enumerate*}
Other Bayesian models include \cite{gale2003bayesian, perreault2012bayesian, ifrach2019bayesian, kanoria2011majority}, and we refer the reader to \cite{acemoglu2011bayesian} for a comprehensive review. Since agents in our model enter the network sequentially and make their decisions based only on information available at the time of arrival, our framework also belongs to the class of sequential social learning models (see, for example, \cite{bahar2020multi, guo2026robust, arieli2020social, lu2024enabling, uradnik2025maximizing}). Notably, similar modeling frameworks can be adapted to a variety of settings in which decisions are shaped by social interactions on a network, including employment outcomes (\citep{calvo2004effects,ioannides2004job}), technology adoption (\citep{munshi2004social,conley2001social}), and many forms of social contagion.

\section{Description of the model}\label{sec:model}
\subsection{When each sample drawn is of the same size, $k$}\label{subsec:model_fixed_sample_size} From here onward, we let $\mathbb{N}$ indicate the set of all positive integers, $\mathbb{N}_{0}$ indicate the set of all non-negative integers, and we set $[m]=\{1,2,\ldots,m\}$ for each $m\in\mathbb{N}$. We begin by fixing $N\in\mathbb{N}$, and letting $G_{N}=\{v_{1},v_{2},\ldots,v_{N}\}$ be an independent set (in other words, a graph with no edges), with each vertex $v_{i}$ endowed with an \emph{opinion} or \emph{state} $X_{i}$. For $n>N$, the directed graph we obtain at the end of epoch $n$, on vertices $v_{1},\ldots,v_{n}$, is referred to as $G_{n}$. The state of $v_{i}$ is denoted by $X_{i}$ for all $i\in[n]$, with each $X_{i}\in\{\pm 1\}$. The state of a vertex is decided right after it is introduced, and it does not change thereafter. We let $\indeg_{n}(v_{i})$ denote the in-degree of the vertex $v_{i}$ in $G_{n}$, i.e.\
\begin{equation}
\indeg_{n}(v_{i})=\sum_{j\in[n]}\chi\left\{(v_{j},v_{i})\in E(G_{n})\right\},\label{indeg_defn}
\end{equation}
where $E(G_{n})$ is the set of all directed edges of $G_{n}$, and $(v_{j},v_{i})$ indicates the edge directed \emph{from} $v_{j}$ \emph{towards} $v_{i}$. The notation $\chi(A)$, for any event $A$, denotes the indicator for the event $A$. As we shall see (from the description that follows), it suffices to consider, in \eqref{indeg_defn}, the sum over all indices $j\in\{i+1,\ldots,n\}$.

At the beginning of epoch $(n+1)$, for $n\geqslant N$, a new vertex, $v_{n+1}$, is introduced into the graph. A random sample of size $k$, for a prespecified $k\in\mathbb{N}$, is drawn, with replacement, from the existing vertices, as follows: if $U_{n,1},\ldots,U_{n,k}$ indicate the indices of the vertices chosen in the sample, then $U_{n,1},\ldots,U_{n,k}$ are i.i.d.\ with
\begin{align}
\Prob\left[U_{n,1}=m\big|\mathcal{F}_{n}\right]=\frac{\indeg_{n}(v_{m})+\beta}{\sum_{j\in[n]}\indeg_{n}(v_{j})+n\beta} \text{ for all }m\in[n],\label{sampling_probability_proportional_to_indegree}
\end{align}
where $\beta>0$ is a prespecified constant, and $\mathcal{F}_{n}$ is the $\sigma$-field consisting of all information pertaining to the process up to and including epoch $n$. Let $P_{n}$ indicate the number of sampled customers, counting repetitions, in state $+1$, i.e.\
\begin{equation}
P_{n}=\sum_{i\in[k]}\chi\left\{X_{U_{n,i}}=+1\right\}.\label{P_{n}_defn}
\end{equation}
The state $X_{n+1}$ of $v_{n+1}$ is now decided as follows:
\begin{equation}\label{X_{n+1}_conditional_distribution}
X_{n+1}=
\begin{cases}
+1 &\text{with probability }pf(P_{n}/k)+(1-p)\{1-f(P_{n}/k)\},\\
-1 &\text{with probability }(1-p)f(P_{n}/k)+p\{1-f(P_{n}/k)\},
\end{cases}
\end{equation}
where $p\in[0,1]\setminus\{1/2\}$ is the \emph{memory parameter} or \emph{self-excitation parameter}, and $f:[0,1]\rightarrow[0,1]$ is the \emph{reinforcement function}. Finally, having decided $X_{n+1}$, the vertex $v_{n+1}$ joins every vertex $v_{i}$ that it sampled and agreed with, via as many outgoing edges as the number of times $v_{i}$ appeared in the sample. In other words, we draw the directed edge $(v_{n+1},v_{U_{n,i}})$ for each $i\in[k]$ such that $X_{n+1}=X_{U_{n,i}}$.
\begin{figure}[h!]
  \centering
    \includegraphics[width=0.5\textwidth]{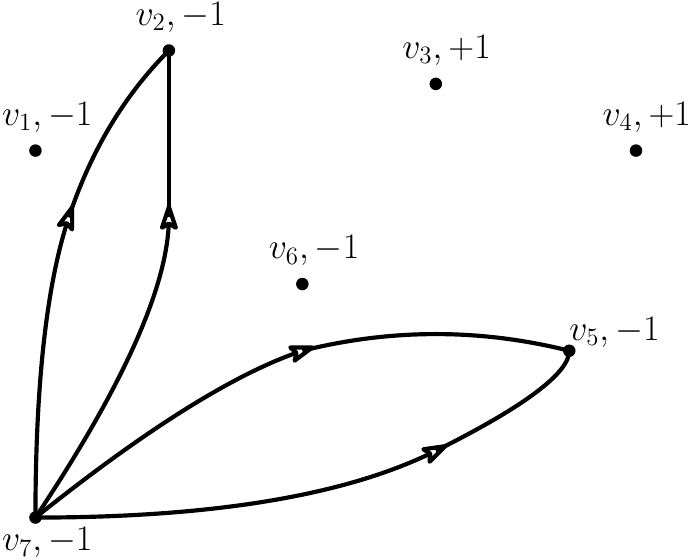}
\caption{In this example, the incoming vertex is $v_{n+1}=v_{7}$, and sample size is $k=6$, with sampled indices $(U_{n,1},\ldots,U_{n,6})=(2,4,2,5,5,4)$. Since $X_{7}=X_{2}=X_{5}=-1$ while $X_{4}=+1$, we draw two directed edges from $v_{7}$ to $v_{2}$, and two directed edges from $v_{7}$ to $v_{5}$.}
  \label{galaxy_graph}
\end{figure}
The primary object of interest in this model is the proportion, in $G_{n}$, of the vertices with opinion $+1$, especially as $n$ grows large. 

\subsection{When the sample size $k_{n}$ grows with $n$}\label{subsec:model_growing_sample_size} In the model described in \S\ref{subsec:model_fixed_sample_size}, $k$ remains fixed as $n$ grows. In this paper, we also consider a generalization of the model described above, wherein, for $n\geqslant N$, the size of the sample drawn during epoch $(n+1)$ (i.e.\ upon the introduction of the vertex $v_{n+1}$) is $k_{n}$, with $\{k_{n}:n\geqslant N\}$ a prespecified sequence of positive integers growing with $n$. The sampling scheme is, as before, with replacement, and if $U_{n,1},\ldots,U_{n,k_{n}}$ are the indices of the sample drawn during epoch $(n+1)$, then these are i.i.d.\ with
\begin{align}
\Prob\left[U_{n,1}=m\big|\mathcal{F}_{n}\right]=\frac{\indeg_{n}(v_{m})+\beta_{n}}{\sum_{j\in[n]}\indeg_{n}(v_{j})+n\beta_{n}} \text{ for all }m\in[n],\label{sampling_probability_proportional_to_indegree_growing_sample_size}
\end{align}
where $\{\beta_{n}:n\geqslant N\}$ is a given sequence of positive reals. We assume the existence of some constant $\alpha>0$ that has no dependence on $n$, such that 
\begin{equation}\label{beta_{n}_s_{n}_defns}
n \beta_{n}=\alpha s_{n}, \text{ where } s_{n}=\sum_{i=N+1}^{n}k_{i-1} \text{ for each }n\geqslant N+1.
\end{equation}
In the probabilities stated in \eqref{X_{n+1}_conditional_distribution}, we now replace $k$ by $k_{n}$. To obtain results for this generalized model, we impose suitable growth conditions on $\{k_{n}\}$ as $n\rightarrow\infty$.

\section{Main results}\label{sec:main_results}
Let us define, up to and including epoch $n$, the following quantities for each of the models described in \S\ref{subsec:model_fixed_sample_size} and \S\ref{subsec:model_growing_sample_size} (in the definitions that follow, $B_{n}=C_{n}=0$ for all $n\leqslant N$):
\begin{enumerate}
\item $A_{n}$, which is the number of vertices in $G_{n}$ with opinion $+1$, i.e.\ 
\begin{equation}\label{A_{n}_defn}
A_{n}=\sum_{i\in[n]}\chi\left\{X_{i}=+1\right\},
\end{equation}
\item $B_{n}$, which is the sum of the in-degrees of all vertices in $G_{n}$ with opinion $+1$, i.e.\
\begin{equation}\label{B_{n}_defn}
B_{n}=\sum_{i\in[n]}\indeg_{n}(v_{i})\chi\left\{X_{i}=+1\right\},
\end{equation}
\item and $C_{n}$, which is the sum of the in-degrees of all vertices in $G_{n}$, i.e.\ $C_{n}=\sum_{i\in[n]}\indeg_{n}(v_{i})$.
\end{enumerate}
We define the function $g:[0,1]\rightarrow[0,1]$ as
\begin{equation}\label{g_defn}
g(x)=pf(x)+(1-p)\{1-f(x)\} \text{ for all }x\in[0,1],
\end{equation}
where $f$, recall, is the reinforcement function appearing in \eqref{X_{n+1}_conditional_distribution}.

We now state the main results of this paper. The results pertaining to the model described in \S\ref{subsec:model_fixed_sample_size} have been stated in \S\ref{subsec:main_results_fixed_sample_size}, whereas the results pertaining to the model described in \S\ref{subsec:model_growing_sample_size} have been stated in \S\ref{subsec:main_results_growing_sample_size}. 

\subsection{Main results pertaining to the model described in \S\ref{subsec:model_fixed_sample_size}}\label{subsec:main_results_fixed_sample_size}
For the model described in \S\ref{subsec:model_fixed_sample_size}, we define the compact set
\begin{equation}\label{domain_defn}
\mathcal{S}=\left\{(a,b,c)\in[0,1]\times[0,k]\times[0,k]:c\geqslant b\right\},
\end{equation}
and the function $q_{\beta}:\mathcal{S}\rightarrow[0,1]$, for any constant $\beta>0$, as
\begin{equation}\label{q_defn}
q_{\beta}(a,b,c)=\frac{b+\beta a}{c+\beta} \quad\text{for each } (a,b,c)\in\mathcal{S}.
\end{equation}
For each $(a,b,c)\in\mathcal{S}$, we set $H(a,b,c)=(H_{1}(a,b,c),H_{2}(a,b,c),H_{3}(a,b,c))$, where
\begin{align}
H_{1}(a,b,c)={}&\sum_{i=0}^{k}g\left(\frac{i}{k}\right){k\choose i}q_{\beta}^{i}(a,b,c)\left\{1-q_{\beta}(a,b,c)\right\}^{k-i},\label{H_{1}_defn}\\
H_{2}(a,b,c)={}&\sum_{i\in[k]}i g\left(\frac{i}{k}\right){k\choose i}q_{\beta}^{i}(a,b,c)\left\{1-q_{\beta}(a,b,c)\right\}^{k-i},\label{H_{2}_defn}\\
H_{3}(a,b,c)={}&\sum_{i=0}^{k}\left[i\left\{2g\left(\frac{i}{k}\right)-1\right\}+k\left\{1-g\left(\frac{i}{k}\right)\right\}\right]{k\choose i}q_{\beta}^{i}(a,b,c)\left\{1-q_{\beta}(a,b,c)\right\}^{k-i}.\label{H_{3}_defn}
\end{align}
Right away, we see that, if $Y$ is a random variable following Binomial$(k,q_{\beta}(a,b,c))$, we have 
\begin{equation}\label{H_{1},H_{2},H_{3}_concise_expressions}
\begin{cases}
&H_{1}(a,b,c)=\E[g(Y/k)], \quad H_{2}(a,b,c)=\E[Yg(Y/k)],\\
&H_{3}(a,b,c)=\E[Y\{2g(Y/k)-1\}+k\{1-g(Y/k)\}]=\E[Yg(Y/k)+(k-Y)\{1-g(Y/k)\}].
\end{cases}
\end{equation}
Recall, as stated right before \eqref{g_defn}, that $g$ takes values in $[0,1]$, ensuring, via \eqref{H_{1},H_{2},H_{3}_concise_expressions}, that $H_{1}(a,b,c)\in[0,1]$. This also ensures, via \eqref{H_{1},H_{2},H_{3}_concise_expressions} and using the fact that $Y\leqslant k$, that $H_{2}(a,b,c)\in[0,k]$ and that $k\geqslant H_{3}(a,b,c)\geqslant H_{2}(a,b,c)$. Thus, we conclude that the function $H=(H_{1},H_{2},H_{3})$ maps the set $\mathcal{S}$ to itself.

We also define the function $h$, on $\mathcal{S}$, as follows: for each $(a,b,c)\in\mathcal{S}$, we set
\begin{equation}\label{h_defn_fixed_size}
h\left(\begin{bmatrix}
a\\
b\\
c
\end{bmatrix}\right)=\begin{bmatrix}
h_{1}(a,b,c)\\
h_{2}(a,b,c)\\
h_{3}(a,b,c)
\end{bmatrix}=H\left(\begin{bmatrix}
a\\
b\\
c
\end{bmatrix}\right)-\begin{bmatrix}
a\\
b\\
c
\end{bmatrix}=\begin{bmatrix}
H_{1}(a,b,c)\\
H_{2}(a,b,c)\\
H_{3}(a,b,c)
\end{bmatrix}-\begin{bmatrix}
a\\
b\\
c
\end{bmatrix}.
\end{equation}

Here and henceforth, the notation $\dot{x}(t)$ indicates the derivative of $x(t)$ with respect to time $t$.
\begin{theorem}\label{thm:main_1}
For the model described in \S\ref{subsec:model_fixed_sample_size}, the sequence $\{(n^{-1}A_{n},n^{-1}B_{n},n^{-1}C_{n})\}$ converges almost surely to a compact, connected, internally chain transitive invariant set, contained in $\mathcal{S}$, corresponding to the autonomous ODE
\begin{equation}\label{ODE_fixed_sample_size}
\left(\dot{a}(t),\dot{b}(t),\dot{c}(t)\right)=h\left(a(t),b(t),c(t)\right) \text{ for }t\geqslant 0, \text{ with } (a(0),b(0),c(0))=(A_{N}/N,0,0),
\end{equation}
where the function $h$ is as defined in \eqref{h_defn_fixed_size}, and $A_{N}$ is as defined via \eqref{A_{n}_defn}. 
\end{theorem}
In order to state the next theorem, we introduce a few definitions. For a random variable $Y$ following Binomial$(k,q)$, for any $q\in[0,1]$, let us define the quantities 
\begin{equation}\label{F_{1},F_{2}_defns}
F_{1}(q)=\E[g(Y/k)] \text{ and } F_{2}(q)=\E[Yg(Y/k)], \text{ with }g\text{ as in \eqref{g_defn}.} 
\end{equation}
For each $q\in[0,1]$, let us define
\begin{equation}\label{G_defn}
G(q)=(1-2q)F_{2}(q)+(\beta+kq)F_{1}(q)-kq(1-q)-\beta q.
\end{equation}
Right away, we see that $F_{1}(0)=g(0)$ and $F_{2}(0)=0$, whereas $F_{1}(1)=g(1)$ and $F_{2}(1)=kg(1)$, so that $G(0)=\beta g(0)$ and $G(1)=-\beta\{1-g(1)\}$. Thus, if we assume that $g(0)>0$ and $g(1)<1$, then $G(0)>0$ and $G(1)<0$, so that the function $G$ must have at least one root in the interval $(0,1)$.  
\begin{theorem}\label{thm:main_2}
Assume $g(0)>0$ and $g(1)<1$. Assume, also, that the function $G$, defined in \eqref{G_defn}, has a unique root, denoted $q^{*}$, in $(0,1)$, and define the set $\Lambda=\{(a,b,c)\in\mathcal{S}:q_{\beta}(a,b,c)=q^{*}\}$, where $\mathcal{S}$ is as in \eqref{domain_defn}, and the function $q_{\beta}$ is as defined in \eqref{q_defn}. Then, for the model described in \S\ref{subsec:model_fixed_sample_size}, the sequence $\{(n^{-1}A_{n},n^{-1}B_{n},n^{-1}C_{n})\}$ converges almost surely to a compact, connected, internally chain transitive invariant set, contained in the set $\Lambda$, corresponding to the autonomous ODE given by \eqref{ODE_fixed_sample_size}.
\end{theorem}

The next result is an illustration of various naturally arising examples of the reinforcement function $f$ for which the function $G$ has a unique root in $[0,1]$:
\begin{prop}\label{prop:examples_g_unique_root_of_G}
Recall the function $f$ defined via \eqref{X_{n+1}_conditional_distribution}, and the function $g$ as defined in \eqref{g_defn}. Assume, throughout, that $p\in(0,1)\setminus\{1/2\}$. Then, in each of the following cases, the function $G$, defined in \eqref{G_defn}, has a unique root in $[0,1]$ (lying strictly between $0$ and $1$):
\begin{enumerate}
\item when $f(x)=c$ for all $x\in[0,1]$, for some $c\in(0,1)$;
\item when $f(x)=x$ for all $x\in[0,1]$, and $p\in(0,1/2)$;
\item when $f(x)=x$ for all $x\in[0,1]$, $p\in(1/2,1)$ and either
\begin{equation}\label{beta_p_cond}
\beta\geqslant \frac{(2p-1)(k-1)}{4(1-p)},
\end{equation}
or \eqref{beta_p_cond} does not hold and one of the following two inequalities is true:
\begin{align}
{}&G\left(\frac{1}{2}-\frac{\sqrt{3(k-1)(2p-1)\{(k-1)(2p-1)-4\beta(1-p)\}}}{6(k-1)(2p-1)}\right)>0,\label{beta_{1}_cond}\\
{}&G\left(\frac{1}{2}+\frac{\sqrt{3(k-1)(2p-1)\{(k-1)(2p-1)-4\beta(1-p)\}}}{6(k-1)(2p-1)}\right)<0;\label{beta_{2}_cond}
\end{align}
\item \label{exponential} when $f(x)=e^{x-1}$ for all $x\in[0,1]$, $k\geqslant 2$, $p\in(0,1/2)$ and $\beta>k(1-2p)$. 
\end{enumerate}
\end{prop}
When $f(x)=e^{x-1}$ for $x\in[0,1]$, in addition to \eqref{exponential} which provides sufficient conditions for $G$ to have a unique root for $p\in(0,1/2)$, we include, in Figure~\ref{fig}, plots that demonstrate that the same conclusion is true for plenty of values of $k$, $\beta$ and $p\in(1/2,1)$. 

Our next result concerns itself with how the stochastic process that is obtained via linear interpolation from the iterates $(n^{-1}A_{n},n^{-1}B_{n},n^{-1}C_{n})$ (arising out of the model described in \S\ref{subsec:model_fixed_sample_size}) deviates from the trajectory driven by the autonomous ODE mentioned in the statement of Theorem~\ref{thm:main_1}:
\begin{theorem}\label{thm:main_6}
Consider the model described in \S\ref{subsec:model_fixed_sample_size}, and let us define $Z_{n}=(n^{-1}A_{n},n^{-1}B_{n},n^{-1}C_{n})$ for all $n\in\mathbb{N}$ with $n\geqslant N$. Fix \emph{any} $T>0$. Writing $t_{N}=0$ and $t_{n}=\sum_{N+1\leqslant i\leqslant n}i^{-1}$ for $n\geqslant N+1$, we set $m_{n}=\min\{m\geqslant n: t_{m}\geqslant t_{n}+T\}$ for each $n\in\mathbb{N}$. Recall the function $h$ as defined in \eqref{h_defn_fixed_size}. For each fixed $n\in\mathbb{N}$ with $n\geqslant N+1$, let $\{Z^{n}(t): t\geqslant t_{n}\}$, with $Z^{n}(t)=(a^{n}(t),b^{n}(t),c^{n}(t))$ for each $t\geqslant t_{n}$, indicate the solution to
\begin{equation}
\left(\dot{a}(t),\dot{b}(t),\dot{c}(t)\right)=h(a(t),b(t),c(t)) \text{ for all }t\geqslant t_{n}, \text{ with } \left(a(t_{n}),b(t_{n}),c(t_{n})\right)=\left(n^{-1}A_{n},n^{-1}B_{n},n^{-1}C_{n}\right),\nonumber
\end{equation}
and let 
\begin{equation}
W_{j}=(j+1)^{-1/2}\left\{\left(j^{-1}A_{j},j^{-1}B_{j},j^{-1}C_{j}\right)-\left(a^{n}(t_{j}),b^{n}(t_{j}),c^{n}(t_{j})\right)\right\} \text{ for each }j\geqslant n.\nonumber
\end{equation}
We now define the process $\{W^{n}(t):t\in [t_{n},t_{n}+T]\}$ by linear interpolation, as follows: 
\begin{equation}
W^{n}(t)=W_{j}+\frac{t-t_{j}}{t_{j+1}-t_{j}}(W_{j+1}-W_{j}) \text{ for }t\in[t_{j},t_{j+1}], \text{ for each }j\in\{n,n+1,\ldots,m_{n}-1\}.\nonumber  
\end{equation}  
Finally, we set $\tilde{W}^{n}(t)=W^{n}(t_{n}+t)$ for all $t\in[0,T]$. Likewise, we define $\tilde{Z}^{n}(t)=Z^{n}(t_{n}+t)$ for all $t\in[0,T]$. Then the laws of the processes $\{(\tilde{Z}^{n}(\cdot),\tilde{W}^{n}(\cdot))\}_{n\in\mathbb{N}}$ are relatively compact in the space $\mathcal{P}(C([0,T];\mathbb{R}^{3}))^{2}$, where $C([0,T];\mathbb{R}^{3})$ is the set of all continuous functions from $[0,T]$ to $\mathbb{R}^{3}$ and $\mathcal{P}(C([0,T];\mathbb{R}^{3}))^{2}$ is the space of all probability measures on $C([0,T];\mathbb{R}^{3})\times C([0,T];\mathbb{R}^{3})$.

If a subsequence of $\{(\tilde{Z}^{n}(\cdot),\tilde{W}^{n}(\cdot))\}_{n\in\mathbb{N}}$ converges in law to a limiting process $\{z^{*}(\cdot),w^{*}(\cdot)\}$, then $z^{*}(\cdot)$ belongs to a compact, connected, internally chain transitive invariant set corresponding to the ODE described by \eqref{ODE_fixed_sample_size}, and, letting $J_{h}$ denote the Jacobian of $h$, each of
\begin{multline}
w^{*}(t)-\int_{0}^{t}\left\{J_{h}\left(z^{*}(s)\right)+\frac{1}{2}I_{3}\right\}w^{*}(s)ds \quad \text{and}\quad \left[w^{*}(t)-\int_{0}^{t}\left\{J_{h}\left(z^{*}(s)\right)+\frac{1}{2}I_{3}\right\}w^{*}(s)ds\right]\\\left[w^{*}(t)-\int_{0}^{t}\left\{J_{h}\left(z^{*}(s)\right)+\frac{1}{2}I_{3}\right\}w^{*}(s)ds\right]^{T}-\int_{0}^{t}Q\left(z^{*}(s)\right)ds,\nonumber
\end{multline}
is a martingale for all $t\in[0,T]$, where $Q:\mathcal{S}\rightarrow\mathbb{R}^{3\times 3}$ is a matrix-valued function with entries given by
\begin{equation}\label{Q_defn}
\begin{cases}
&Q_{1,1}(a,b,c)=H_{1}(a,b,c)\{1-H_{1}(a,b,c)\},\\
&Q_{1,2}(a,b,c)=Q_{2,1}(a,b,c)=H_{2}(a,b,c)\{1-H_{1}(a,b,c)\},\\
&Q_{2,2}(a,b,c)=H_{4}(a,b,c)-H_{2}^{2}(a,b,c),\\
&Q_{1,3}(a,b,c)=Q_{3,1}(a,b,c)=H_{2}(a,b,c)-H_{1}(a,b,c)H_{3}(a,b,c),\\
&Q_{2,3}(a,b,c)=Q_{3,2}(a,b,c)=H_{4}(a,b,c)-H_{2}(a,b,c)H_{3}(a,b,c),\\
&Q_{3,3}(a,b,c)=kq_{\beta}(a,b,c)\{1-q_{\beta}(a,b,c)\}+k^{2}\{1-q_{\beta}(a,b,c)\}^{2}-k^{2}H_{1}(a,b,c)\\&+2k H_{2}(a,b,c)-H_{3}^{2}(a,b,c),
\end{cases}
\end{equation}
with $H_{1}$, $H_{2}$ and $H_{3}$ as defined in \eqref{H_{1}_defn}, \eqref{H_{2}_defn} and \eqref{H_{3}_defn}, and $H_{4}:\mathcal{S}\rightarrow[0,k^{2}]$ defined as
\begin{equation}
H_{4}(a,b,c)=\sum_{i\in[k]}i^{2}g\left(\frac{i}{k}\right){k\choose i}q_{\beta}(a,b,c)^{i}\{1-q_{\beta}(a,b,c)\}^{k-i} \text{ for } (a,b,c)\in\mathcal{S},\label{H_{4}_defn}
\end{equation}
where $\mathcal{S}$ is as defined in \eqref{domain_defn} and $q_{\beta}(\cdot,\cdot,\cdot)$ is as defined in \eqref{q_defn}.
\end{theorem}

\subsection{Main results pertaining to the model described in \S\ref{subsec:model_growing_sample_size}}\label{subsec:main_results_growing_sample_size} When it comes to the model described in \S\ref{subsec:model_growing_sample_size}, we define the compact set
\begin{equation}
\hat{\mathcal{S}}=\left\{(a,b,c)\in[0,1]^{3}:c\geqslant b\right\},\label{domain_defn_growing_size}
\end{equation}
and we define the function $\hat{H}=(\hat{H}_{1},\hat{H}_{2},\hat{H}_{3}):\hat{\mathcal{S}}\rightarrow\hat{\mathcal{S}}$ as follows, with $q_{\alpha}$ as defined in \eqref{q_defn} (but this time, we need only be concerned with $q_{\alpha}$ defined on $\hat{\mathcal{S}}$, which is a subset of $\mathcal{S}$ defined in \eqref{domain_defn}):
\begin{align}
\hat{H}_{1}(a,b,c)={}&g\left(q_{\alpha}(a,b,c)\right),\label{hat{H}_{1}_defn}\\
\hat{H}_{2}(a,b,c)={}&q_{\alpha}(a,b,c)g\left(q_{\alpha}(a,b,c)\right),\label{hat{H}_{2}_defn}\\
\hat{H}_{3}(a,b,c)={}&q_{\alpha}(a,b,c)g\left(q_{\alpha}(a,b,c)\right)+\left\{1-q_{\alpha}(a,b,c)\right\}\left\{1-g\left(q_{\alpha}(a,b,c)\right)\right\}.\label{hat{H}_{3}_defn}
\end{align}
That $\hat{H}$ indeed maps $\hat{\mathcal{S}}$ to itself follows from the fact that the function $g$, defined in \eqref{g_defn}, takes values in $[0,1]$, and the fact that $q_{\alpha}(a,b,c)\in[0,1]$ for each $(a,b,c)\in\hat{\mathcal{S}}$. We also define
\begin{equation}\label{hat{h}_defn_growing_size}
\hat{h}\left(\begin{bmatrix}
a\\
b\\
c
\end{bmatrix}\right)=\begin{bmatrix}
\hat{h}_{1}(a,b,c)\\
\hat{h}_{2}(a,b,c)\\
\hat{h}_{3}(a,b,c)
\end{bmatrix}=\begin{bmatrix}
\hat{H}_{1}(a,b,c)\\
\hat{H}_{2}(a,b,c)\\
\hat{H}_{3}(a,b,c)
\end{bmatrix}-\begin{bmatrix}
a\\
b\\
c
\end{bmatrix} \quad \text{for each }(a,b,c)\in\hat{\mathcal{S}}.
\end{equation}

Before stating Theorem~\ref{thm:main_3}, we recall that, given any open subset $S$ of $\mathbb{R}^{d}$ for any $d\in\mathbb{N}$, and any $i\in\mathbb{N}$, we let $\mathcal{C}^{(i)}(S)$ indicate the set of all functions $f:S\rightarrow\mathbb{R}$ such that the $i$-th order partial derivatives of $f$ (which are of the form $\partial^{j_{1}}/\partial x_{1}^{j_{1}}\cdots\partial^{j_{d}}/\partial x_{d}^{j_{d}}f(x_{1},\ldots,x_{d})$ for $j_{1},\ldots,j_{d}\in\mathbb{N}_{0}$ with $j_{1}+\cdots+j_{d}=i$) exist and are continuous throughout $S$. We let $\mathcal{C}^{(\infty)}(S)$ indicate the set of all $f:S\rightarrow\mathbb{R}$ such that the $i$-th order partial derivatives of $f$ exist throughout $S$ for each $i\in\mathbb{N}$ (evidently, for such an $f$, partial derivatives of every order must be continuous throughout $S$, as well). If $S'$ is a compact subset of $\mathbb{R}^{d}$, and $f:S'\rightarrow\mathbb{R}$, then by $f\in\mathcal{C}^{(i)}(S')$ we mean that there exists some open subset $S$ of $\mathbb{R}^{d}$ with $S'\subset S$, and some function $\tilde{f}:S\rightarrow\mathbb{R}$, such that $\tilde{f}|_{S'}\equiv f$ (i.e.\ $\tilde{f}$, when restricted to $S'$, is identical to $f$) and $\tilde{f}\in\mathcal{C}^{(i)}(S)$. 

Given any $S\subset\mathbb{R}^{d}$, we let $\mathcal{C}^{(0)}(S)$ indicate the set of all $f:S\rightarrow\mathbb{R}$ such that $f$ is continuous throughout $S$. In particular, when $S$ is a compact subset of $\mathbb{R}^{d}$, for $f\in\mathcal{C}^{(0)}(S)$ and $\delta>0$, we define $\omega_{f}(\delta)=\sup\{|f(\mathbf{x})-f(\mathbf{y})|:\mathbf{x},\mathbf{y}\in S,||\mathbf{x}-\mathbf{y}||<\delta\}$. Evidently, $\omega_{f}(\delta)$ approaches $0$ as $\delta\rightarrow0$, due to the continuity of $f$ on $S$.
\begin{theorem}\label{thm:main_3}
Recall the model described in \S\ref{subsec:model_growing_sample_size}, and assume that 
\begin{equation}\label{k_{n}_sequence_conditions}
\sum_{n\geqslant N}\frac{k_{n}}{s_{n+1}}\text{ diverges}, \sum_{n\geqslant N}\frac{k_{n}^{2}}{s_{n+1}^{2}}\text{ converges, and }0<\eta_{1}=\inf_{n\geqslant N}\frac{s_{n+1}}{(n+1)k_{n}}\leqslant\sup_{n\geqslant N}\frac{s_{n+1}}{(n+1)k_{n}}=\eta_{2}<\infty,
\end{equation}
where $s_{n}$ is as defined in \eqref{beta_{n}_s_{n}_defns}. Assume that one of the following is true:
\begin{enumerate}[label=(A\arabic*), ref=A\arabic*]
\item\label{A1} When $g$ is Lipschitz on $[0,1]$, each of the series $\sum_{n\geqslant N}(n+1)^{-1}k_{n}^{-1/2}$ and $\sum_{n\geqslant N}s_{n+1}^{-1}k_{n}^{1/2}$ converges.
\item\label{A2} When $g\in\mathcal{C}^{(1)}[0,1]$, each of the series 
\begin{equation}
\sum_{n\geqslant N}(n+1)^{-1}k_{n}^{-1/2}\omega_{g'}(k_{n}^{-1/2}),\ \sum_{n\geqslant N}s_{n+1}^{-1}k_{n}^{1/2}\omega_{\varphi'}(k_{n}^{-1/2}) \text{ and }\sum_{n\geqslant N}s_{n+1}^{-1}k_{n}^{1/2}\omega_{\psi'}(k_{n}^{-1/2})\label{C^{1}_convergence_criteria}
\end{equation}
converges, where $\varphi(x)=xg(x)$ and $\psi(x)=xg(x)+(1-x)\{1-g(x)\}$ for all $x\in[0,1]$. If each of $\sum_{n\geqslant N}s_{n+1}^{-1}$ and $\sum_{n\geqslant N}s_{n+1}^{-1}k_{n}^{1/2}\omega_{g'}(k_{n}^{-1/2})$ converges, then so do the second and third series of \eqref{C^{1}_convergence_criteria}.
\item\label{A3} When $g\in\mathcal{C}^{(2)}[0,1]$, each of the series $\sum_{n\geqslant N}(n+1)^{-1}k_{n}^{-1}$ and $\sum_{n\geqslant N}s_{n+1}^{-1}$ converges.
\end{enumerate}
Then the sequence $\{(n^{-1}A_{n},s_{n}^{-1}B_{n},s_{n}^{-1}C_{n})\}$ converges almost surely to a (possibly path dependent) compact, connected, internally chain transitive invariant set, contained in $\hat{\mathcal{S}}$, corresponding to the differential inclusion $(\dot{a}(t),\dot{b}(t),\dot{c}(t))\in F(a(t),b(t),c(t))$ for $t\geqslant 0$, where we define, using \eqref{hat{h}_defn_growing_size}: 
\begin{equation}\label{set_valued_function}
F\left(\begin{bmatrix}
a\\
b\\
c
\end{bmatrix}\right)=\left\{\begin{bmatrix}
\eta\hat{h}_{1}(a,b,c)\\
\hat{h}_{2}(a,b,c)\\
\hat{h}_{3}(a,b,c)
\end{bmatrix}:\eta\in[\eta_{1},\eta_{2}]\right\}, \text{ for each }\begin{bmatrix}
a\\
b\\
c
\end{bmatrix}\in\hat{\mathcal{S}}.
\end{equation}
\end{theorem}

Before stating Theorem~\ref{thm:main_5}, we remark that if $g\in\mathcal{C}^{(1)}[0,1]$, then $\sup\{|g'(q)|:q\in[0,1]\}$ is finite and attained in $[0,1]$ since $[0,1]$ is compact.
\begin{theorem}\label{thm:main_5}
\sloppy Assume that the function $g$, defined in \eqref{g_defn}, is in $\mathcal{C}^{(1)}[0,1]$ and monotonically increasing throughout $[0,1]$, with $g(0)>0$ and $g(1)<1$, and that the function $\hat{G}:[0,1]\rightarrow\mathbb{R}$, defined as
\begin{align}\label{hat{G}_defn}
\hat{G}(q)=q(1-q)\{2g(q)-1\}+\alpha\{g(q)-q\} \text{ for all }q\in[0,1],
\end{align}
with $\alpha$ as in \eqref{beta_{n}_s_{n}_defns}, has a unique root, say $q^{*}$, in $[0,1]$. In addition to the constraints described in \eqref{k_{n}_sequence_conditions} on $\eta_{1}$ and $\eta_{2}$, we assume that $\eta_{2}\leqslant 1$, and letting $\upsilon=\sup\{g'(q):q\in[0,1]\}$, we assume that $0<\upsilon\leqslant\min\{2,(1+\sqrt{5})(1-\eta_{1})^{-1}\eta_{1}\}$. Then, for the model in \S\ref{subsec:model_growing_sample_size}, the sequence $\{(n^{-1}A_{n},s_{n}^{-1}B_{n},s_{n}^{-1}C_{n})\}$ converges almost surely to a compact, connected, internally chain transitive invariant set, corresponding to the differential inclusion $(\dot{a}(t),\dot{b}(t),\dot{c}(t))\in F(a(t),b(t),c(t))$ for $t\geqslant 0$ (with $F$ as defined in \eqref{set_valued_function}), contained in the set
\begin{equation}
\hat{\Lambda}=\left\{(a,b,c)\in\hat{\mathcal{S}}:q_{\alpha}(a,b,c)=q^{*},a=g(q^{*})\right\}.\label{hat{Lambda}_defn}
\end{equation}
\end{theorem}
Some discussion is now in order pertaining to the equilibrium points corresponding to the differential inclusion mentioned in the statement of Theorem~\ref{thm:main_5}. Let $(a^{*},b^{*},c^{*})$ be such an equilibrium point, so that $(0,0,0)\in F(a^{*},b^{*},c^{*})$. Since $\eta_{1}>0$, this becomes equivalent to having
\begin{equation}\label{equilibrium_equations}
\begin{cases}
&a^{*}=g\left(q_{\alpha}(a^{*},b^{*},c^{*})\right),\quad b^{*}=q_{\alpha}(a^{*},b^{*},c^{*})g\left(q_{\alpha}(a^{*},b^{*},c^{*})\right),\\
&c^{*}=q_{\alpha}(a^{*},b^{*},c^{*})g\left(q_{\alpha}(a^{*},b^{*},c^{*})\right)+\left\{1-q_{\alpha}(a^{*},b^{*},c^{*})\right\}\left\{1-g\left(q_{\alpha}(a^{*},b^{*},c^{*})\right)\right\}.
\end{cases}
\end{equation}
Substituting these expressions in the definition of $q_{\alpha}$ given by \eqref{q_defn}, we obtain:
\begin{align}
{}&q_{\alpha}(a^{*},b^{*},c^{*})=\frac{\left\{q_{\alpha}(a^{*},b^{*},c^{*})+\alpha\right\}g\left(q_{\alpha}(a^{*},b^{*},c^{*})\right)}{q_{\alpha}(a^{*},b^{*},c^{*})g\left(q_{\alpha}(a^{*},b^{*},c^{*})\right)+\left\{1-q_{\alpha}(a^{*},b^{*},c^{*})\right\}\left\{1-g\left(q_{\alpha}(a^{*},b^{*},c^{*})\right)\right\}+\alpha}\nonumber\\
{}&\Longleftrightarrow\hat{G}\left(q_{\alpha}(a^{*},b^{*},c^{*})\right)=0 \Longleftrightarrow q_{\alpha}(a^{*},b^{*},c^{*})=q^{*},\label{q(equilibrium)=q^{*}}
\end{align}
where $\hat{G}$ is as defined in \eqref{hat{G}_defn}, and the final implication is true since $\hat{G}$ has a unique root, $q^{*}$, in $[0,1]$. Substituting the final conclusion from \eqref{q(equilibrium)=q^{*}} into \eqref{equilibrium_equations}, we conclude that 
\begin{equation}
a^{*}=g(q^{*}),\quad b^{*}=q^{*}g(q^{*}) \quad \text{and} \quad c^{*}=q^{*}g(q^{*})+\{1-q^{*}\}\{1-g(q^{*})\}.\label{equilibrium_defn}
\end{equation}
\sloppy This leads to the conclusion that $(a^{*},b^{*},c^{*})$, defined by \eqref{equilibrium_defn}, is, in fact, the \emph{only} equilibrium point corresponding to the differential inclusion mentioned in the statement of Theorem~\ref{thm:main_5}, and while it is evidently contained in the set $\hat{\Lambda}$ described in \eqref{hat{Lambda}_defn}, we cannot conclude that the stochastic process $\{(n^{-1}A_{n},s_{n}^{-1}B_{n},s_{n}^{-1}C_{n})\}$ converges almost surely to $(a^{*},b^{*},c^{*})$ \emph{alone}.

A special case of Theorem~\ref{thm:main_3} is obtained when $(n+1)^{-1}k_{n}^{-1}s_{n+1}$ converges to a limit as $n\rightarrow\infty$:
\begin{theorem}\label{thm:main_4}
Consider the model described in \S\ref{subsec:model_growing_sample_size}, and assume that the limit $\lim_{n\rightarrow\infty}(n+1)^{-1}k_{n}^{-1}s_{n+1}$ exists and equals $\eta$ for some $\eta\in(0,1)$. Assume that the first two criteria of \eqref{k_{n}_sequence_conditions} are satisfied, and that 
\begin{equation}\label{k_{n}_eta_conditions}
\sum_{n\geqslant N}\frac{k_{n}}{s_{n+1}}\left|\frac{s_{n+1}}{(n+1)k_{n}}-\eta\right| \text{ converges}.
\end{equation}
Assume, also, that one of \eqref{A1}, \eqref{A2} and \eqref{A3} is satisfied. Then the sequence $\{(n^{-1}A_{n},s_{n}^{-1}B_{n},s_{n}^{-1}C_{n})\}$ converges almost surely to a (possibly path dependent) compact, connected, internally chain transitive invariant set, contained in $\hat{\mathcal{S}}$, corresponding to the autonomous ODE 
\begin{equation}
\begin{bmatrix}
\dot{a}(t)\\
\dot{b}(t)\\
\dot{c}(t)
\end{bmatrix}=\begin{bmatrix}
\eta\hat{h}_{1}(a(t),b(t),c(t))\\
\hat{h}_{2}(a(t),b(t),c(t))\\
\hat{h}_{3}(a(t),b(t),c(t))
\end{bmatrix},\quad \text{with} \quad \begin{bmatrix}
a(0)\\
b(0)\\
c(0)
\end{bmatrix}=\begin{bmatrix}
A_{N}/N\\
0\\
0
\end{bmatrix},\label{ODE_special_case}
\end{equation}
where $\hat{h}_{1}$, $\hat{h}_{2}$ and $\hat{h}_{3}$ are as defined via \eqref{hat{h}_defn_growing_size}, \eqref{hat{H}_{1}_defn}, \eqref{hat{H}_{2}_defn} and \eqref{hat{H}_{3}_defn}.
\end{theorem} 

We include here a brief discussion consisting of examples of sequences $\{k_{n}:n\geqslant N\}$ for which \eqref{k_{n}_sequence_conditions} is satisfied. Fixing $0<\zeta\leqslant 1$ and constants $\pi_{1},\pi_{2}>0$ (if $\zeta=1$, one would usually set $\max\{\pi_{1},\pi_{2}\}\leqslant 1$), if we set $k_{n}=\pi_{1}n^{\zeta}$ for each even $n$ and $k_{n}=\pi_{2}n^{\zeta}$ for each odd $n$, for all $n\geqslant N$, we see that 
\begin{align}
{}&\frac{s_{2n+1}}{(2n+1)k_{2n}}=\frac{1}{\pi_{1}(2n)^{\zeta}(2n+1)}\left\{\sum_{i=1}^{n}\pi_{1}(2i)^{\zeta}+\sum_{i=1}^{n}\pi_{2}(2i-1)^{\zeta}\right\}+o(1)\nonumber\\
={}&\frac{1}{\pi_{1}(2n)^{\zeta}(2n+1)}\left\{(\pi_{1}-\pi_{2})\sum_{i=1}^{n}(2i)^{\zeta}+\pi_{2}\sum_{i=1}^{2n}i^{\zeta}\right\}+o(1)\nonumber\\
={}&\frac{1}{\pi_{1}(2n)^{\zeta}(2n+1)}\left\{(\pi_{1}-\pi_{2})2^{\zeta}\int_{1}^{n}x^{\zeta}dx(1+o(1))+\pi_{2}\int_{1}^{2n}x^{\zeta}dx(1+o(1))\right\}+o(1)
\rightarrow\frac{\pi_{1}+\pi_{2}}{2\pi_{1}(\zeta+1)},\nonumber
\end{align}
and
\begin{align}
{}&\frac{s_{2n+2}}{(2n+2)k_{2n+1}}=\frac{1}{\pi_{2}(2n+1)^{\zeta}(2n+2)}\left\{\sum_{i=1}^{n}\pi_{1}(2i)^{\zeta}+\sum_{i=1}^{n+1}\pi_{2}(2i-1)^{\zeta}\right\}+o(1)\nonumber\\
={}&\frac{1}{\pi_{2}(2n+1)^{\zeta}(2n+2)}\left\{(\pi_{1}-\pi_{2})\sum_{i=1}^{n}(2i)^{\zeta}+\pi_{2}\sum_{i=1}^{2n+1}i^{\zeta}\right\}+o(1)\nonumber\\
={}&\frac{1}{\pi_{2}(2n+1)^{\zeta}(2n+2)}\left\{(\pi_{1}-\pi_{2})2^{\zeta}\int_{1}^{n}x^{\zeta}dx(1+o(1))+\pi_{2}\int_{1}^{2n+1}x^{\zeta}dx(1+o(1))\right\}+o(1)
\rightarrow\frac{\pi_{1}+\pi_{2}}{2\pi_{2}(\zeta+1)}.\nonumber
\end{align}
Assuming $\pi_{1}>\pi_{2}$, we see, right away, that 
\begin{equation}
\liminf_{n\rightarrow\infty}\frac{s_{n+1}}{(n+1)k_{n}}=\frac{\pi_{1}+\pi_{2}}{2\pi_{1}(\zeta+1)} \quad \text{and} \quad \limsup_{n\rightarrow\infty}\frac{s_{n+1}}{(n+1)k_{n}}=\frac{\pi_{1}+\pi_{2}}{2\pi_{2}(\zeta+1)}.\nonumber
\end{equation}
By definition of $\limsup$ and $\liminf$, we know that, given an \emph{arbitrarily small} $\epsilon>0$, we may choose $N$ sufficiently large such that 
\begin{equation}
\eta_{1}=\inf\{\frac{s_{n+1}}{(n+1)k_{n}}:n\geqslant N\}>\frac{\pi_{1}+\pi_{2}}{2\pi_{1}(\zeta+1)}-\epsilon \text{ and } \eta_{2}=\sup\{\frac{s_{n+1}}{(n+1)k_{n}}:n\geqslant N\}<\frac{\pi_{1}+\pi_{2}}{2\pi_{2}(\zeta+1)}+\epsilon,\nonumber
\end{equation}
demonstrating that the third criterion stated in \eqref{k_{n}_sequence_conditions} is satisfied for $\epsilon$, and accordingly, $N$, chosen suitably. Depending on $\zeta$, we may even choose $\pi_{1}$ and $\pi_{2}$ such that $\eta_{2}\leqslant 1$, which is a criterion appearing in the statement of Theorem~\ref{thm:main_5}. Moreover, given any arbitrarily small $\delta>0$, there exists $n_{\delta}\in\mathbb{N}$ sufficiently large such that, for all $n\geqslant n_{\delta}$:
\begin{align}
\eta_{1}-\delta<\frac{s_{n+1}}{k_{n}(n+1)}<\eta_{2}+\delta \Longleftrightarrow \frac{1}{(\eta_{2}+\delta)(n+1)}<\frac{k_{n}}{s_{n+1}}<\frac{1}{(\eta_{1}-\delta)(n+1)},\nonumber
\end{align} 
and these bounds, together, ensure that the first two criteria stated in \eqref{k_{n}_sequence_conditions} are satisfied as well. Setting $\pi_{1}=\pi_{2}=\pi$, so that $k_{n}=\pi n^{\zeta}$ for all $n\in\mathbb{N}$ with $n\geqslant N$, we have 
\begin{align}
s_{n+1}={}&
\pi\sum_{i=N}^{n}\int_{i}^{i+1}x^{\zeta}dx+\pi\sum_{i=N}^{n}\int_{i}^{i+1}\left(i^{\zeta}-x^{\zeta}\right)dx\nonumber\\
={}&\pi\int_{N}^{n+1}x^{\zeta}dx+\pi\sum_{i=N}^{n}\int_{i}^{i+1}\zeta(i-x)\xi_{1}(x)^{\zeta-1}dx\quad\text{for some }i<\xi_{1}(x)<x;\nonumber\\
={}&\pi\left\{\frac{(n+1)^{\zeta+1}}{\zeta+1}-\frac{N^{\zeta+1}}{\zeta+1}\right\}+O\left(\sum_{i=N}^{n}\int_{i}^{i+1}x^{\zeta-1}dx\right)=\frac{\pi(n+1)^{\zeta+1}}{\zeta+1}+O\left((n+1)^{\zeta}\right),\label{intermediate_1}
\end{align}
where the second-last step leading to the derivation of \eqref{intermediate_1} is obtained by noting that $|i-x|\leqslant 1$ for all $x\in[i,i+1]$, and since $N\leqslant i<\xi_{1}(x)<x\leqslant(i+1)$ and $\zeta\leqslant 1$, we have $\xi_{1}(x)^{\zeta-1}/x^{\zeta-1}=\{x/\xi_{1}(x)\}^{1-\zeta}\leqslant\{(i+1)/i\}^{1-\zeta}\leqslant\{(N+1)/N\}^{1-\zeta}$. From \eqref{intermediate_1}, we have $k_{n}/s_{n+1}\leqslant(\zeta+1)\{n+O(1)\}^{-1}$, as well as 
\begin{align}
\frac{s_{n+1}}{(n+1)k_{n}}-\frac{1}{\zeta+1}
={}&\frac{1}{(\zeta+1)n^{\zeta}}\left\{(n+1)^{\zeta}-n^{\zeta}\right\}+O\left(\frac{(n+1)^{\zeta-1}}{n^{\zeta}}\right)=O(n^{-1}),\nonumber
\end{align}
where the last step follows from noting that $(n+1)^{\zeta}-n^{\zeta}=\zeta \xi_{2}(n)^{\zeta-1}$ for some $n<\xi_{2}(n)<n+1$, so that $(n+1)^{\zeta}-n^{\zeta}\leqslant \zeta n^{\zeta-1}$ (keeping in mind that $\zeta-1\leqslant 0$). These estimates show that \begin{enumerate*}
\item the limit $\eta$ of the sequence $\{(n+1)^{-1}k_{n}^{-1}s_{n+1}\}$ exists and equals $(\zeta+1)^{-1}$, and
\item the criterion of \eqref{k_{n}_eta_conditions} is satisfied.
\end{enumerate*} 

We end \S\ref{sec:main_results} with the following proposition, similar in flavour to Proposition~\ref{prop:examples_g_unique_root_of_G}:
\begin{prop}\label{prop:hat{G}_unique_root_examples}
Recall $f$ and $g$ from \eqref{X_{n+1}_conditional_distribution} and \eqref{g_defn} respectively. The function $g$ is monotonically increasing, and the function $\hat{G}$, defined in \eqref{hat{G}_defn}, has a unique root in $[0,1]$, whenever
\begin{enumerate}
\item $f(x)=x$ for all $x\in[0,1]$, with $p\in(1/2,1)$ and
\begin{equation}
\alpha\geqslant\frac{(2p-1)}{4(1-p)},\label{alpha_inequality}
\end{equation}
\item $f(x)=x$ for all $x\in[0,1]$, with $p\in(1/2,1)$, when \eqref{alpha_inequality} fails to hold but we have 
\begin{equation}
\text{either }\hat{G}\left(\frac{1}{2}-\frac{\sqrt{3}}{6}\sqrt{1-\frac{4\alpha(1-p)}{2p-1}}\right)>0 \text{ or } \hat{G}\left(\frac{1}{2}+\frac{\sqrt{3}}{6}\sqrt{1-\frac{4\alpha(1-p)}{2p-1}}\right)<0.\label{root_inequalities_hat{G}}
\end{equation}
\item $f(x)=e^{x-1}$ for all $x\in[0,1]$, $p\in(1/2,1)$ and $\alpha>0.15218(2p-1)/(1-p)$.
\end{enumerate}
\end{prop}

\section{Results from the literature crucial for proving our main results}\label{sec:results_from_literature}
This section is dedicated to stating, sometimes with modifications or slight generalizations, results from the literature that are crucial for proving the main results of this paper. Before we begin, we include here a small paraphernalia of definitions of terms that have appeared in the statements and proofs of the results in \S\ref{sec:main_results} as well as in Lemmas~\ref{lem:positively_invariant_domain} and \ref{lem:strongly_positively_invariant_growing_sample_size}.

Given a function $F:\mathbb{R}^{d}\rightarrow\mathbb{R}^{d}$, the ordinary differential equation $\dot{\mathbf{x}}(t)=F(\mathbf{x}(t))$, with initial condition $\mathbf{x}(0)=\mathbf{a}$, is said to be \emph{well-posed} if, for every choice of $\mathbf{a}\in\mathbb{R}^{d}$, there exists a unique solution $\mathbf{x}_{\mathbf{a}}(t)$ for all $t\geqslant 0$, and the map $\mathbf{a}\mapsto\mathbf{x}_{\mathbf{a}}(\cdot)\in C([0,\infty),\mathbb{R}^{d})$ is continuous (here, $C([0,\infty),\mathbb{R}^{d})$ indicates the space of all continuous functions from $[0,\infty)$ to $\mathbb{R}^{d}$). A well-known sufficient condition for this ODE to be well-posed is that $F$ is Lispchitz throughout $\mathbb{R}^{d}$, i.e.\ there exists $L>0$ such that $||F(\mathbf{x})-F(\mathbf{y})||\leqslant L||\mathbf{x}-\mathbf{y}||$ for all $\mathbf{x},\mathbf{y}\in\mathbb{R}^{d}$ (see, for instance, Theorem B.1 of \cite{borkar2008stochastic}). The map $\Phi_{t}:\mathbb{R}^{d}\rightarrow\mathbb{R}^{d}$, with $\Phi_{t}(\mathbf{a})=\mathbf{x}_{\mathbf{a}}(t)$ for every $\mathbf{a}\in\mathbb{R}^{d}$ and every $t\in\mathbb{R}$, is called the \emph{flow} of this ODE (assuming it to be well-posed). It is termed a \emph{semiflow} if $\Phi_{t}(\mathbf{a})=\mathbf{x}_{\mathbf{a}}(t)$ for every $\mathbf{a}\in\mathbb{R}^{d}$ and every $t\geqslant 0$.

According to Chapter 2 of \cite{borkar2008stochastic}, we call $A\subset\mathbb{R}^{d}$ an \emph{invariant set} corresponding to the ODE mentioned in the previous paragraph, assuming it to be well-posed, if $\mathbf{x}_{\mathbf{a}}(t)\in A$ for all $t\in\mathbb{R}$ as long as $\mathbf{a}\in A$ -- equivalently, $\Phi_{t}(A)\subset A$ for all $t\in\mathbb{R}$ (in other words, a trajectory initiated inside $A$ must stay inside $A$ at all times, whether going forward or backward). On the other hand, \S 5 of \cite{benaim2006dynamics} imposes a stricter criterion: a set $A\subset\mathbb{R}^{d}$ is called invariant if $\Phi_{t}(A)=A$ for all $t\in\mathbb{R}$, requiring every element of $A$ to have a pre-image (in $A$) under $\Phi_{t}$, for all $t\in\mathbb{R}$. The two main definitions (in this paper) that involve the notion of invariant sets are those of \emph{internally chain transitive invariant sets} and \emph{attractors}. It follows (from the definitions in the next paragraph) that for any $A\subset\mathbb{R}^{d}$ to be either internally chain transitive or an attractor, $A$ must be invariant in the stricter sense of the term (i.e.\ according to the definition given by \cite{benaim2006dynamics}). 

We call $A$ \emph{positively invariant} if $\mathbf{x}_{\mathbf{a}}(t)\in A$ for all $t\geqslant 0$ as long as $\mathbf{a}\in A$ (equivalently, $\Phi_{t}(A)\subset A$ for all $t\geqslant 0$). We call $A$ \emph{internally chain transitive} if, in addition to being invariant, 
\begin{enumerate*}
\item it is compact, and 
\item for all $\mathbf{u},\mathbf{v}\in A$ and all $\epsilon>0$ and $T>0$, there exist $n\in\mathbb{N}$ and $\mathbf{u}_{1},\ldots,\mathbf{u}_{n-1}\in A$ such that, if we set $\mathbf{u}_{0}=\mathbf{u}$ and $\mathbf{u}_{n}=\mathbf{v}$, then the solution $\mathbf{x}_{\mathbf{u}_{i-1}}$, initiated at $\mathbf{u}_{i-1}$, meets the $\epsilon$-neighourhood of $\mathbf{u}_{i}$ after a time $\geqslant T$, for each $i\in[n]$ -- equivalently, there exists $t_{i}\geqslant T$ such that $||\mathbf{x}_{\mathbf{u}_{i-1}}(t_{i})-\mathbf{u}_{i}||=||\Phi_{t_{i}}(\mathbf{u}_{i-1})-\mathbf{u}_{i}||<\epsilon$.
\end{enumerate*}
An invariant set $A$ is said to be \emph{Lyapunov stable} if for each $\epsilon>0$, there exists $\delta>0$ such that every trajectory $\mathbf{x}_{\mathbf{a}}(\cdot)$ initiated in the $\delta$-neighbourhood of $A$ remains forever inside the $\epsilon$-neighbourhood of $A$ -- in other words, whenever dist$(\mathbf{a},A)=\inf\{||\mathbf{a}-\mathbf{b}||:\mathbf{b}\in A\}<\delta$, we have dist$(\mathbf{x}_{\mathbf{a}}(t),A)=\text{dist}(\Phi_{t}(\mathbf{a}),A)<\epsilon$ for all $t\geqslant 0$. A compact invariant set $A$ is called an \emph{attractor} if it is Lyapunov stable and has a positively invariant open neighbourhood $O$ such that every trajectory $\mathbf{x}_{\mathbf{a}}(\cdot)$ initiated inside $O$ converges uniformly to $A$ -- equivalently, dist$(\Phi_{t}(\mathbf{a}),A)=\text{dist}(\mathbf{x}_{\mathbf{a}}(t),A)\rightarrow 0$ as $t\rightarrow\infty$, uniformly in $\mathbf{a}\in O$. The neighbourhood $O$ is referred to as a \emph{fundamental neighbourhood} of $A$. We refer to Appendix B of \cite{borkar2008stochastic} as well as \S 5 of \cite{benaim2006dynamics} for these definitions. 
 
A generalization of ODEs that plays a crucial role in the analysis of the model described in \S\ref{subsec:model_growing_sample_size} is the notion of \emph{differential inclusions}. For definitions pertaining to differential inclusions, we refer to \cite{benaim2005stochastic}. We consider a \emph{closed, set-valued map} $F$, which implies that $F(\mathbf{x})\subset\mathbb{R}^{d}$ for each $\mathbf{x}\in\mathbb{R}^{d}$, and the graph $\{(\mathbf{x},\mathbf{y})\in\mathbb{R}^{d}\times\mathbb{R}^{d}:\mathbf{y}\in F(\mathbf{x})\}$ is a closed subset of $\mathbb{R}^{d}\times\mathbb{R}^{d}$. It is assumed that $F(\mathbf{x})$ is non-empty, compact and convex for each $\mathbf{x}\in\mathbb{R}^{d}$, and there exists $c>0$ such that $\sup\{||\mathbf{z}||:\mathbf{z}\in F(\mathbf{x})\}\leqslant c(1+||\mathbf{x}||)$ for all $\mathbf{x}\in\mathbb{R}^{d}$. A solution to the differential inclusion $\dot{\mathbf{x}}\in F(\mathbf{x})$, with initial value $\mathbf{a}\in\mathbb{R}^{d}$, is an absolutely continuous mapping $\mathbf{x}_{\mathbf{a}}:\mathbb{R}\rightarrow\mathbb{R}^{d}$ such that $\mathbf{x}_{\mathbf{a}}(0)=\mathbf{a}$ and $\dot{\mathbf{x}_{\mathbf{a}}}(t)\in F(\mathbf{x}_{\mathbf{a}}(t))$ for almost every $t\in\mathbb{R}$. This induces a set-valued dynamical system $\Phi=\{\Phi_{t}\}_{t\in\mathbb{R}}$, with $\Phi_{t}=\{\Phi_{t}(\mathbf{a}):\mathbf{a}\in\mathbb{R}^{d}\}$, where $\Phi_{t}(\mathbf{a})$ consists of $\mathbf{x}_{\mathbf{a}}(t)$ for all solutions $\mathbf{x}_{\mathbf{a}}(\cdot)$, for each $\mathbf{a}\in\mathbb{R}^{d}$ and each $t\in\mathbb{R}$. We let $\Phi(\mathbf{a})$ consist of all solutions $\mathbf{x}_{\mathbf{a}}(\cdot)$, for each $\mathbf{a}\in\mathbb{R}^{d}$. A set $A\subset\mathbb{R}^{d}$ is termed \emph{invariant} if for every $\mathbf{a}\in A$, there exists at least one solution $\mathbf{x}_{\mathbf{a}}(\cdot)$ such that $\mathbf{x}_{\mathbf{a}}(t)\in A$ for each $t\in\mathbb{R}$, and it is termed \emph{strongly positively invariant} if $\Phi_{t}(\mathbf{a})\subset A$ for all $t\geqslant 0$ and all $\mathbf{a}\in A$. 

We call a compact subset $A$ of $\mathbb{R}^{d}$ \emph{internally chain transitive} if for every $\epsilon>0$, every $T>0$, and all $\mathbf{a},\mathbf{b}\in A$, there exist $n\in\mathbb{N}$ and points $\mathbf{a}_{1},\ldots,\mathbf{a}_{n-1}\in A$, along with \emph{some} solution $\mathbf{x}_{\mathbf{a}_{i-1}}$ for each $i\in[n]$ (where we set $\mathbf{a}_{0}=\mathbf{a}$ and $\mathbf{a}_{n}=\mathbf{b}$), and times $t_{1},\ldots,t_{n}$ each of which exceeds $T$, such that 
\begin{enumerate*}
\item $\mathbf{x}_{\mathbf{a}_{i-1}}(s)\in A$ for all $0\leqslant s\leqslant t_{i}$ for each $i\in[n]$, and
\item $||\mathbf{x}_{\mathbf{a}_{i-1}}(t_{i})-\mathbf{a}_{i}||<\epsilon$ for all $i\in[n]$. 
\end{enumerate*}
Given a closed, invariant set $L$ corresponding to the differential inclusion $\dot{\mathbf{x}}\in F(\mathbf{x})$, let $\Phi_{t}^{L}(\mathbf{a})$ indicate the set of $\mathbf{x}_{\mathbf{a}}(t)$ for all solutions $\mathbf{x}_{\mathbf{a}}(\cdot)$ such that $\mathbf{x}_{\mathbf{a}}(\mathbb{R})\subset L$ (in other words, only those solutions $\mathbf{x}_{\mathbf{a}}(\cdot)$ that lie \emph{entirely} inside $L$), for each $\mathbf{a}\in L$. We indicate by $\Phi^{L}(\mathbf{a})$ the collection of all such solutions (i.e.\ $\Phi^{L}(\mathbf{a})=\{\mathbf{x}_{\mathbf{a}}(\cdot):\mathbf{x}_{\mathbf{a}}(\mathbb{R})\subset L\}$) initiated from $\mathbf{a}$, for each $\mathbf{a}\in L$. A compact subset $A$ of $L$ is called an \emph{attracting set} for $\Phi^{L}$ if there exists a neighbourhood $U$ of $A$ (in the topology induced on $L$) such that for all $\epsilon>0$, there exists $t_{\epsilon}>0$ such that $\Phi^{L}_{t}(U)=\bigcup_{\mathbf{a}\in U}\Phi^{L}_{t}(\mathbf{a})\subset N^{\epsilon}(A)$ for all $t\geqslant t_{\epsilon}$, where $N^{\epsilon}(A)=\{\mathbf{x}:\text{dist}(\mathbf{x},A)<\epsilon\}$. If, in addition, $A$ is invariant, then $A$ is termed an \emph{attractor} for $\Phi^{L}$, in which case the set $U$ is referred to as a \emph{fundamental neighbourhood} of $A$ for $\Phi^{L}$. 

The final definition to be included in this brief discussion is that of \emph{perturbed solutions} -- paths that are obtained as (deterministic or random) perturbations of solutions to the differential inclusion mentioned in the previous paragraph. A function $\mathbf{y}:[0,\infty)\rightarrow\mathbb{R}^{d}$ is called a \emph{perturbed solution} to the differential inclusion $\dot{\mathbf{x}}\in F(\mathbf{x})$ if 
\begin{enumerate*}
\item it is absolutely continuous,
\item there exists a locally integrable function $t\mapsto U(t)$ such that $\lim_{t\rightarrow\infty}\sup\{||\int_{t}^{t+v}U(s)ds||:0\leqslant v\leqslant T\}=0$ for each $T>0$, and
\item $\dot{\mathbf{y}}(t)-U(t)\in F^{\delta(t)}(\mathbf{y}(t))$ for almost every $t\geqslant 0$, for some function $\delta:[0,\infty)\rightarrow[0,\infty)$ with $\delta(t)\rightarrow 0$ as $t\rightarrow\infty$, where $F^{\delta}(\mathbf{x})=\{\mathbf{z}:\exists \mathbf{w} \text{ such that }||\mathbf{w}-\mathbf{x}||<\delta \text{ and dist}(\mathbf{z},F(\mathbf{w}))<\delta\}$.
\end{enumerate*}

We are now ready for the main content of \S\ref{sec:results_from_literature}, and we begin by stating a result that pertains to almost sure convergence of the iterates of a stochastic approximation process corresponding to an ODE:
\begin{theorem}\label{thm:borkar_a.s.}
Consider the stochastic approximation process $\{Z_{n}\}$, with $Z_{n}\in\mathbb{R}^{d}$, given by
\begin{equation}\label{sa_borkar_general}
Z_{n+1}=Z_{n}+a_{n+1}\left\{F(Z_{n})+\Delta M_{n+1}+\delta_{n}\right\} \text{ for }n\in\mathbb{N}_{0},
\end{equation}
with prescribed $Z_{0}$ and the following assumptions:
\begin{enumerate}[label=(B\arabic*), ref=B\arabic*]
\item \label{borkar_1} the map $F:\mathbb{R}^{d}\rightarrow\mathbb{R}^{d}$ is Lipschitz,
\item \label{borkar_2} the step-size sequence $\{a_{n}\}$, of positive scalars, satisfies $\sum_{n}a_{n}=\infty$ and $\sum_{n}a_{n}^{2}<\infty$,
\item \label{borkar_3} the sequence $\{\Delta M_{n+1}\}$ is a martingale difference sequence with respect to the filtration $\{\mathcal{F}_{n+1}\}$, where $\mathcal{F}_{n}$ is the $\sigma$-field consisting of all information about the process up to and including epoch $n$, and $\{\delta_{n}\}$ is a sequence of additional error terms adapted to $\{\mathcal{F}_{n}\}$, such that $\sum_{i=0}^{n-1}a_{i+1}(\Delta M_{i+1}+\delta_{i})$ converges almost surely as $n\rightarrow\infty$, 
\item \label{borkar_4} and $\sup\{||Z_{n}||:n\in\mathbb{N}_{0}\}<\infty$ almost surely, where $||\cdot||$ indicates the usual Euclidean norm in $\mathbb{R}^{d}$.
\end{enumerate}
Then, almost surely, $\{Z_{n}\}$ converges to a (possibly sample path dependent) compact connected internally chain transitive invariant set corresponding to the autonomous ODE $\dot{x}(t)=F(x(t))$ for $t\geqslant 0$.
\end{theorem}
Note that Theorem~\ref{thm:borkar_a.s.} is a slight generalization of Theorem 2, Chapter 2 of \cite{borkar2008stochastic} in that, we consider two different error sequences, $\{\Delta M_{n+1}\}$ and $\{\delta_{n}\}$, and instead of bounding $\E[||\Delta M_{n+1}||^{2}|\mathcal{F}_{n}]$ individually for each $n$, we impose Assumption~\eqref{borkar_3}, since this ensures that Lemma 1, Chapter 2 of \cite{borkar2008stochastic} remains true. We may now state Theorem 2, Chapter 2 of \cite{borkar2008stochastic}, as a corollary of Theorem~\ref{thm:borkar_a.s.} (here, $\mathbf{0}$ indicates the tuple in $\mathbb{R}^{d}$ in which each coordinate equals $0$):
\begin{corollary}\label{cor:borkar_a.s.}
Consider the set-up described in Theorem~\ref{thm:borkar_a.s.}, with $\delta_{n}=\mathbf{0}$ for each $n\in\mathbb{N}_{0}$. Let Assumptions~\eqref{borkar_1}, \eqref{borkar_2} and \eqref{borkar_4} be satisfied, and let there exist $K>0$ such that $\E[||\Delta M_{n+1}||^{2}|\mathcal{F}_{n}]\leqslant K(1+||Z_{n}||^{2})$ almost surely for each $n\in\mathbb{N}_{0}$. Then, the same conclusion as drawn in the statement of Theorem~\ref{thm:borkar_a.s.} remains true.
\end{corollary}

We now state a modification to Theorem 7.1 of \cite{borkar2008stochastic}, since it helps us examine the convergence of suitably scaled fluctuations of the stochastic process resulting from the model in \S\ref{subsec:model_fixed_sample_size} around the trajectory dictated by the corresponding autonomous ODE.
\begin{theorem}\label{thm:borkar_fclt}
Consider the stochastic approximation process in \eqref{sa_borkar_general}, with $\delta_{n}=\mathbf{0}$ for each $n$. In addition to Assumptions~\eqref{borkar_1}, \eqref{borkar_2} and \eqref{borkar_4}, we assume  
\begin{enumerate}[label=(B'\arabic*), ref=B'\arabic*]
\item \label{borkar_5} that $F$ is continuously differentiable, and its Jacobian $J_{F}$ is uniformly Lipschitz,
\item \label{borkar_6} that $\alpha:=\lim_{n\rightarrow\infty}\{a_{n+1}^{-1}-a_{n}^{-1}\}$ exists and is finite,
\item \label{borkar_7} that $\sup_{n}\E[||Z_{n}||^{4}]<\infty$,
\item \label{borkar_8} that for some $K>0$, and some $Q:\mathbb{R}^{d}\rightarrow\mathbb{R}^{d\times d}$ such that $Q(\mathbf{x})$ is a non-negative definite symmetric matrix for each $\mathbf{x}\in\mathbb{R}^{d}$, we have
\begin{equation}
\E\left[\Delta M_{n+1}\Delta M_{n+1}^{T}\big|\mathcal{F}_{n}\right]=Q(Z_{n}) \text{ and } \E\left[\left|\left|\Delta M_{n+1}\right|\right|^{4}\big|\mathcal{F}_{n}\right]\leqslant K\left(1+||Z_{n}||^{4}\right),\nonumber
\end{equation}
where $A^{T}$, for any matrix $A$ with real entries, indicates its transpose.
\end{enumerate}
Fix \emph{any} $T>0$. Writing $t_{0}=0$ and $t_{n}=\sum_{i\in[n]}a_{i}$ for $n\in\mathbb{N}$, we set $m_{n}=\min\{m\geqslant n: t_{m}\geqslant t_{n}+T\}$ for each $n\in\mathbb{N}$. For each fixed $n\in\mathbb{N}$, let $\{Z^{n}(t): t\geqslant t_{n}\}$ indicate the solution to
\begin{equation}
\dot{\mathbf{x}}(t)=F(\mathbf{x}(t)) \text{ for all }t\geqslant t_{n}, \quad \text{with} \quad \mathbf{x}(t_{n})=Z_{n},\nonumber
\end{equation}
and let $W_{j}=a_{j+1}^{-1/2}\{Z_{j}-Z^{n}(t_{j})\}$ for each $j\geqslant n$. We now define the process $\{W^{n}(t):t\in [t_{n},t_{n}+T]\}$: 
\begin{equation}
W^{n}(t)=W_{j}+\frac{t-t_{j}}{t_{j+1}-t_{j}}(W_{j+1}-W_{j}) \text{ for }t\in[t_{j},t_{j+1}], \text{ for each }j\in\{n,n+1,\ldots,m_{n}-1\}.\nonumber  
\end{equation}  
Finally, we set $\tilde{W}^{n}(t)=W^{n}(t_{n}+t)$ for all $t\in[0,T]$. Likewise, we define $\tilde{Z}^{n}(t)=Z^{n}(t_{n}+t)$ for all $t\in[0,T]$. Then the laws of the processes $\{(\tilde{Z}^{n}(\cdot),\tilde{W}^{n}(\cdot))\}_{n\in\mathbb{N}}$ are relatively compact in the space $\mathcal{P}(C([0,T];\mathbb{R}^{d}))^{2}$, where $C([0,T];\mathbb{R}^{d})$ is the set of all continuous functions from $[0,T]$ to $\mathbb{R}^{d}$ and $\mathcal{P}(C([0,T];\mathbb{R}^{d}))^{2}$ is the space of all probability measures on $C([0,T];\mathbb{R}^{d})\times C([0,T];\mathbb{R}^{d})$.

If a subsequence of $\{(\tilde{Z}^{n}(\cdot),\tilde{W}^{n}(\cdot))\}_{n\in\mathbb{N}}$ converges in law to a limiting process $\{z^{*}(\cdot),w^{*}(\cdot)\}$, then $z^{*}(\cdot)$ is a solution to the ODE $\dot{\mathbf{x}}(t)=F(\mathbf{x}(t))$ for $t\geqslant 0$, belonging to a compact, connected, internally chain transitive invariant set corresponding to this ODE, and each of
\begin{align}
&w^{*}(t)-\int_{0}^{t}\left(J_{F}\left(z^{*}(s)\right)+\frac{\alpha}{2}I_{d}\right)w^{*}(s)ds \quad \text{and}\nonumber\\
&\left[w^{*}(t)-\int_{0}^{t}\left(J_{F}\left(z^{*}(s)\right)+\frac{\alpha}{2}I_{d}\right)w^{*}(s)ds\right]\left[w^{*}(t)-\int_{0}^{t}\left(J_{F}\left(z^{*}(s)\right)+\frac{\alpha}{2}I_{d}\right)w^{*}(s)ds\right]^{T}-\int_{0}^{t}Q\left(z^{*}(s)\right)ds,\nonumber
\end{align}
is a martingale for $t\in[0,T]$, where $I_{d}$ indicates the $d$-dimensional identity matrix.
\end{theorem}
Although this theorem differs from Theorem 7.1 of \cite{borkar2008stochastic} in that 
\begin{enumerate*} 
\item Assumption~\eqref{borkar_8} is not quite Assumption (A4), Chapter 7 of \cite{borkar2008stochastic}, 
\item and we can no longer write the limiting process $w^{*}(\cdot)$ as the solution to a stochastic differential equation (such as Equation (7.14), Chapter 7 of \cite{borkar2008stochastic}),
\end{enumerate*}
the proof of Theorem~\ref{thm:borkar_fclt} follows the same line of argument as that outlined in Chapter 7 of \cite{borkar2008stochastic} (more specifically, the proof of relative compactness of $\{(\tilde{Z}^{n}(\cdot),\tilde{W}^{n}(\cdot))\}_{n\in\mathbb{N}}$ is accomplished by following the line of reasoning presented in \S 7.2 of \cite{borkar2008stochastic}, culminating in Lemma 7.5, and the characterization of $w^{*}(\cdot)$ is deduced exactly as has been shown in \S 7.3 of \cite{borkar2008stochastic}, up to but not including Equation (7.14)). Our inability to express the limiting process $w^{*}(\cdot)$, resulting from the model described in \S\ref{subsec:model_fixed_sample_size}, stems from the fact that, when it comes to the corresponding covariance matrix function $Q:\mathcal{S}\rightarrow\mathbb{R}^{3\times 3}$ (with $\mathcal{S}$ as defined in \eqref{domain_defn}), the rank of $Q(a,b,c)$, for $(a,b,c)\in\mathcal{S}$, does not remain constant throughout $\mathcal{S}$. It is immediate that whenever $a=b=0$, at least one eigenvalue of $Q(a,b,c)$ equals $0$, making it positive-semidefinite instead of positive definite. This prevents us from being able to apply Theorem 5.2.2 of \cite{stroock2007multidimensional} to our set-up to claim the existence of a unique function $D:\mathcal{S}\rightarrow\mathbb{R}^{3\times 3}$ such that 
\begin{enumerate*}
\item $D(a,b,c)$ is symmetric and non-negative definite for each $(a,b,c)\in\mathcal{S}$,
\item $D(a,b,c)^{2}=Q(a,b,c)$ for each $(a,b,c)\in\mathcal{S}$,
\item and $D$ is Lipschitz on $\mathcal{S}$, or, equivalently, there exists some constant $M>0$ such that the operator norm $||D(a_{1},b_{1},c_{1})-D(a_{2},b_{2},c_{2})||\leqslant M||(a_{1},b_{1},c_{1})-(a_{2},b_{2},c_{2})||$ for all $(a_{1},b_{1},c_{1}), (a_{2},b_{2},c_{2})\in\mathcal{S}$.
\end{enumerate*}
Ensuring that $D$ is Lipschitz remains crucial for the stochastic differential equation in Equation (7.14) of \cite{borkar2008stochastic} to be well-posed. Note that one could possibly replace the function $D$ by the Cholesky decomposition factor $L$ of $Q$, i.e.,\ by a function $L:\mathcal{S}\rightarrow\mathbb{R}^{3\times 3}$ such that 
\begin{enumerate*}
\item $L(a,b,c)$ is lower-triangular for each $(a,b,c)\in\mathcal{S}$,
\item and $L(a,b,c)L(a,b,c)^{T}=Q(a,b,c)$ for each $(a,b,c)\in\mathcal{S}$ --
\end{enumerate*}
however, in our set-up, while a Cholesky decomposition factor can be computed explicitly with a fairly neat expression, it is easily seen to be \emph{not} Lipschitz on $\mathcal{S}$.

A final means for expressing our $w^{*}(\cdot)$ in a form analogous to Equation (7.14) of \cite{borkar2008stochastic} would have been to resort to Theorem 5.2.3 of \cite{stroock2007multidimensional}, which would require an extension $\tilde{Q}$ of our function $Q$ from $\mathcal{S}$ to all of $\mathbb{R}^{3}$ such that 
\begin{enumerate*}
\item $\tilde{Q}(a,b,c)$ is a non-negative definite symmetric matrix for each $(a,b,c)\in\mathbb{R}^{3}$ with $\tilde{Q}|_{\mathcal{S}}\equiv Q$,
\item $\tilde{Q}$ has continuous second-order partial derivatives throughout $\mathbb{R}^{3}$,
\item and each second-order partial derivative of $Q$ is bounded throughout $\mathbb{R}^{3}$.
\end{enumerate*}
While Whitney extension results (see, for instance, \S 2.3, Chapter VI of \cite{stein1970singular}) allow us to extend $Q$ from $\mathcal{S}$ to $\tilde{Q}$ on $\mathbb{R}^{3}$ that respects the criteria pertaining to second-order partial derivatives, we can, in no way, keep $\tilde{Q}(a,b,c)$ non-negative definite for each $(a,b,c)\in\mathbb{R}^{3}$, thus preventing us from being able to apply Theorem 5.2.3 of \cite{stroock2007multidimensional} to obtain a Lipschitz square root of $\tilde{Q}$ (and hence, of $Q$).

The next theorem that we state finds applications in the analysis of the model described in \S\ref{subsec:model_growing_sample_size}, and is obtained by combining Proposition 1.3 and Theorem 3.6 of \cite{benaim2005stochastic}. Here, given any subset $A$ of $\mathbb{R}^{d}$, we let $\overline{A}$ indicate its closure.
\begin{theorem}\label{thm:differential_inclusions}
Consider the stochastic approximation process $\{Z_{n}\}$, 
\begin{equation}\label{sa_differential_inclusion} 
Z_{n+1}-Z_{n}-a_{n+1}E_{n+1}\in a_{n+1}F(Z_{n}),\text{ for each }n\in\mathbb{N}_{0},
\end{equation} 
satisfying the following conditions: 
\begin{enumerate}[label=(C\arabic*), ref=C\arabic*]
\item \label{C1} the iterates are almost surely bounded, i.e.\ $\sup_{n}||Z_{n}||<\infty$ almost surely,
\item \label{C2} $F$ is a closed set-valued map on $\mathbb{R}^{d}$, with $F(\mathbf{x})$ a non-empty, compact, convex subset of $\mathbb{R}^{d}$ for each $\mathbf{x}\in\mathbb{R}^{d}$, 
\item \label{C3} $\sup\{||\mathbf{z}||:\mathbf{z}\in F(\mathbf{x})\}\leqslant \gamma(1+||\mathbf{x}||)$ for all $\mathbf{x}\in\mathbb{R}^{d}$, for some $\gamma>0$
\item \label{C4} the step-size sequence $\{a_{n}\}$ is such that $\sum_{n}a_{n}=\infty$ and $\lim_{n\rightarrow\infty}a_{n}=0$, 
\item \label{C5} and having defined $\tau_{0}=0$, $\tau_{n}=\sum_{i\in[n]}a_{i}$ for $n\in\mathbb{N}$, and $m(t)=\sup\{k\geqslant 0:t\geqslant \tau_{k}\}$ for $t>0$,
\begin{equation}\label{error_sequence_cond_differential_inclusions}
\lim_{n\rightarrow\infty}\sup\left\{\left|\left|\sum_{i=n}^{k-1}a_{i+1}E_{i+1}\right|\right|:k\in\left\{n+1,\ldots,m(\tau_{n}+T)\right\}\right\}=0 \text{ almost surely for each }T>0.
\end{equation}
\end{enumerate}
Then the continuous time affine interpolated process $\{Z(t):t\geqslant 0\}$, defined as
\begin{equation}
Z(\tau_{n}+s)=Z_{n}+\frac{s(Z_{n+1}-Z_{n})}{\tau_{n+1}-\tau_{n}} \text{ for each }s\in[0,a_{n+1}),\label{interpolated_process}
\end{equation}
is a perturbed solution to the differential inclusion $\dot{\mathbf{x}}(t)\in F(\mathbf{x}(t))$ for $t\geqslant 0$, and its limit set $L(Z)=\bigcap_{t\geqslant 0}\overline{\{Z(s):s\geqslant t\}}$ is internally chain transitive.
\end{theorem}

For the model in \S\ref{subsec:model_fixed_sample_size}, the iterates $(n^{-1}A_{n},n^{-1}B_{n},n^{-1}C_{n})$, for $n\geqslant N$, come from the compact set $\mathcal{S}$ defined in \eqref{domain_defn}, while for the model in \S\ref{subsec:model_growing_sample_size}, the iterates $(n^{-1}A_{n},s_{n}^{-1}B_{n},s_{n}^{-1}C_{n})$, for $n\geqslant N$, come from the compact set $\hat{\mathcal{S}}$ defined in \eqref{domain_defn_growing_size}. Therefore, Assumption~\eqref{borkar_4} of Theorem~\ref{thm:borkar_a.s.} as well as Assumption~\eqref{C1} of Theorem~\ref{thm:differential_inclusions} is automatically satisfied, and we need not verify these again in either of \S\ref{sec:proofs_fixed_sample_size} and \S\ref{sec:proofs_growing_sample_size}. 

We would like to end \S\ref{sec:results_from_literature} with a brief discussion on weak convergence results -- in particular, convergence in distribution -- for stochastic approximation processes, that are available in the literature so far. A vast majority of such results (see, for instance, Theorem 2.2.12 of \cite{duflo2013random}, Theorems 2.1, 2.2 and 2.3 of \cite{zhang2016central}, among others) assume that the iterates of the stochastic approximation process under consideration converge almost surely to a single equilibrium point, also known as an \emph{attractive target}. A handful of results, such as those proposed in \cite{pelletier1998weak} and \cite{pelletier1999almost}, extend the analysis to scenarios where multiple such attractive targets may exist -- given an attractive target $z^{*}$ such that the probability of the event $\Gamma(z^{*})$ is strictly positive, where $\{Z_{n}\}$ forms the sequence of iterates of the stochastic approximation process described in \eqref{sa_borkar_general} and $\Gamma(z^{*})=\{\omega:Z_{n}(\omega)\rightarrow z^{*}\}$, convergence in distribution of a suitably scaled version of $(Z_{n}-z^{*})$, on the event $\Gamma(z^{*})$, to the stationary distribution of a diffusion process has been established. While such results may well be applicable to our set-up, they are unable to capture the full picture, as 
\begin{enumerate*}
\item it is difficult to specify the exact internal structure of an internally chain transitive invariant set in general, and 
\item such a set may contain structures a lot more complicated than individual equilibrium points.
\end{enumerate*}


\section{Proofs of the results from \S\ref{subsec:main_results_fixed_sample_size}}\label{sec:proofs_fixed_sample_size}
This section is dedicated to the proofs of the results, stated in \S\ref{subsec:main_results_fixed_sample_size}, pertaining to the model described in \S\ref{subsec:model_fixed_sample_size}, i.e.\ where the size of the sample, $k$, remains fixed as $n$ varies. Representation of the model in \S\ref{subsec:model_fixed_sample_size} as a stochastic approximation process proves crucial for carrying out our analysis. Recall the definitions of $A_{n}$, $B_{n}$ and $C_{n}$ from the beginning of \S\ref{sec:main_results}. Recall, also, that in a directed graph, the sum of the in-degrees of all vertices is equal to the sum of the out-degrees of all vertices. For the model in \S\ref{subsec:model_fixed_sample_size}, the vertices $v_{1},\ldots,v_{N}$ have no out-degree in $G_{n}$, while $v_{j}$, for each $j\in\{N+1,\ldots,n\}$, has out-degree equal to $P_{j-1}$ if $X_{j}=+1$, and equal to $(k-P_{j-1})$ if $X_{j}=-1$. We thus have 
\begin{align}
C_{n}={}&\sum_{i=N+1}^{n}P_{i-1}\chi\left\{X_{i}=+1\right\}+\sum_{i=N+1}^{n}\left(k-P_{i-1}\right)\chi\left\{X_{i}=-1\right\}=\sum_{i=N+1}^{n}P_{i-1}X_{i}+k\sum_{i=N+1}^{n}\chi\left\{X_{i}=-1\right\}.\label{C_{n}_defn_fixed_size}
\end{align}
The stochastic approximation will be written for the sequence $\{(n^{-1}A_{n},n^{-1}B_{n},n^{-1}C_{n})\}$, for $n\geqslant N$. 

From \eqref{sampling_probability_proportional_to_indegree} and \eqref{P_{n}_defn}, it is evident that $P_{n}$, conditioned on $\mathcal{F}_{n}$, follows a binomial distribution in which the total number of trials is equal to $k$, and the probability of success is given by (using the definition of $q_{\beta}$ from \eqref{q_defn}):
\begin{align}
\sum_{m\in[n]}\frac{\indeg_{n}(v_{m})+\beta}{\sum_{j\in[n]}\indeg_{n}(v_{j})+n\beta}\chi\left\{X_{m}=+1\right\}=\frac{B_{n}+\beta A_{n}}{C_{n}+n\beta}=q_{\beta}\left(\frac{A_{n}}{n},\frac{B_{n}}{n},\frac{C_{n}}{n}\right),\label{success_probab_P_{n}_distribution}
\end{align}
where $C_{n}$ is as given by \eqref{C_{n}_defn_fixed_size}. This observation, along with \eqref{X_{n+1}_conditional_distribution}, yields
\begin{align}
{}&\E\left[A_{n+1}-A_{n}\big|\mathcal{F}_{n}\right]=\sum_{i=0}^{k}\Prob\left[X_{n+1}=+1\big|P_{n}=i,\mathcal{F}_{n}\right]\Prob\left[P_{n}=i\big|\mathcal{F}_{n}\right]=H_{1}\left(\frac{A_{n}}{n},\frac{B_{n}}{n},\frac{C_{n}}{n}\right),\label{A_{n+1}-A_{n}_cond_exp_fixed_size}
\end{align}
where $H_{1}$ is the function defined in \eqref{H_{1}_defn}.

To find the conditional expectation of $B_{n+1}-B_{n}$, conditioned on $\mathcal{F}_{n}$, we note that this difference is strictly positive, and in fact, equal to $P_{n}$, if and only if $X_{n+1}=+1$. Therefore,
\begin{align}
\E\left[B_{n+1}-B_{n}\big|\mathcal{F}_{n}\right]=\E\left[P_{n}\chi\left\{X_{n+1}=+1\right\}\big|\mathcal{F}_{n}\right]=H_{2}\left(\frac{A_{n}}{n},\frac{B_{n}}{n},\frac{C_{n}}{n}\right),\label{B_{n+1}-B_{n}_cond_exp_fixed_size}
\end{align}
where $H_{2}$ is the function defined in \eqref{H_{2}_defn}.

Finally, from \eqref{X_{n+1}_conditional_distribution}, \eqref{C_{n}_defn_fixed_size} and \eqref{success_probab_P_{n}_distribution}, we have:
\begin{align}
\Prob\left[C_{n+1}-C_{n}=i\big|\mathcal{F}_{n}\right]={}&\Prob\left[P_{n}X_{n+1}+k\chi\left\{X_{n+1}=-1\right\}=i\big|\mathcal{F}_{n}\right]\nonumber\\
={}&\Prob\left[P_{n}=i,X_{n+1}=+1\big|\mathcal{F}_{n}\right]+\Prob\left[P_{n}=k-i,X_{n+1}=-1\big|\mathcal{F}_{n}\right]\nonumber\\
={}&g\left(\frac{i}{k}\right){k\choose i}q_{\beta}^{i}\left(\frac{A_{n}}{n},\frac{B_{n}}{n},\frac{C_{n}}{n}\right)\left\{1-q_{\beta}\left(\frac{A_{n}}{n},\frac{B_{n}}{n},\frac{C_{n}}{n}\right)\right\}^{k-i}\nonumber\\&+\left\{1-g\left(\frac{k-i}{k}\right)\right\}{k\choose k-i}q_{\beta}^{k-i}\left(\frac{A_{n}}{n},\frac{B_{n}}{n},\frac{C_{n}}{n}\right)\left\{1-q_{\beta}\left(\frac{A_{n}}{n},\frac{B_{n}}{n},\frac{C_{n}}{n}\right)\right\}^{i}.\nonumber
\end{align}
This yields
\begin{align}
\E\left[C_{n+1}-C_{n}\big|\mathcal{F}_{n}\right]={}&\sum_{i\in[k]}ig\left(\frac{i}{k}\right){k\choose i}q_{\beta}^{i}\left(\frac{A_{n}}{n},\frac{B_{n}}{n},\frac{C_{n}}{n}\right)\left\{1-q_{\beta}\left(\frac{A_{n}}{n},\frac{B_{n}}{n},\frac{C_{n}}{n}\right)\right\}^{k-i}+\sum_{j=0}^{k-1}(k-j)\left\{1-g\left(\frac{j}{k}\right)\right\}\nonumber\\&{k\choose j}q_{\beta}^{j}\left(\frac{A_{n}}{n},\frac{B_{n}}{n},\frac{C_{n}}{n}\right)\left\{1-q_{\beta}\left(\frac{A_{n}}{n},\frac{B_{n}}{n},\frac{C_{n}}{n}\right)\right\}^{k-j}=H_{3}\left(\frac{A_{n}}{n},\frac{B_{n}}{n},\frac{C_{n}}{n}\right),\label{C_{n+1}-C_{n}_cond_exp_fixed_size}
\end{align}
where $H_{3}$ is the function defined in \eqref{H_{3}_defn}.

We can now write the stochastic approximation process as follows:
\begin{align}
\begin{bmatrix}
(n+1)^{-1}A_{n+1}\\
(n+1)^{-1}B_{n+1}\\
(n+1)^{-1}C_{n+1}
\end{bmatrix}={}&
\begin{bmatrix}
n^{-1}A_{n}\\
n^{-1}B_{n}\\
n^{-1}C_{n}
\end{bmatrix}+\frac{1}{n+1}\left\{\begin{bmatrix}
\epsilon_{n+1,1}\\
\epsilon_{n+1,2}\\
\epsilon_{n+1,3}
\end{bmatrix}+h\left(\begin{bmatrix}
n^{-1}A_{n}\\
n^{-1}B_{n}\\
n^{-1}C_{n}
\end{bmatrix}\right)\right\},\label{sa_fixed}
\end{align}
where the step-sizes are given by the sequence $\{n^{-1}\}$, the drift function $h$ is as defined in \eqref{h_defn_fixed_size}, and the error terms $\{\epsilon_{n+1,i}\}$, for $i\in[3]$, are defined as
\begin{equation}\label{martingale_difference_noise_fixed_size}
\begin{cases}
\epsilon_{n+1,1}=A_{n+1}-A_{n}-H_{1}\left(n^{-1}A_{n},n^{-1}B_{n},n^{-1}C_{n}\right),\\
\epsilon_{n+1,2}=B_{n+1}-B_{n}-H_{2}\left(n^{-1}A_{n},n^{-1}B_{n},n^{-1}C_{n}\right),\\
\epsilon_{n+1,3}=C_{n+1}-C_{n}-H_{3}\left(n^{-1}A_{n},n^{-1}B_{n},n^{-1}C_{n}\right).
\end{cases}
\end{equation}
From \eqref{A_{n+1}-A_{n}_cond_exp_fixed_size}, \eqref{B_{n+1}-B_{n}_cond_exp_fixed_size} and \eqref{C_{n+1}-C_{n}_cond_exp_fixed_size}, it is evident that $\{\epsilon_{n+1,i}:n\geqslant N\}$ is a martingale difference sequence for each $i\in[3]$.

\begin{lemma}\label{lem:positively_invariant_domain}
The set $\mathcal{S}$ is positively invariant corresponding to the ODE in \eqref{ODE_fixed_sample_size}.
\end{lemma}
\begin{proof}
Recall, from the proof of Theorem~\ref{thm:main_1}, that the function $h$, defined in \eqref{h_defn_fixed_size}, is Lipschitz throughout $\mathcal{S}$. By Kirszbraun's Theorem (see \cite{kirszbraun1934zusammenziehende} and Theorem 1.2 of \cite{azagra2021kirszbraun}), there exists a function $\tilde{h}:\mathbb{R}^{3}\rightarrow\mathbb{R}^{3}$ such that $\tilde{h}|_{\mathcal{S}}\equiv h$ (i.e.\ $\tilde{h}$, restricted to $\mathcal{S}$, agrees with $h$, or, in other words, $\tilde{h}$ is an extension of $h$ to the entire $\mathbb{R}^{3}$) and $\tilde{h}$ is Lipschitz throughout $\mathbb{R}^{3}$ with the same Lipschitz constant as that of $h$ on $\mathcal{S}$. By Theorem B.1, Appendix B of \cite{borkar2008stochastic}, the initial value problem 
\begin{equation}\label{ODE_fixed_sample_size_extended}
(\dot{a}(t),\dot{b}(t),\dot{c}(t))=\tilde{h}(a(t),b(t),c(t)), \text{ with initial condition }(a(0),b(0),c(0))=(a,b,c)\in\mathbb{R}^{3},
\end{equation}
is well-posed for each $(a,b,c)\in\mathbb{R}^{3}$, i.e.\ there exists a unique solution, say $\varphi_{t}(a,b,c)$, to this initial value problem for all times $t\geqslant 0$.

For any subset $S$ of $\mathbb{R}^{d}$, for $d\in\mathbb{N}$, and any $\mathbf{z}\in\mathbb{R}^{d}$, we define $\dist(\mathbf{z},S)=\inf\{||\mathbf{z}-\mathbf{w}||:\mathbf{w}\in S\}$. For a closed subset $S$ of $\mathbb{R}^{d}$, and $\mathbf{x}\in S$, the \emph{tangent cone} to $S$ at $\mathbf{x}$ (see Definition 4.6 of \cite{blanchini2008set}, or \cite{bouligand1932introduction}) is defined as 
\begin{equation}
\mathcal{T}_{S}(\mathbf{x})=\{\mathbf{y}\in\mathbb{R}^{d}:\liminf_{t\rightarrow 0+}t^{-1}\dist\left(\mathbf{x}+t\mathbf{y},S\right)=0\}.\nonumber
\end{equation}
From \eqref{domain_defn}, it is immediate that the set $\mathcal{S}$ is compact as well as convex. From the argument following \eqref{H_{1},H_{2},H_{3}_concise_expressions}, the function $H=(H_{1},H_{2},H_{3})$ maps $\mathcal{S}$ to itself, so that $H(a,b,c)\in\mathcal{S}$ for any $(a,b,c)\in\mathcal{S}$. Any convex combination of $H(a,b,c)$ and $(a,b,c)$, for each $(a,b,c)\in\mathcal{S}$, must, therefore, be in $\mathcal{S}$ again. Thus, for each $(a,b,c)\in\mathcal{S}$, and for any $t\in(0,1]$, using \eqref{h_defn_fixed_size}, we have:
\begin{align}
\begin{bmatrix}
a\\
b\\
c
\end{bmatrix}+t h\left(\begin{bmatrix}
a\\
b\\
c
\end{bmatrix}\right)=(1-t)\begin{bmatrix}
a\\
b\\
c
\end{bmatrix}+t H\left(\begin{bmatrix}
a\\
b\\
c
\end{bmatrix}\right)\in\mathcal{S} \implies h\left(\begin{bmatrix}
a\\
b\\
c
\end{bmatrix}\right)\in\mathcal{T}_{\mathcal{S}}\left(\begin{bmatrix}
a\\
b\\
c
\end{bmatrix}\right).\nonumber
\end{align}
 
We now conclude, via Nagumo's Theorem (see Corollary 4.8 of \cite{blanchini2008set}), that $\mathcal{S}$ is a positively invariant set corresponding to \eqref{ODE_fixed_sample_size_extended}, and consequently, to \eqref{ODE_fixed_sample_size}.
\end{proof}

\begin{proof}[Proof of Theorem~\ref{thm:main_1}]
Since $\beta>0$, the function $q_{\beta}$, defined in \eqref{q_defn}, has continuous partial derivatives throughout the compact set $\mathcal{S}$ defined in \eqref{domain_defn}. Each of $H_{1}(a,b,c)$, $H_{2}(a,b,c)$ and $H_{3}(a,b,c)$, defined in \eqref{H_{1}_defn}, \eqref{H_{2}_defn} and \eqref{H_{3}_defn}, is a polynomial in $q_{\beta}(a,b,c)$. Combining these observations, we conclude that each of $H_{1}$, $H_{2}$ and $H_{3}$ has continuous partial derivatives throughout the compact set $\mathcal{S}$, and hence, the function $h$, defined in \eqref{h_defn_fixed_size}, is Lipschitz on $\mathcal{S}$, thereby ensuring Assumption~\eqref{borkar_1} for \eqref{sa_fixed}.

The step-sizes are given by $a_{n}=n^{-1}$ in \eqref{sa_fixed}, so that, right away, Assumption~\eqref{borkar_2} holds.

Here, $\Delta M_{n+1}^{T}=(\epsilon_{n+1,1},\epsilon_{n+1,2},\epsilon_{n+1,3})$, defined via \eqref{martingale_difference_noise_fixed_size}, while $\delta_{n}^{T}=(0,0,0)$, for all $n\geqslant N$. By the description of the model in \S\ref{subsec:model_fixed_sample_size}, it is evident that $A_{n+1}-A_{n}\in\{0,1\}$ while each of $B_{n+1}-B_{n}$ and $C_{n+1}-C_{n}$ takes values in $\{0,1,\ldots,k\}$. Therefore, $\E[\epsilon_{n+1,1}^{2}|\mathcal{F}_{n}]\leqslant 1$ and $\E[\epsilon_{n+1,i}^{2}|\mathcal{F}_{n}]\leqslant k^{2}$ for each $i\in\{2,3\}$, so that $\E[||\Delta M_{n+1}||^{2}|\mathcal{F}_{n}]\leqslant 2k^{2}+1$, making sure that the final criterion of Corollary~\ref{cor:borkar_a.s.} is satisfied (with $K=2k^{2}+1$).   The conclusion of Theorem~\ref{thm:main_1} now follows from Corollary~\ref{cor:borkar_a.s.}.
\end{proof}

\begin{proof}[Proof of Theorem~\ref{thm:main_2}]
Lemma~\ref{lem:positively_invariant_domain}, along with the first paragraph of its proof, ensures that the initial value problem
\begin{equation}
\left(\dot{a}(t),\dot{b}(t),\dot{c}(t)\right)=h\left(a(t),b(t),c(t)\right),\text{ with }(a(0),b(0),c(0))=(a,b,c),\label{general_IVP_fixed_sample_size}
\end{equation}
is well-posed for \emph{any} $(a,b,c)\in\mathcal{S}$, and its unique solution, $(a(t),b(t),c(t))=\varphi_{t}(a,b,c)$, lies in $\mathcal{S}$ for all $t\geqslant 0$. Note that \eqref{general_IVP_fixed_sample_size} reduces to \eqref{ODE_fixed_sample_size} when we take $(a,b,c)=(A_{N}/N,0,0)$ (which belongs to $\mathcal{S}$). 

From \eqref{H_{1},H_{2},H_{3}_concise_expressions}, \eqref{h_defn_fixed_size} and \eqref{F_{1},F_{2}_defns}, the ODE in \eqref{general_IVP_fixed_sample_size} can be rewritten as:
\begin{equation}\label{ODE_fixed_sample_size_concise}
\begin{cases}
&\dot{a}(t)=F_{1}\left(q_{\beta}(a(t),b(t),c(t))\right)-a(t), \quad \dot{b}(t)=F_{2}\left(q_{\beta}(a(t),b(t),c(t))\right)-b(t)\\
&\dot{c}(t)=2F_{2}\left(q_{\beta}(a(t),b(t),c(t))\right)-kF_{1}\left(q_{\beta}(a(t),b(t),c(t))\right)+k\left\{1-q_{\beta}(a(t),b(t),c(t))\right\}-c(t).
\end{cases}
\end{equation}
\begin{enumerate*}
\item Substituting in \eqref{q_defn} the unique solution to \eqref{ODE_fixed_sample_size_concise}, i.e.\ $(a(t),b(t),c(t))=\varphi_{t}(a,b,c)$ for all $t\geqslant 0$, 
\item differentiating both sides of the identity $q_{\beta}(a(t),b(t),c(t))\{c(t)+\beta\}=b(t)+\beta a(t)$, for all $t\geqslant 0$, with respect to $t$, 
\item and using the definition of the function $G$ from \eqref{G_defn},
\end{enumerate*}
we obtain:
\begin{align}
\dot{q}_{\beta}(a(t),b(t),c(t))=\left\{c(t)+\beta\right\}^{-1}G\left(q_{\beta}(a(t),b(t),c(t))\right).\label{q_derivative_1}
\end{align}
Recalling that $q^{*}$ is the unique root of $G$ in $(0,1)$ (in fact, in $[0,1]$), we define the function $V:\mathcal{S}\rightarrow\mathbb{R}$ as
\begin{equation}\label{V_defn}
V(a,b,c)=-\int_{q^{*}}^{q_{\beta}(a,b,c)}G(s)ds, \text{ for each }(a,b,c)\in\mathcal{S}.
\end{equation}
In the next two paragraphs, we make a couple of observations regarding the function $V$. 

Recall, from the discussion immediately preceding the statement of Theorem~\ref{thm:main_2}, that $G(0)>0$ and $G(1)<0$ when $g(0)>0$ and $g(1)<1$. This, along with the assumption that $G$ has a unique root, $q^{*}$, in $(0,1)$, ensures that $G(s)>0$ for all $s\in[0,q^{*})$ and $G(s)<0$ for all $s\in(q^{*},1]$. Now, we consider any $(a,b,c)\in\mathcal{S}$ such that $q_{\beta}(a,b,c)\neq q^{*}$. We have
\begin{equation}\label{V_nonnegative}
V(a,b,c)
\begin{cases} 
=-\int_{q^{*}}^{q_{\beta}(a,b,c)}G(s)ds=\int_{q_{\beta}(a,b,c)}^{q^{*}}G(s)ds>0 &\text{when }q_{\beta}(a,b,c)<q^{*},\\
=-\int_{q^{*}}^{q_{\beta}(a,b,c)}G(s)ds=\int_{q^{*}}^{q_{\beta}(a,b,c)}\{-G(s)\}ds>0 &\text{when } q_{\beta}(a,b,c)>q^{*},
\end{cases}
\end{equation}
thus ensuring that $V(a,b,c)\geqslant 0$ for all $(a,b,c)\in\mathcal{S}$, with equality if and only if $q_{\beta}(a,b,c)=q^{*}$, or, equivalently, $(a,b,c)\in\Lambda$, with $\Lambda$ as defined in the statement of Theorem~\ref{thm:main_2}.   

For $(a(t),b(t),c(t))=\varphi_{t}(a,b,c)$ satisfying \eqref{general_IVP_fixed_sample_size} for all $t\geqslant 0$, we have, applying the Leibniz rule of differentiation of an integral and using \eqref{q_derivative_1}:
\begin{align}\label{V_derivative}
\dot{V}(a(t),b(t),c(t))={}&-\dot{q}_{\beta}(a(t),b(t),c(t))G\left(q_{\beta}(a(t),b(t),c(t))\right)=-\frac{G^{2}\left(q_{\beta}(a(t),b(t),c(t))\right)}{c(t)+\beta}.
\end{align}
Recall that $(a(t),b(t),c(t))=\varphi_{t}(a,b,c)\in\mathcal{S}$ for all $t\geqslant 0$ provided $(a,b,c)\in\mathcal{S}$, by Lemma~\ref{lem:positively_invariant_domain}, so that $c(t)\geqslant 0\implies c(t)+\beta\geqslant\beta>0$ for all $t\geqslant 0$. This ensures that the denominator of \eqref{V_derivative} is strictly positive for all $t\geqslant 0$ provided $(a,b,c)\in\mathcal{S}$. This observation, along with the assumption that $G$ has a unique root, $q^{*}$, in $[0,1]$, yields: 
\begin{equation}\label{V_decreasing}
V\left(\varphi_{t}(a,b,c)\right)
\begin{cases}
\text{is monotonically decreasing in }t, \text{ for all }t\geqslant 0 \text{ and all }(a,b,c)\in\mathcal{S},\\
\text{is strictly decreasing in }t \text{ whenever }q_{\beta}\left(\varphi_{t}(a,b,c)\right)\neq q^{*} (\text{equivalently, }\varphi_{t}(a,b,c)\notin\Lambda).
\end{cases}
\end{equation}
Moreover, the continuity of the functions $G$ on $[0,1]$ and $q_{\beta}$ on $\mathcal{S}$, the continuity of $(a(t),b(t),c(t))=\varphi_{t}(a,b,c)$ as a function of $t$ for $t\geqslant 0$, and the fact (as observed above) that the denominator of \eqref{V_derivative} is bounded away from $0$ uniformly for all $t\geqslant 0$ provided $(a,b,c)\in\mathcal{S}$, together ensure that $\dot{V}(a(t),b(t),c(t))=\dot{V}(\varphi_{t}(a,b,c))$ is continuous in $t$ for all $t\geqslant 0$, for each $(a,b,c)\in\mathcal{S}$.

Let $L$ be \emph{any} compact, connected, internally chain transitive invariant set, contained in $\mathcal{S}$, corresponding to the ODE in \eqref{general_IVP_fixed_sample_size} (recall that \eqref{ODE_fixed_sample_size} is a special case of \eqref{general_IVP_fixed_sample_size}). Let $\theta=\inf\{V(a,b,c):(a,b,c)\in L\}$. The function $V$, as is evident from its definition in \eqref{V_defn} and the continuity of $q_{\beta}$ on $\mathcal{S}$, is continuous throughout $\mathcal{S}$, and as $L$ is compact, the minimum of $V$ on $L$ must be attained. Let $(a_{1},b_{1},c_{1})$ be \emph{any} element in $L$ such that $V(a_{1},b_{1},c_{1})=\theta$. Our goal, to begin with, is to show that $(a_{1},b_{1},c_{1})\in\Lambda$.

If possible, let $(a_{1},b_{1},c_{1})\notin\Lambda$, so that $q_{\beta}(\varphi_{0}(a_{1},b_{1},c_{1}))=q_{\beta}(a_{1},b_{1},c_{1})\neq q^{*}$. From the second assertion stated in \eqref{V_decreasing} and the observation immediately following \eqref{V_decreasing}, we conclude the existence of some $\epsilon_{1}>0$ such that $V(\varphi_{t}(a_{1},b_{1},c_{1}))$ is strictly decreasing in $t$ for all $t\in[0,\epsilon_{1}]$, i.e.\
\begin{equation}
V(\varphi_{t}(a_{1},b_{1},c_{1}))<V(\varphi_{0}(a_{1},b_{1},c_{1}))=V(a_{1},b_{1},c_{1})=\theta\text{ for all }t\in(0,\epsilon_{1}].\label{V_strictly_decreasing}
\end{equation}
However, $\varphi_{t}(a_{1},b_{1},c_{1})\in L$ for each $t>0$, since $(a_{1},b_{1},c_{1})\in L$ and $L$ is an invariant set corresponding to \eqref{general_IVP_fixed_sample_size}. Consequently, \eqref{V_strictly_decreasing} contradicts the minimality of $\theta$ as defined above, leading to the conclusion that
\begin{equation}
(a_{1},b_{1},c_{1})\in L \text{ and } V(a_{1},b_{1},c_{1})=\theta \implies (a_{1},b_{1},c_{1})\in\Lambda.\label{conclusion:in_Lambda}
\end{equation}

This conclusion, which is equivalent to $q_{\beta}(a_{1},b_{1},c_{1})=q^{*}$, further implies, by \eqref{V_defn}, that $\theta=V(a_{1},b_{1},c_{1})=0$. Our goal, now, is to show that $V$ is constant on the set $L$ (and hence, the only value it takes on $L$ is $\theta=0$). To this end, for \emph{any} sequence $\{\theta_{n}:n\in\mathbb{N}\}$ with $\theta_{n}\downarrow 0$ as $n\rightarrow\infty$, let us define the sequence of subsets
\begin{equation}
L_{n}=\{(a,b,c)\in L:V(a,b,c)<\theta_{n}\}.\nonumber
\end{equation}
Since the function $V$ is continuous, each $L_{n}$ is open relative to $L$. We let $\overline{L}_{n}$ indicate the closure of $L_{n}$.  Since $L_{n}\subset L$ and $L$ is compact, we have $\overline{L}_{n}\subset L$ and $\overline{L}_{n}$ is compact as well. Fixing any $T>0$, we now show that $\varphi_{T}(\overline{L}_{n})\subset L_{n}$ for each $n\in\mathbb{N}$. To show this, consider any $(a,b,c)\in\overline{L}_{n}$:
\begin{enumerate}
\item If $(a,b,c)\in\Lambda$, then $q_{\beta}(a,b,c)=q^{*}$, implying $V(a,b,c)=0$. By \eqref{V_defn}, \eqref{V_nonnegative} and the first assertion stated in \eqref{V_decreasing}, we have $0\leqslant V(\varphi_{T}(a,b,c))\leqslant V(\varphi_{0}(a,b,c))=V(a,b,c)=0$, so that $V(\varphi_{T}(a,b,c))=0<\theta_{n} \implies \varphi_{T}(a,b,c)\in L_{n}$. 
\item If $(a,b,c)\notin\Lambda$, then $q_{\beta}(a,b,c)\neq q^{*}$. The second assertion stated in \eqref{V_decreasing}, along with the remark immediately following \eqref{V_decreasing}, ensures the existence of some $\epsilon>0$ such that $V(\varphi_{t}(a,b,c))$ is strictly decreasing in $t$ for all $t\in[0,\epsilon]$. Since $(a,b,c)\in\overline{L}_{n}$, we have $V(a,b,c)\leqslant \theta_{n}$. Combining these two observations, along with the first assertion stated in \eqref{V_decreasing}, we conclude that 
\[
\begin{cases}
V\left(\varphi_{\epsilon}(a,b,c)\right)\leqslant V\left(\varphi_{T}(a,b,c)\right)<V\left(\varphi_{0}(a,b,c)\right)=V(a,b,c)\leqslant \theta_{n} &\text{if }T\in(0,\epsilon],\\
V\left(\varphi_{T}(a,b,c)\right)\leqslant V\left(\varphi_{\epsilon}(a,b,c)\right)<V\left(\varphi_{0}(a,b,c)\right)=V(a,b,c)\leqslant \theta_{n} &\text{if }T\in(\epsilon,\infty),
\end{cases}
\]
so that in either scenario, $V\left(\varphi_{T}(a,b,c)\right)<\theta_{n} \implies \varphi_{T}(a,b,c)\in L_{n}$.
\end{enumerate}  
This completes the proof of our claim that $\varphi_{T}(\overline{L}_{n})\subset L_{n}$ for all $T>0$, for each $n\in\mathbb{N}$. By Proposition 3.19 of \cite{benaim2005stochastic}, or by Lemma 5.2 of \cite{benaim2006dynamics}, we conclude that $L_{n}$ must be a fundamental neighbourhood of some attractor $A$. However, $L_{n}\subset L$, and by Proposition 3.20 of \cite{benaim2005stochastic}, or Proposition 5.3 of \cite{benaim2006dynamics}, we know that $L$, being internally chain transitive, can contain no attractor as a proper subset. Therefore, we must have $L=L_{n}$ for each $n\in\mathbb{N}$, or, in other words,
\begin{equation}
L=\bigcap_{n\in\mathbb{N}}L_{n}=\bigcap_{n\in\mathbb{N}}\{(a,b,c)\in L:V(a,b,c)<\theta_{n}\}=\{(a,b,c)\in L:V(a,b,c)=0\}.\nonumber
\end{equation}
This, along with \eqref{conclusion:in_Lambda}, yields: $(a,b,c)\in L\implies V(a,b,c)=\theta=0 \implies (a,b,c)\in\Lambda$, proving $L\subset\Lambda$.
\end{proof}

\begin{proof}[Proof of Proposition~\ref{prop:examples_g_unique_root_of_G}]
Recalling $G$ from \eqref{G_defn}, we first compute its derivative when $f(x)=c$ for all $x\in[0,1]$, with $c\in(0,1)$. In this case, $g(x)=d$, where $d=pf(x)+(1-p)\{1-f(x)\}=(2p-1)c+(1-p)$, for all $x\in[0,1]$. This yields $G(q)=(1-2q)dkq+(\beta+kq)d-kq+kq^{2}-\beta q$, so that
\begin{align}
G'(q)
={}&k(2d-1)(1-2q)-\beta \implies G''(q)=-2k(2d-1).\nonumber
\end{align}
Right away, we see that the sign of $G''(q)$ remains unchanged throughout $[0,1]$, so that $G$ is either strictly convex (when $d<1/2$), or strictly concave (when $d>1/2$), or a straight line (when $d=1/2$). In the first couple of cases, the curve $y=G(x)$ can intersect the $x$-axis at most once when $x$ lies in $[0,1]$, whereas in the third case, it can intersect the $x$-axis at most once. As argued right after \eqref{G_defn}, we have $G(0)>0$ and $G(1)<0$ assuming $g(0)>0$ and $g(1)<1$, both of which are true in this case since $\min\{p,1-p\}<d<\max\{p,1-p\}$. Combining these observations, we conclude that, in fact, for each possible value of $d$, the curve $y=G(x)$ intersects the $x$-axis precisely once in the interval $[0,1]$ (in fact, the intersection happens in $(0,1)$), leading to a unique root for $G$ in $[0,1]$.

For $f(x)=x$ for all $x\in[0,1]$, we have $g(x)=(2p-1)x+(1-p)$ for all $x\in[0,1]$. From \eqref{F_{1},F_{2}_defns}, we obtain:
\begin{align}
F_{1}(q)
=1-p-q+2pq,\quad
F_{2}(q)
=(2p-1)q(1-q)+kq(1-p-q+2pq).\nonumber
\end{align}
This yields, from \eqref{G_defn}:
\begin{align}
{}&G(q)
=2q^{3}(k-1)(1-2p)-3q^{2}(k-1)(1-2p)+q(k-1)(1-2p)-2q\beta(1-p)+\beta(1-p)\nonumber\\
\implies{}&G'(q)
=(6q^{2}-6q+1)(k-1)(1-2p)-2\beta(1-p).\label{G'_f(x)=x}
\end{align}
The polynomial $P(q)=6q^{2}-6q+1$ is the only expression via which \eqref{G'_f(x)=x} depends on $q$. Note that $P(q)=-6q(1-q)+1$ is symmetric around $q=1/2$, and $P'(q)=-6(1-2q)$, so that $P(q)$ is strictly decreasing for $q\in[0,1/2)$ and strictly increasing for $q\in(1/2,1]$. Therefore, we have the following scenarios:
\begin{enumerate}
\item When $p\in[0,1/2)$, the function $G'(q)$ is strictly decreasing for $q\in[0,1/2)$ and strictly increasing for $q\in(1/2,1]$, so that its maximum is attained at $q=0$ as well as $q=1$, and its minimum is attained at $q=1/2$. Evidently, if $G'(0)\leqslant 0$, then $G'(q)<0$ for all $q\in(0,1)$, implying that $G$ is strictly decreasing, and hence, the curve $y=G(x)$ can intersect the $x$-axis at most once in the interval $[0,1]$. 

Suppose $G'(0)>0$. We note that $G'(1/2)=-(k-1)(1-2p)/2-2\beta(1-p)<0$, and the symmetry of $P(q)$ around $q=1/2$ implies that $G'(q)$ is symmetric around $1/2$ as well, so that, if $\gamma_{1}$ and $\gamma_{2}$, with $\gamma_{1}<1/2<\gamma_{2}$, are the two roots of $G'(q)$, then $\gamma_{1}+\gamma_{2}=1$ and $G'(q)$ is strictly positive for $q\in[0,\gamma_{1})$, strictly negative for $q\in(\gamma_{1},\gamma_{2})$, and strictly positive for $q\in(\gamma_{2},1]$. This ensures that $G(q)$ is strictly increasing for $q\in[0,\gamma_{1})$, strictly decreasing for $q\in(\gamma_{1},\gamma_{2})$, and strictly increasing for $q\in(\gamma_{2},1]$. But since we already have $G(0)>0$ and $G(1)<0$ (by the discussion following \eqref{G_defn}), the curve $y=G(x)$ must intersect the $x$-axis only once at some point in $(\gamma_{1},\gamma_{2})$, and neither in $[0,\gamma_{1}]$ nor in $[\gamma_{2},1]$. Thus, when $p\in[0,1/2)$, the function $G$ \emph{always} has a unique root in $[0,1]$.

\item When $p\in(1/2,1]$, the function $G'(q)$ is strictly increasing for $q\in[0,1/2)$ and strictly decreasing for $q\in(1/2,1]$, so that its minimum is attained at $q=0$ as well as $q=1$, and its maximum is attained at $q=1/2$. The maximum value is $G'(1/2)=(2p-1)(k-1)/2-2\beta(1-p)$, so that if \eqref{beta_p_cond} is true, we have $G'(q)\leqslant 0$ for all $q\in[0,1]$, with equality iff $q=1/2$, forcing $G$ to be strictly decreasing and therefore, have only one root in $[0,1]$.

On the other hand, if \eqref{beta_p_cond} fails to hold, then $G'(1/2)>0$. Noting that $G'(0)=-(k-1)(2p-1)-2\beta(1-p)<0$, and using the symmetry of $G'(q)$ around $q=1/2$, we conclude that $G'(q)$ is strictly negative for $q\in[0,\gamma_{1})$, strictly positive in $(\gamma_{1},\gamma_{2})$, and strictly negative again in $(\gamma_{2},1]$, where $\gamma_{1}$ and $\gamma_{2}$ are the roots of $G'$ in $[0,1]$, equidistant from $1/2$, as also mentioned earlier. This means that $G$ is strictly decreasing for $q\in[0,\gamma_{1})$, strictly increasing for $q\in(\gamma_{1},\gamma_{2})$, and strictly decreasing for $q\in(\gamma_{2},1]$. We consider two possible scenarios:
\begin{enumerate}
\item If $G(\gamma_{1})>0$, then it is evident that the curve $y=G(x)$ lies \emph{strictly above} the $x$-axis for $q\in[0,\gamma_{2}]$. Therefore, $y=G(x)$ must intersect the $x$-axis in $(\gamma_{2},1]$, and it can do so at most once due to its strictly decreasing nature in this domain of values of $q$. This scenario happens whenever \eqref{beta_{1}_cond} is true.
\item If $G(\gamma_{2})<0$, then it is evident that the curve $y=G(x)$ lies \emph{strictly beneath} the $x$-axis for $q\in[\gamma_{1},1]$. Therefore, $y=G(x)$ must intersect the $x$-axis in $[0,\gamma_{1})$, and it can do so at most once due to its strictly decreasing nature in this domain of values of $q$. This scenario happens whenever \eqref{beta_{2}_cond} is true.
\end{enumerate}
\end{enumerate}

Let us now consider $f(x)=e^{x-1}$ for all $x\in[0,1]$, so that $g(x)=(2p-1)e^{x-1}+(1-p)$ for each $x\in[0,1]$. From \eqref{F_{1},F_{2}_defns}, we obtain:
\begin{align}
F_{1}(q)
=(2p-1)e^{-1}(1-q+q e^{1/k})^{k}+(1-p),\
F_{2}(q)
=k(2p-1)q e^{1/k-1}(1-q+q e^{1/k})^{k-1}+(1-p)kq,\nonumber
\end{align}
so that, setting $A(q)=(1-q+q e^{1/k})$, we have 
\begin{align}
G(q)={}&k(1-2q)(2p-1)q e^{1/k-1}A(q)^{k-1}+k(1-2q)q(1-p)+(\beta+kq)(2p-1)e^{-1}A(q)^{k}\nonumber\\&+(\beta+kq)(1-p)-kq(1-q)-\beta q,\label{G(q)_f_exponential}
\end{align}
leading to $G'(q)=B(q)+C(q)+D(q)$, where
\begin{align}
B(q)={}&k(2p-1)e^{-1}A(q)^{k-2}\left\{e^{1/k}(1-4q)A(q)-e^{1/k}\left(e^{1/k}-1\right)q(1-2q)+A(q)^{2}+\left(e^{1/k}-1\right)\beta A(q)\right\}\nonumber\\
C(q)={}&k^{2}(2p-1)e^{-1}\left(e^{1/k}-1\right)q A(q)^{k-2}\left\{e^{1/k}(1-2q)+A(q)\right\},\quad D(q)=k(1-2q)(1-2p)-\beta.\nonumber
\end{align}
For $p\in(0,1/2)$, using the inequalities $2p-1<0$ and $e^{1/k}-1>1/k$, we obtain: $C(q)<k(2p-1)e^{-1}q A(q)^{k-2}\{e^{1/k}(1-2q)+A(q)\}$, so that combining $B(q)$ with this upper bound and adding $D(q)$ yields
\begin{align}
{}&G'(q)<
k(2p-1)\{e^{-1}(1+e^{1/k})A(q)^{k-2}(1-q-q^{2}e^{1/k})+(2q-1)\}-\beta\{k(1-2p)e^{-1}(e^{1/k}-1)A(q)^{k-1}+1\}\nonumber\\
{}&<k(2p-1)\{e^{-1}(1+e^{1/k})A(q)^{k-2}(1-q-q^{2}e^{1/k})+(2q-1)\}-\beta\{(1-2p)e^{-1}A(q)^{k-1}+1\}\nonumber\\
{}&<
k(2p-1)\{e^{-1}(1+e^{1/k})A(q)^{k-2}(1-q-q^{2}e^{1/k})+(1-2p)e^{-1}A(q)^{k-1}+2q\},\label{G'(q)_upper_bound_p<1/2}
\end{align}
where the final step follows from our assumption that $\beta>k(1-2p)$. Note that the only root of the quadratic polynomial $(1-q-q^{2}e^{1/k})$ that lies in $[0,1]$ (in fact, in $(0,1)$) is $\gamma=e^{-1/k}(\sqrt{1+4e^{1/k}}-1)/2$, and it is evident that the final expression in \eqref{G'(q)_upper_bound_p<1/2} is strictly negative for all $q\in[0,\gamma]$. In order to deal with $q\in(\gamma,1]$, we note that $A(q)$ is strictly increasing for $q\in[0,1]$, so that $1\leqslant A(q)\leqslant e^{1/k}$, so that, continuing with the inequality in \eqref{G'(q)_upper_bound_p<1/2}, we can write, for $q\in(\gamma,1]$:
\begin{align}
G'(q)<k(2p-1)\{(1+e^{1/k})e^{-2/k}(1-q-q^{2}e^{1/k})+(1-2p)e^{-1}+2q\}.\label{G'(q)_second_upper_bound_p<1/2_q>gamma}
\end{align}
Differentiating, twice, the expression in \eqref{G'(q)_second_upper_bound_p<1/2_q>gamma} (ignoring the constant factor of $k(2p-1)$), we obtain $-2(1+e^{1/k})e^{-1/k}$, showing us that the final upper bound in \eqref{G'(q)_second_upper_bound_p<1/2_q>gamma} is strictly convex throughout $[0,1]$. Moreover, the value of this upper bound at $q=\gamma$ equals $k(2p-1)\{(1-2p)e^{-1}+2\gamma\}<0$, and its value at $q=1$ equals $k(2p-1)\{-(1+e^{1/k})e^{-1/k}+(1-2p)e^{-1}+2\}=k(2p-1)\{1-e^{-1/k}+(1-2p)e^{-1}\}<0$. Thus, we conclude that the final upper bound in \eqref{G'(q)_second_upper_bound_p<1/2_q>gamma} is strictly negative for all $q\in(\gamma,1]$. This proves that when $k\geqslant 2$, $p\in(0,1/2)$ and $\beta>k(1-2p)$, the function $G$ is strictly decreasing in $[0,1]$, and hence may have at most one root in $[0,1]$. From \eqref{G(q)_f_exponential}, we have $G(0)=\beta p e^{-1}+\beta(1-p)(1-e^{-1})>0$ and $G(1)=-\beta(1-p)<0$, so that $G$ must have at least one root in $(0,1)$. Combining these two observations, the final conclusion follows.
\end{proof}
\begin{figure}[h!]
  \centering
    \includegraphics[width=0.7\textwidth]{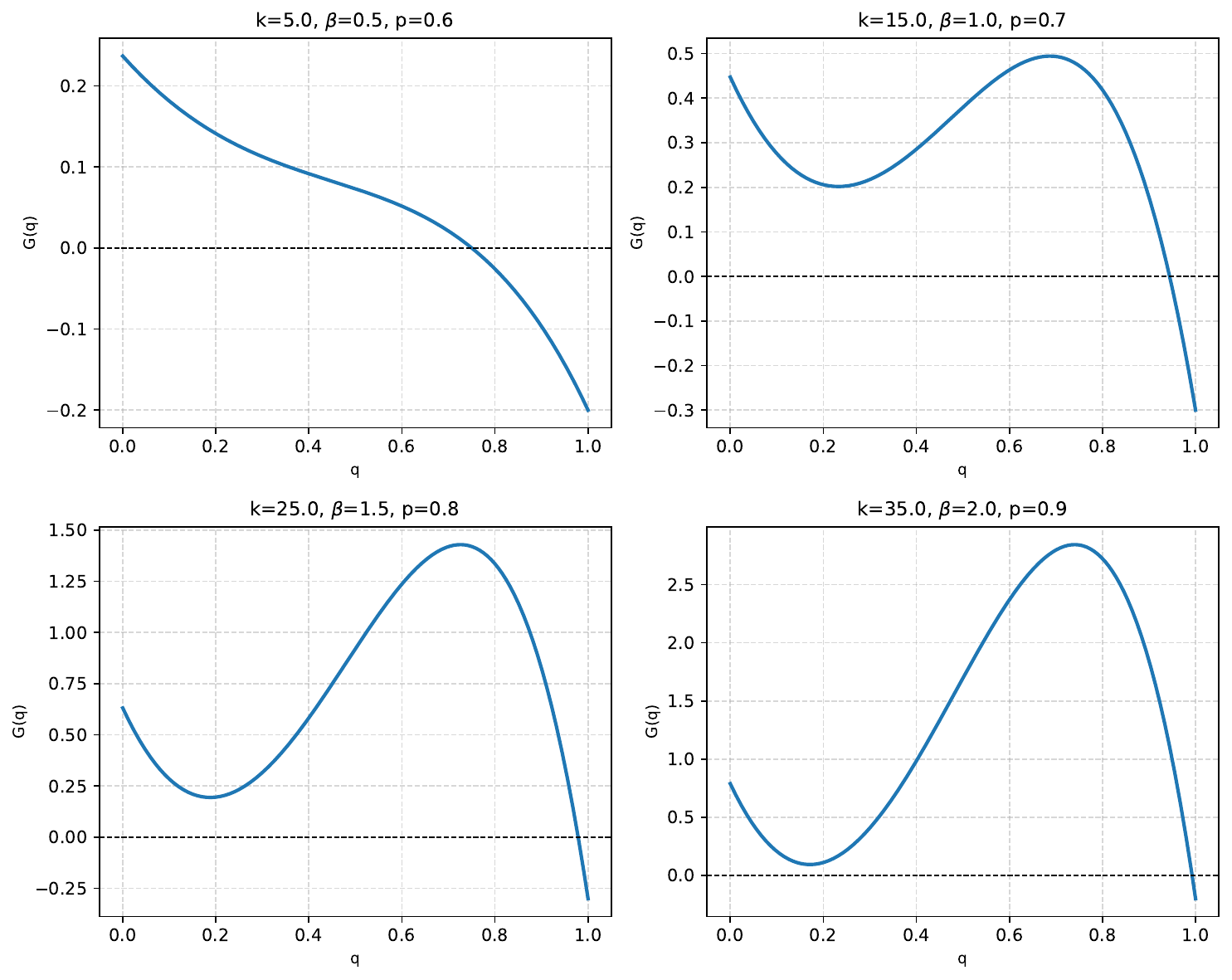}
\caption{Examples of values of the triple $(k,\beta,p)$, where $p\in(1/2,1)$, for which $G$ has a unique root in $[0,1]$ when $f(x)=e^{x-1}$ for all $x\in[0,1]$.}
  \label{fig}
\end{figure}

\begin{proof}[Proof of Theorem~\ref{thm:main_6}]
The relevant stochastic approximation process to be considered here is \eqref{sa_fixed}, with the drift function $h$ as defined in \eqref{h_defn_fixed_size}. By Lemma~\ref{lem:positively_invariant_domain}, we know that the compact set $\mathcal{S}$, as defined in \eqref{domain_defn} and on which all of our functions have been defined, is positively invariant corresponding to the ODE in \eqref{ODE_fixed_sample_size}. Therefore, as long as we ensure that the initial value that our stochastic process, as well as the trajectory corresponding to the boundary value problem governed by \eqref{ODE_fixed_sample_size}, begins from, lies inside $\mathcal{S}$, we shall forever remain confined to $\mathcal{S}$. It, therefore, suffices to ensure that, of the assumptions mentioned in the statement of Theorem~\ref{thm:borkar_fclt}, the ones concerning the drift function $h$ and the covariance matrix function $Q$ are satisfied within $\mathcal{S}$.

Right away, we see that, 
\begin{enumerate*}
\item since each of $H_{1}$, $H_{2}$ and $H_{3}$, defined in \eqref{H_{1}_defn}, \eqref{H_{2}_defn} and \eqref{H_{3}_defn}, is a polynomial in $q_{\beta}(\cdot,\cdot,\cdot)$, and 
\item $q_{\beta}$, defined via \eqref{q_defn}, is in $\mathcal{C}^{(\infty)}(\mathcal{S})$,
\end{enumerate*}
each of $H_{1}$, $H_{2}$ and $H_{3}$, and consequently, the function $h$ itself, is in $\mathcal{C}^{(\infty)}(\mathcal{S})$. Consequently, the Jacobian matrix, $J_{h}$, of $h$ is also in $\mathcal{C}^{(\infty)}(\mathcal{S})$. These observations, along with the fact that $\mathcal{S}$ is compact, validate Assumption~\eqref{borkar_5}. Comparing \eqref{sa_fixed} with \eqref{sa_borkar_general}, we have $a_{n+1}=(n+1)^{-1}$, so that $a_{n+1}^{-1}-a_{n}^{-1}=1$ for each $n\geqslant N$, ensuring that the limit $\alpha$, mentioned in Assumption~\eqref{borkar_6}, exists and equals $1$. Since, in our set-up, $Z_{n}=(n^{-1}A_{n},n^{-1}B_{n},n^{-1}C_{n})\in\mathcal{S}$ and $\mathcal{S}$ is a compact set, Assumption~\eqref{borkar_7} is satisfied right away.

To verify Assumption~\eqref{borkar_8}, we consider $\Delta M_{n+1}=(\epsilon_{n+1,1},\epsilon_{n+1,2},\epsilon_{n+1,3})^{T}$, where the coordinates of $\Delta M_{n+1}$ are as defined in \eqref{martingale_difference_noise_fixed_size}. The definition of $\mathcal{S}$ ensures that $\Delta M_{n+1}\in[-1,1]\times[-k,k]\times[-k,k]$ (for instance, $A_{n+1}-A_{n}$ lies in $\{0,1\}$, so that $\E[A_{n+1}-A_{n}|\mathcal{F}_{n}]\in[0,1]$, implying that $\epsilon_{n+1,1}\in[-1,1]$), so that $||\Delta M_{n+1}||^{4}$ is bounded above by a constant almost surely. Right away, this ensures that the inequality in Assumption~\eqref{borkar_8} is satisfied.

All we have now left to do is compute $\E[\Delta M_{n+1}\Delta M_{n+1}^{T}|\mathcal{F}_{n}]$, for which we are required to compute $\E[\epsilon_{n+1,i}\epsilon_{n+1,j}|\mathcal{F}_{n}]$ for each $i,j\in[3]$. Note that, for computing each of $\E[(A_{n+1}-A_{n})(B_{n+1}-B_{n})|\mathcal{F}_{n}]$, $\E[(A_{n+1}-A_{n})(C_{n+1}-C_{n})|\mathcal{F}_{n}]$ and $\E[(B_{n+1}-B_{n})(C_{n+1}-C_{n})|\mathcal{F}_{n}]$, we need only care about the event $\{X_{n+1}=+1\}$, since for the event $\{X_{n+1}=-1\}$, each of the random variables $(A_{n+1}-A_{n})$ and $(B_{n+1}-B_{n})$ equals $0$. Moreover, on the event $\{X_{n+1}=+1\}$, we have $(B_{n+1}-B_{n})=(C_{n+1}-C_{n})$ since the sum of the in-degrees of all the vertices with opinion $-1$ does not change. Thus,
\begin{align}
{}&\E[(A_{n+1}-A_{n})(B_{n+1}-B_{n})|\mathcal{F}_{n}]=\E[(A_{n+1}-A_{n})(C_{n+1}-C_{n})|\mathcal{F}_{n}]=\sum_{i\in[k]}i\Prob[X_{n+1}=+1,P_{n}=i|\mathcal{F}_{n}]\nonumber\\
={}&\sum_{i\in[k]}i\Prob[X_{n+1}=+1|P_{n}=i,\mathcal{F}_{n}]\Prob[P_{n}=i|\mathcal{F}_{n}]=H_{2}(n^{-1}A_{n},n^{-1}B_{n},n^{-1}C_{n}),\nonumber
\end{align}
where $H_{2}$ is as defined in \eqref{H_{2}_defn}, while
\begin{align}
&\E\left[(B_{n+1}-B_{n})^{2}|\mathcal{F}_{n}\right]=\E[(B_{n+1}-B_{n})(C_{n+1}-C_{n})|\mathcal{F}_{n}]\nonumber\\
={}&\sum_{i\in[k]}i^{2}\Prob[X_{n+1}=+1,P_{n}=i|\mathcal{F}_{n}]=H_{4}(n^{-1}A_{n},n^{-1}B_{n},n^{-1}C_{n}),\nonumber
\end{align}
where $H_{4}$ is as defined in \eqref{H_{4}_defn}. Since $C_{n+1}-C_{n}=P_{n}\chi\{X_{n+1}=+1\}+(k-P_{n})\chi\{X_{n+1}=-1\}$, we have 
\begin{align}
{}&\E\left[(C_{n+1}-C_{n})^{2}|\mathcal{F}_{n}\right]=\sum_{i\in[k]}i^{2}\Prob[X_{n+1}=+1,P_{n}=i|\mathcal{F}_{n}]+\sum_{i=0}^{k}(k-i)^{2}\Prob[X_{n+1}=-1,P_{n}=i|\mathcal{F}_{n}]\nonumber\\
={}&kq_{\beta}\left(\frac{A_{n}}{n},\frac{B_{n}}{n},\frac{C_{n}}{n}\right)\left\{1-q_{\beta}\left(\frac{A_{n}}{n},\frac{B_{n}}{n},\frac{C_{n}}{n}\right)\right\}+k^{2}\left\{1-q_{\beta}\left(\frac{A_{n}}{n},\frac{B_{n}}{n},\frac{C_{n}}{n}\right)\right\}^{2}-k^{2}H_{1}\left(\frac{A_{n}}{n},\frac{B_{n}}{n},\frac{C_{n}}{n}\right)\nonumber\\&+2k H_{2}\left(\frac{A_{n}}{n},\frac{B_{n}}{n},\frac{C_{n}}{n}\right).\nonumber
\end{align}
Finally, we have $\E[(A_{n+1}-A_{n})^{2}|\mathcal{F}_{n}]=H_{1}(n^{-1}A_{n},n^{-1}B_{n},n^{-1}C_{n})$ (since $A_{n+1}-A_{n}\in\{0,1\}\implies A_{n+1}-A_{n}=(A_{n+1}-A_{n})^{2}$), with $H_{1}$ as defined in \eqref{H_{1}_defn}. These computations, along with \eqref{A_{n+1}-A_{n}_cond_exp_fixed_size}, \eqref{B_{n+1}-B_{n}_cond_exp_fixed_size} and \eqref{C_{n+1}-C_{n}_cond_exp_fixed_size}, yield $\E[\Delta M_{n+1}\Delta M_{n+1}^{T}|\mathcal{F}_{n}]=Q(n^{-1}A_{n},n^{-1}B_{n},n^{-1}C_{n})$, with $Q:\mathcal{S}\rightarrow\mathbb{R}^{3\times 3}$ whose elements are given by \eqref{Q_defn}. The conclusions in the statement of Theorem~\ref{thm:main_6} now follow from Theorem~\ref{thm:borkar_fclt}.
\end{proof}

\section{Proofs of the results stated in \S\ref{subsec:main_results_growing_sample_size}}\label{sec:proofs_growing_sample_size} 
Once again, the key to proving results, stated in \S\ref{subsec:main_results_growing_sample_size}, pertaining to the model described in \S\ref{subsec:model_growing_sample_size} is its representation as a stochastic approximation process. For this model, the vertices $v_{1},\ldots,v_{N}$ have out-degree equal to $0$ in $G_{n}$, whereas the out-degree of $v_{j}$, for each $j\in\{N+1,\ldots,n\}$, equals $P_{j-1}$ if $X_{j}=+1$, and $(k_{j-1}-P_{j-1})$ if $X_{j}=-1$. A computation similar to \eqref{C_{n}_defn_fixed_size} shows that the sum of all in-degrees of $G_{n}$ equals
\begin{align}\label{C_{n}_defn_growing_size}
C_{n}=\sum_{i=N+1}^{n}P_{i-1}X_{i}+\sum_{i=N+1}^{n}k_{i-1}\chi\left\{X_{i}=-1\right\}.
\end{align}
Here, the random variable $P_{n}$, conditioned on $\mathcal{F}_{n}$, follows a binomial distribution in which the total number of trials equals $k_{n}$, and the probability of success may be computed the same way as in \eqref{success_probab_P_{n}_distribution}, but with $C_{n}$ now defined as in \eqref{C_{n}_defn_growing_size}, and making use of \eqref{beta_{n}_s_{n}_defns} and \eqref{q_defn}: 
\begin{align}
\sum_{m\in[n]}\frac{\indeg_{n}(v_{m})+\beta_{n}}{\sum_{j\in[n]}\indeg_{n}(v_{j})+n\beta_{n}}\chi\left\{X_{m}=+1\right\}=\frac{B_{n}+\beta_{n}A_{n}}{C_{n}+n\beta_{n}}=q_{\alpha}\left(\frac{A_{n}}{n},\frac{B_{n}}{s_{n}},\frac{C_{n}}{s_{n}}\right),\label{success_probab_P_{n}_distribution_growing_size}
\end{align}

The computations in \eqref{A_{n+1}-A_{n}_cond_exp_fixed_size}, \eqref{B_{n+1}-B_{n}_cond_exp_fixed_size} and \eqref{C_{n+1}-C_{n}_cond_exp_fixed_size} can be emulated to obtain the following identities for the model in \S\ref{subsec:model_growing_sample_size}, for each $n\geqslant N$:
\begin{equation}\label{increments_cond_exp_growing_sample_size}
\begin{cases}
\E\left[A_{n+1}-A_{n}\big|\mathcal{F}_{n}\right]=&H_{1,n}\left(n^{-1}A_{n},s_{n}^{-1}B_{n},s_{n}^{-1}C_{n}\right),\\
\E\left[B_{n+1}-B_{n}\big|\mathcal{F}_{n}\right]=&H_{2,n}\left(n^{-1}A_{n},s_{n}^{-1}B_{n},s_{n}^{-1}C_{n}\right),\\\E\left[C_{n+1}-C_{n}\big|\mathcal{F}_{n}\right]=&H_{3,n}\left(n^{-1}A_{n},s_{n}^{-1}B_{n},s_{n}^{-1}C_{n}\right),\end{cases}
\end{equation}
with the functions $H_{1,n}$, $H_{2,n}$ and $H_{3,n}$ defined, for each $(a,b,c)\in\hat{\mathcal{S}}$ (with $\hat{\mathcal{S}}$ as in \eqref{domain_defn_growing_size}), as follows:
\begin{align}
H_{1,n}(a,b,c)={}&\sum_{i=0}^{k_{n}}g\left(\frac{i}{k_{n}}\right){k_{n}\choose i}q_{\alpha}^{i}(a,b,c)\left\{1-q_{\alpha}(a,b,c)\right\}^{k_{n}-i},\label{H_{1,n}_defn}\\ 
H_{2,n}(a,b,c)={}&\sum_{i=0}^{k_{n}}ig\left(\frac{i}{k_{n}}\right){k_{n}\choose i}q_{\alpha}^{i}(a,b,c)\left\{1-q_{\alpha}(a,b,c)\right\}^{k_{n}-i},\label{H_{2,n}_defn}\\
H_{3,n}(a,b,c)={}&\sum_{i=0}^{k_{n}}\left[i g\left(\frac{i}{k_{n}}\right)+\left(k_{n}-i\right)\left\{1-g\left(\frac{i}{k_{n}}\right)\right\}\right]{k_{n}\choose i}q_{\alpha}^{i}(a,b,c)\left\{1-q_{\alpha}(a,b,c)\right\}^{k_{n}-i}.\label{H_{3,n}_defn}
\end{align}
The model in \S\ref{subsec:model_growing_sample_size} gives rise to a stochastic approximation process captured by \eqref{sa_growing_1} and \eqref{sa_growing_2}: 
\begin{align}
{}&\frac{A_{n+1}}{n+1}-\frac{A_{n}}{n}-\frac{k_{n}}{s_{n+1}}\left[\hat{\epsilon}_{n+1,1}+\hat{\delta}_{n,1}\right]=\frac{k_{n}}{s_{n+1}}\cdot\frac{s_{n+1}}{(n+1)k_{n}}\hat{h}_{1}\left(\frac{A_{n}}{n},\frac{B_{n}}{s_{n}},\frac{C_{n}}{s_{n}}\right)\in\frac{k_{n}}{s_{n+1}}\cdot[\eta_{1},\eta_{2}]\hat{h}_{1}\left(\frac{A_{n}}{n},\frac{B_{n}}{s_{n}},\frac{C_{n}}{s_{n}}\right),\label{sa_growing_1}
\end{align}
where $\eta_{1}$ and $\eta_{2}$ are as defined in \eqref{k_{n}_sequence_conditions}, and
\begin{align}
{}&\begin{bmatrix}
s_{n+1}^{-1}B_{n+1}\\
s_{n+1}^{-1}C_{n+1}
\end{bmatrix}=\begin{bmatrix}
s_{n}^{-1}B_{n}\\
s_{n}^{-1}C_{n}
\end{bmatrix}+\frac{k_{n}}{s_{n+1}}\left(\begin{bmatrix}
\hat{\epsilon}_{n+1,2}\\
\hat{\epsilon}_{n+1,3}
\end{bmatrix}+\begin{bmatrix}
\hat{\delta}_{n,2}\\
\hat{\delta}_{n,3}
\end{bmatrix}+\begin{bmatrix}
\hat{h}_{2}\left(n^{-1}A_{n},s_{n}^{-1}B_{n},s_{n}^{-1}C_{n}\right)\\
\hat{h}_{3}\left(n^{-1}A_{n},s_{n}^{-1}B_{n},s_{n}^{-1}C_{n}\right)
\end{bmatrix}\right),\label{sa_growing_2}
\end{align}
where $\hat{h}_{1}$, $\hat{h}_{2}$ and $\hat{h}_{3}$ are as defined in \eqref{hat{h}_defn_growing_size}. We set
\begin{equation}\label{error_defns_sa_growing}
\begin{cases}
\hat{\epsilon}_{n+1,1}&=(n+1)^{-1}k_{n}^{-1}s_{n+1}\left\{A_{n+1}-A_{n}-H_{1,n}\left(n^{-1}A_{n},s_{n}^{-1}B_{n},s_{n}^{-1}C_{n}\right)\right\},\\ 
\hat{\delta}_{n,1}&=(n+1)^{-1}k_{n}^{-1}s_{n+1}\left\{H_{1,n}\left(n^{-1}A_{n},s_{n}^{-1}B_{n},s_{n}^{-1}C_{n}\right)-\hat{H}_{1}\left(n^{-1}A_{n},s_{n}^{-1}B_{n},s_{n}^{-1}C_{n}\right)\right\}\\
\hat{\epsilon}_{n+1,2}&=k_{n}^{-1}\left(B_{n+1}-B_{n}\right)-k_{n}^{-1}H_{2,n}\left(n^{-1}A_{n},s_{n}^{-1}B_{n},s_{n}^{-1}C_{n}\right),\\
\hat{\epsilon}_{n+1,3}&=k_{n}^{-1}\left(C_{n+1}-C_{n}\right)-k_{n}^{-1}H_{3,n}\left(n^{-1}A_{n},s_{n}^{-1}B_{n},s_{n}^{-1}C_{n}\right),\\
\hat{\delta}_{n,2}&=k_{n}^{-1}H_{2,n}\left(n^{-1}A_{n},s_{n}^{-1}B_{n},s_{n}^{-1}C_{n}\right)-\hat{H}_{2}\left(n^{-1}A_{n},s_{n}^{-1}B_{n},s_{n}^{-1}C_{n}\right),\\
\hat{\delta}_{n,3}&=k_{n}^{-1}H_{3,n}\left(n^{-1}A_{n},s_{n}^{-1}B_{n},s_{n}^{-1}C_{n}\right)-\hat{H}_{3}\left(n^{-1}A_{n},s_{n}^{-1}B_{n},s_{n}^{-1}C_{n}\right),
\end{cases}
\end{equation}
where $\hat{H}_{1}$, $\hat{H}_{2}$ and $\hat{H}_{3}$ are as defined in \eqref{hat{H}_{1}_defn}, \eqref{hat{H}_{2}_defn} and \eqref{hat{H}_{3}_defn} respectively. From \eqref{increments_cond_exp_growing_sample_size} and \eqref{error_defns_sa_growing}, it is evident that $\{\epsilon_{n+1,i}:n\geqslant N\}$ forms a martingale difference sequence for each $i\in[3]$. Combining \eqref{sa_growing_1} and \eqref{sa_growing_2}, we can write, with $F$ as defined in \eqref{set_valued_function}:
\begin{align}
\begin{bmatrix}
A_{n+1}/(n+1)\\
B_{n+1}/s_{n+1}\\
C_{n+1}/s_{n+1}
\end{bmatrix}-\begin{bmatrix}
A_{n}/n\\
B_{n}/s_{n}\\
C_{n}/s_{n}
\end{bmatrix}-\frac{k_{n}}{s_{n+1}}\left\{\begin{bmatrix}
\hat{\epsilon}_{n+1,1}\\
\hat{\epsilon}_{n+1,2}\\
\hat{\epsilon}_{n+1,3}
\end{bmatrix}+\begin{bmatrix}
\hat{\delta}_{n,1}\\
\hat{\delta}_{n,2}\\
\hat{\delta}_{n,3}
\end{bmatrix}\right\}\in\frac{k_{n}}{s_{n+1}}F\left(\begin{bmatrix}
n^{-1}A_{n}\\
s_{n}^{-1}B_{n}\\
s_{n}^{-1}C_{n}
\end{bmatrix}\right),\label{sa_growing_differential_inclusion}
\end{align}
which is of exactly the same form as Equation (III) of \cite{benaim2005stochastic}.  

\begin{proof}[Proof of Theorem~\ref{thm:main_3}]
We prove Theorem~\ref{thm:main_3} by verifying the assumptions mentioned in the statement of Theorem~\ref{thm:differential_inclusions}. The definition of $F$ in \eqref{set_valued_function} and the third assumption stated in \eqref{k_{n}_sequence_conditions}, i.e.\ $0<\eta_{1}\leqslant\eta_{2}<\infty$, together ensure that $F(a,b,c)$ is compact as well as convex for each $(a,b,c)\in\hat{\mathcal{S}}$. This verifies \eqref{C2}. Recall, from \eqref{hat{H}_{1}_defn}, \eqref{hat{H}_{2}_defn} and \eqref{hat{H}_{3}_defn}, and the remark following \eqref{hat{H}_{3}_defn}, that the function $\hat{H}$ maps the set $\hat{\mathcal{S}}$ to itself, so that $0\leqslant \hat{H}_{1}(a,b,c)\leqslant 1$ and $0\leqslant \hat{H}_{2}(a,b,c)\leqslant \hat{H}_{3}(a,b,c)\leqslant 1$ for each $(a,b,c)\in\hat{\mathcal{S}}$. These, along with the definition of $\hat{h}_{i}$, for each $i\in[3]$, from \eqref{hat{h}_defn_growing_size}, yield 
\begin{equation}
\hat{h}_{i}^{2}(a,b,c)\leqslant \hat{H}_{i}^{2}(a,b,c)+\max\{a,b,c\}^{2}\leqslant 2 \text{ for each }(a,b,c)\in\hat{\mathcal{S}},\nonumber
\end{equation}
for each $i\in[3]$. Therefore, for each $\eta\in[\eta_{1},\eta_{2}]$ and each $(a,b,c)\in\hat{\mathcal{S}}$, we have  
\begin{equation}
\left|\left|\begin{bmatrix}
\eta\hat{h}_{1}(a,b,c)\\
\hat{h}_{2}(a,b,c)\\
\hat{h}_{3}(a,b,c)
\end{bmatrix}\right|\right|=\sqrt{\eta^{2}\hat{h}_{1}^{2}(a,b,c)+\hat{h}_{2}^{2}(a,b,c)+\hat{h}_{3}^{2}(a,b,c)}\leqslant \sqrt{2\left(\eta_{2}^{2}+2\right)}\nonumber
\end{equation}
thereby ensuring that Assumption~\eqref{C3} holds. The second assumption stated in \eqref{k_{n}_sequence_conditions} ensures that $a_{n+1}=k_{n}/s_{n+1}\rightarrow 0$ as $n\rightarrow\infty$. This, along with the first assumption stated in \eqref{k_{n}_sequence_conditions}, ensures that Assumption~\eqref{C4} is satisfied in this set-up.

A comparison of \eqref{sa_growing_differential_inclusion} with \eqref{sa_differential_inclusion} reveals that $E_{n+1}=(\hat{\epsilon}_{n+1,1},\hat{\epsilon}_{n+1,2},\hat{\epsilon}_{n+1,3})^{T}+(\hat{\delta}_{n,1},\hat{\delta}_{n,2},\hat{\delta}_{n,3})^{T}$ for each $n\geqslant N$. Suppose each of $\sum_{n\geqslant N}s_{n+1}^{-1}k_{n}(\hat{\epsilon}_{n+1,1},\hat{\epsilon}_{n+1,2},\hat{\epsilon}_{n+1,3})$ and $\sum_{n\geqslant N}s_{n+1}^{-1}k_{n}\sum_{j\in[3]}|\hat{\delta}_{n,j}|$ converges almost surely. Let $\Omega_{0}$ indicate the subset of the sample space, say $\Omega$, on which both of these series converge, so that $\Omega\setminus\Omega_{0}$ is of measure $0$. We fix any $\omega\in\Omega_{0}$. Setting $S_{n}=\sum_{N\leqslant i\leqslant n}s_{i+1}^{-1}k_{i}(\hat{\epsilon}_{i+1,1}(\omega),\hat{\epsilon}_{i+1,2}(\omega),\hat{\epsilon}_{i+1,3}(\omega))$ for each $n\geqslant N$, we see that the sequence $\{S_{n}:n\geqslant N\}$ must be Cauchy, so that, given any $\epsilon>0$, there exists some $N_{\epsilon}\in\mathbb{N}$ such that $||S_{m}-S_{n}||<\epsilon$ for all $m,n\geqslant N_{\epsilon}$. Consequently, for each $n\geqslant N_{\epsilon}+1$,
\begin{align}
\sup\left\{\left|\left|\sum_{i=n}^{k-1}\frac{k_{i}}{s_{i+1}}\begin{bmatrix}
\hat{\epsilon}_{i+1,1}(\omega)\\
\hat{\epsilon}_{i+1,2}(\omega)\\
\hat{\epsilon}_{i+1,3}(\omega)
\end{bmatrix}\right|\right|:k\in\{n+1,\ldots,m(\tau_{n}+T)\}\right\}\leqslant\sup\left\{\left|\left|S_{k-1}-S_{n-1}\right|\right|:k\geqslant n+1\right\}\leqslant\epsilon.\nonumber
\end{align}    
On the other hand, the tail sums of the series $\sum_{n\geqslant N}s_{n+1}^{-1}k_{n}\sum_{j\in[3]}|\hat{\delta}_{n,j}(\omega)|$ converge to $0$, so that 
\begin{align}
\sup\left\{\left|\left|\sum_{i=n}^{k-1}\frac{k_{i}}{s_{i+1}}\begin{bmatrix}
\hat{\delta}_{i,1}(\omega)\\
\hat{\delta}_{i,2}(\omega)\\
\hat{\delta}_{i,3}(\omega)
\end{bmatrix}\right|\right|:k\in\{n+1,\ldots,m(\tau_{n}+T)\}\right\}\leqslant\sum_{i=n}^{\infty}\frac{k_{i}}{s_{i+1}}\left|\left|\begin{bmatrix}
\hat{\delta}_{i,1}(\omega)\\
\hat{\delta}_{i,2}(\omega)\\
\hat{\delta}_{i,3}(\omega)
\end{bmatrix}\right|\right|\rightarrow 0 \text{ as }n\rightarrow\infty.\nonumber
\end{align}
Combining these two observations, we now conclude that, if each of $\sum_{n\geqslant N}s_{n+1}^{-1}k_{n}(\hat{\epsilon}_{n+1,1},\hat{\epsilon}_{n+1,2},\hat{\epsilon}_{n+1,3})^{T}$ and $\sum_{n\geqslant N}s_{n+1}^{-1}k_{n}\sum_{j\in[3]}|\hat{\delta}_{n,j}|$ converges almost surely, then \eqref{error_sequence_cond_differential_inclusions} holds, thereby ensuring that \eqref{C5} is satisfied. Therefore, the convergence of each of these two series is precisely what we verify in what follows.

Keeping in mind that $A_{n+1}-A_{n}$ lies in $\{0,1\}$, while each of $B_{n+1}-B_{n}$ and $C_{n+1}-C_{n}$ lies in $\{0,1,\ldots,k_{n}\}$, we have, from \eqref{increments_cond_exp_growing_sample_size}, \eqref{error_defns_sa_growing} and the second criterion stated in \eqref{k_{n}_sequence_conditions}: 
\begin{equation}
\sum_{n\geqslant N}\frac{k_{n}^{2}}{s_{n+1}^{2}}\left(\E\left[\hat{\epsilon}_{n+1,1}^{2}\big|\mathcal{F}_{n}\right]+\E\left[\hat{\epsilon}_{n+1,2}^{2}\big|\mathcal{F}_{n}\right]+\E\left[\hat{\epsilon}_{n+1,3}^{2}\big|\mathcal{F}_{n}\right]\right)\leqslant \sum_{n\geqslant N}\frac{1}{(n+1)^{2}}+2\sum_{n\geqslant N}\frac{k_{n}^{2}}{s_{n+1}^{2}}<\infty.\nonumber
\end{equation}
By \eqref{increments_cond_exp_growing_sample_size}, we know that $\{s_{n+1}^{-1}k_{n}(\hat{\epsilon}_{n+1,1},\hat{\epsilon}_{n+1,2},\hat{\epsilon}_{n+1,3})^{T}:n\geqslant N\}$ forms a martingale difference sequence, so that by Theorem 4.5.2.\ of \cite{durrett2019probability}, we conclude that $\sum_{n\geqslant N}s_{n+1}^{-1}k_{n}(\hat{\epsilon}_{n+1,1},\hat{\epsilon}_{n+1,2},\hat{\epsilon}_{n+1,3})^{T}$ converges almost surely.  

We now come to the analysis of the rate(s) at which $\hat{\delta}_{n,i}$, for each $i\in[3]$, decays as $n\rightarrow\infty$. The key is to observe that $H_{1,n}$ is a Bernstein polynomial approximation to $\hat{H}_{1}$, whereas $k_{n}^{-1}\hat{H}_{i,n}$ is a Bernstein polynomial approximation to $\hat{H}_{i}$ for $i\in\{2,3\}$, along the sequence $\{k_{n}\}$. We now consider three different cases:
\begin{enumerate}
\item When $g$ is assumed to be merely Lipschitz on $[0,1]$, it is immediate that each of $\hat{H}_{1}$, $\hat{H}_{2}$ and $\hat{H}_{3}$, defined in \eqref{hat{H}_{1}_defn}, \eqref{hat{H}_{2}_defn} and \eqref{hat{H}_{3}_defn}, is Lipschitz on $\hat{\mathcal{S}}$ as well. By \cite{mathe1999approximation} and the inequality that $x(1-x)\leqslant 1/4$ for all $x\in[0,1]$, each of $(n+1)k_{n}s_{n+1}^{-1}|\hat{\delta}_{n,1}|$, $|\hat{\delta}_{n,2}|$ and $|\hat{\delta}_{n,3}|$ is $O(k_{n}^{-1/2})$, so that
\begin{align}
\sum_{n\geqslant N}\frac{k_{n}}{s_{n+1}}\sum_{i\in[3]}\left|\hat{\delta}_{n,i}\right|=O\left(\sum_{n\geqslant N}\frac{k_{n}^{-1/2}}{(n+1)}\right)+O\left(\sum_{n\geqslant N}\frac{k_{n}^{1/2}}{s_{n+1}}\right),\nonumber
\end{align}
which converges because of the criteria stated in \eqref{A1}.

\item Suppose $g\in\mathcal{C}^{(1)}[0,1]$, which, along with the definition of $q_{\alpha}$ in \eqref{q_defn} and the fact that $\alpha>0$, ensures that each of $\hat{H}_{1}$, $\hat{H}_{2}$ and $\hat{H}_{3}$ is in $\mathcal{C}^{(1)}(\hat{\mathcal{S}})$. By Theorem 1.6.2.\ of \cite{lorentz2012bernstein}, we have
\begin{align} 
{}&\frac{k_{n}}{s_{n+1}}\left|\hat{\delta}_{n,1}\right|\leqslant\frac{1}{n+1}\sup\left\{\left|H_{1,n}(a,b,c)-\hat{H}_{1}(a,b,c)\right|:(a,b,c)\in\hat{\mathcal{S}}\right\}\leqslant \frac{3k_{n}^{-1/2}}{4(n+1)}\omega_{g'}\left(k_{n}^{-1/2}\right),\nonumber\\
{}&\frac{k_{n}}{s_{n+1}}\left|\hat{\delta}_{n,2}\right|\leqslant\frac{k_{n}}{s_{n+1}}\sup\left\{\left|\frac{1}{k_{n}}H_{2,n}(a,b,c)-\hat{H}_{2}(a,b,c)\right|:(a,b,c)\in\hat{\mathcal{S}}\right\}\leqslant \frac{3k_{n}^{1/2}}{4s_{n+1}}\omega_{\varphi'}\left(k_{n}^{-1/2}\right),\nonumber\\
{}&\frac{k_{n}}{s_{n+1}}\left|\hat{\delta}_{n,3}\right|\leqslant\frac{k_{n}}{s_{n+1}}\sup\left\{\left|\frac{1}{k_{n}}H_{3,n}(a,b,c)-\hat{H}_{3}(a,b,c)\right|:(a,b,c)\in\hat{\mathcal{S}}\right\}\leqslant \frac{3k_{n}^{1/2}}{4s_{n+1}}\omega_{\psi'}\left(k_{n}^{-1/2}\right),\nonumber
\end{align}
where the functions $\varphi$ and $\psi$ are as defined in \eqref{A2}. Right away, we see that $\sum_{n\geqslant N}s_{n+1}^{-1}k_{n}\sum_{i\in[3]}|\hat{\delta}_{n,i}|$ converges if each of the series in \eqref{C^{1}_convergence_criteria} converges. Simple computations show, for all $x,y\in[0,1]$: 
\begin{align}
{}&\text{each of }|\varphi'(x)-\varphi'(y)|\text{ and }|\psi'(x)-\psi'(y)|\text{ is }O(|g'(x)-g'(y)|)+O(|x-y|)\nonumber\\
\implies{}&\text{each of }\omega_{\varphi'}(k_{n}^{-1/2})\text{ and }\omega_{\psi'}(k_{n}^{-1/2})\text{ is }O(\omega_{g'}(k_{n}^{-1/2}))+O(k_{n}^{-1/2}),\nonumber
\end{align}
\sloppy so that each of the last two series in \eqref{C^{1}_convergence_criteria} converges whenever each of the two series $\sum_{n\geqslant N}s_{n+1}^{-1}$ and $\sum_{n\geqslant N}s_{n+1}^{-1}k_{n}^{1/2}\omega_{g'}(k_{n}^{-1/2})$ converges, showing the relevance of the last part of \eqref{A2}.

\item Finally, suppose $g\in\mathcal{C}^{(2)}[0,1]$, which, as in the previous case, is enough to ensure that each of $\hat{H}_{1}$, $\hat{H}_{2}$ and $\hat{H}_{3}$ is in $\mathcal{C}^{(2)}(\hat{\mathcal{S}})$. By the result (from \cite{voronovskaja1932determination}) stated in \S 1.6.1.\ of \cite{lorentz2012bernstein}, each of $(n+1)k_{n}s_{n+1}^{-1}|\hat{\delta}_{n,1}|$, $|\hat{\delta}_{n,2}|$ and $|\hat{\delta}_{n,3}|$ is $O(k_{n}^{-1})$, so that 
\begin{align} 
\sum_{n\geqslant N}\frac{k_{n}}{s_{n+1}}\sum_{i\in[3]}\left|\hat{\delta}_{n,i}\right|=O\left(\sum_{n\geqslant N}\frac{k_{n}^{-1}}{(n+1)}\right)+O\left(\sum_{n\geqslant N}\frac{1}{s_{n+1}}\right),\nonumber
\end{align}
which converges because of the criteria stated in \eqref{A3}.
\end{enumerate}
We may, thus, conclude that whenever one of \eqref{A1}, \eqref{A2} and \eqref{A3} holds, the series $\sum_{n\geqslant N}s_{n+1}^{-1}k_{n}\sum_{i\in[3]}|\hat{\delta}_{n,i}|$ converges. As justified above, this concludes the verification of Assumption~\eqref{C5}, and the conclusion of Theorem~\ref{thm:main_3} follows from the conclusion drawn in Theorem~\ref{thm:differential_inclusions}.
\end{proof}

\begin{lemma}\label{lem:strongly_positively_invariant_growing_sample_size}
Recall $\eta_{2}$ defined in \eqref{k_{n}_sequence_conditions}. As long as $\eta_{2}\leqslant 1$, the set $\hat{\mathcal{S}}$, defined in \eqref{domain_defn_growing_size}, is strongly positively invariant corresponding to the differential inclusion $(\dot{a}(t),\dot{b}(t),\dot{c}(t))\in F(a(t),b(t),c(t))$ for all $t\geqslant 0$, where $F$ is as defined in \eqref{set_valued_function}. 
\end{lemma}
\begin{proof}
Recall that, given $(a,b,c)\in\hat{\mathcal{S}}$, a solution $\varphi=\{\varphi(t):t\geqslant 0\}$ to the differential inclusion $(\dot{a}(t),\dot{b}(t),\dot{c}(t))\in F(a(t),b(t),c(t))$ for all $t\geqslant 0$, with $(a(0),b(0),c(0))=(a,b,c)$, is such that 
\begin{enumerate*}
\item $t\mapsto \varphi(t)$ is absolutely continuous in $t\in[0,\infty)$, and
\item $\dot{\varphi}(t)\in F(\varphi(t))$ for almost every $t\geqslant 0$, with $\varphi(0)=(a,b,c)$.
\end{enumerate*}
Let $\Phi(a,b,c)$ denote the set of all such solutions $\varphi$. Our goal, here, is to show that $\varphi(t)\in\hat{\mathcal{S}}$ for every $t\geqslant 0$ and every $\varphi\in\Phi(a,b,c)$, for each $(a,b,c)\in\hat{\mathcal{S}}$. To this end, we make use of Theorem 5.2.1 of \cite{Aubin1991}.

For $(a,b,c)\in\hat{\mathcal{S}}$, a typical element of $F(a,b,c)$ can be written as $(\eta\hat{h}_{1}(a,b,c),\hat{h}_{2}(a,b,c),\hat{h}_{3}(a,b,c))$ for \emph{some} $\eta\in[\eta_{1},\eta_{2}]$. Using the definitions introduced in \eqref{hat{h}_defn_growing_size}, we obtain:
\begin{align}
\begin{bmatrix}
a\\
b\\
c
\end{bmatrix}+t\begin{bmatrix}
\eta\hat{h}_{1}(a,b,c)\\
\hat{h}_{2}(a,b,c)\\
\hat{h}_{3}(a,b,c)
\end{bmatrix}=\begin{bmatrix}
(1-\eta t)a+\eta t\hat{H}_{1}(a,b,c)\\
(1-t)b+t\hat{H}_{2}(a,b,c)\\
(1-t)c+t\hat{H}_{3}(a,b,c)
\end{bmatrix},\label{convex_combination}
\end{align}
and keeping in mind that $\hat{H}=(\hat{H}_{1},\hat{H}_{2},\hat{H}_{3})$ maps $\hat{\mathcal{S}}$ to itself (as explained right after \eqref{hat{H}_{3}_defn}) and that $\eta\in[\eta_{1},\eta_{2}]\implies 0<\eta\leqslant 1$, we see that the final tuple in \eqref{convex_combination} is an element of $\hat{\mathcal{S}}$ for all $t\in(0,1]$ -- therefore, the distance between the set $\hat{\mathcal{S}}$ and the final tuple in \eqref{convex_combination} equals $0$. Thus, the cotingent cone to $\hat{\mathcal{S}}$ at any $(a,b,c)\in\hat{\mathcal{S}}$ (see Definition 5.1.1 of \cite{Aubin1991}) contains $F(a,b,c)$. By Theorem 5.2.1 of \cite{Aubin1991}, we conclude that $\hat{\mathcal{S}}$ is strongly positively invariant, i.e.\ $\varphi(t)\in\hat{\mathcal{S}}$ for all $t\geqslant 0$, for every $\varphi\in\Phi(a,b,c)$, for each $(a,b,c)\in\hat{\mathcal{S}}$.  
\end{proof}

\begin{proof}[Proof of Theorem~\ref{thm:main_5}]
The proof of Theorem~\ref{thm:main_5} consists of two parts: 
\begin{enumerate*}
\item coming up with a suitable Lyapunov function $V:\hat{\mathcal{S}}\rightarrow\mathbb{R}$, with $\hat{\mathcal{S}}$ as defined in \eqref{domain_defn_growing_size}, such that $V$ is non-negative everywhere on $\hat{\mathcal{S}}$ and its time-derivative, $\dot{V}$, is strictly negative on $\hat{\mathcal{S}}\setminus\hat{\Lambda}$ and zero on $\hat{\Lambda}$, where $\hat{\Lambda}$ is as defined in \eqref{hat{Lambda}_defn};
\item then showing, using $V$, that \emph{every} compact, connected, internally chain transitive invariant set must be contained in $\hat{\Lambda}$.
\end{enumerate*}

\sloppy \textbf{Coming up with a suitable Lyapunov function:} For $(a(t),b(t),c(t))$ satisfying the differential inclusion $(\dot{a}(t),\dot{b}(t),\dot{c}(t))\in F(a(t),b(t),c(t))$ for all $t\geqslant 0$, with $F$ as defined in \eqref{set_valued_function}, we, henceforth, write $\dot{a}(t)=\eta(t)\{g(q_{\alpha}(a(t),b(t),c(t)))-a(t)\}$ for $\eta(t)\in[\eta_{1},\eta_{2}]$. Using this, \eqref{hat{h}_defn_growing_size}, \eqref{hat{H}_{2}_defn}, \eqref{hat{H}_{3}_defn} and \eqref{hat{G}_defn}, we compute the time-derivative of $q_{\alpha}(a(t),b(t),c(t))$ by differentiating both sides of the identity $q_{\alpha}(a(t),b(t),c(t))\{c(t)+\alpha\}=b(t)+\alpha a(t)$ and substantial simplification:
\begin{align}
{}&\dot{q}_{\alpha}(a(t),b(t),c(t))\{c(t)+\alpha\}+q_{\alpha}(a(t),b(t),c(t))\dot{c}(t)=\dot{b}(t)+\alpha \dot{a}(t)\nonumber\\
\Longleftrightarrow{}&\dot{q}_{\alpha}(a(t),b(t),c(t))=\frac{\hat{G}\left(q_{\alpha}(a(t),b(t),c(t))\right)+\alpha\{1-\eta(t)\}\left\{a(t)-g(q_{\alpha}(a(t),b(t),c(t)))\right\}}{c(t)+\alpha}.\label{q_{alpha}_time_derivative}
\end{align}

Let us now define, for each $(a,b,c)\in\hat{\mathcal{S}}$, and for a suitable $\kappa>0$ to be chosen eventually, the function
\begin{align}
V(a,b,c)={}&-\int_{q^{*}}^{q_{\alpha}(a,b,c)}\hat{G}(s)ds+\frac{\kappa}{2}\left\{a-g\left(q_{\alpha}(a,b,c)\right)\right\}^{2}.\label{candidate_Lyapunov}
\end{align}
From \eqref{hat{G}_defn} and two of the assumptions mentioned in the statement of Theorem~\ref{thm:main_5}, we have $\hat{G}(0)=\alpha g(0)>0$ and $\hat{G}(1)=-\alpha\{1-g(1)\}<0$. Along with the assumption that $\hat{G}$ has a unique root in $[0,1]$, this ensures that $\hat{G}(q)>0$ for all $q\in[0,q^{*})$ and $\hat{G}(q)<0$ for all $q\in(q^{*},1]$. Therefore, the integrated term in \eqref{candidate_Lyapunov} is strictly positive for all $q_{\alpha}(a,b,c)\neq q^{*}$. Even when we have $q_{\alpha}(a,b,c)= q^{*}$, the second term in \eqref{candidate_Lyapunov} is strictly positive for all $a\neq a^{*}=g(q^{*})$. These observations lead to the conclusion that $V(a,b,c)>0$ for all $(a,b,c)\in\hat{\mathcal{S}}\setminus\hat{\Lambda}$.

Using \eqref{q_{alpha}_time_derivative} and the Leibniz rule for differentiation of an integral, we compute the time-derivative
\begin{align}
{}&\dot{V}(a(t),b(t),c(t))=-\dot{q}_{\alpha}(a(t),b(t),c(t))\hat{G}(q_{\alpha}(a(t),b(t),c(t)))+\kappa\{\dot{a}(t)-\dot{q}_{\alpha}(a(t),b(t),c(t))\nonumber\\&g'\left(q_{\alpha}(a(t),b(t),c(t))\right)\}\left\{a(t)-g\left(q_{\alpha}(a(t),b(t),c(t))\right)\right\}\nonumber\\
={}&-\frac{\hat{G}^{2}\left(q_{\alpha}(a(t),b(t),c(t))\right)+\alpha\{1-\eta(t)\}\left\{a(t)-g(q_{\alpha}(a(t),b(t),c(t)))\right\}\hat{G}(q_{\alpha}(a(t),b(t),c(t)))}{c(t)+\alpha}\nonumber\\&-\frac{\kappa g'\left(q_{\alpha}(a(t),b(t),c(t))\right)\hat{G}\left(q_{\alpha}(a(t),b(t),c(t))\right)\left\{a(t)-g\left(q_{\alpha}(a(t),b(t),c(t))\right)\right\}}{c(t)+\alpha}\nonumber\\&-\frac{\kappa g'\left(q_{\alpha}(a(t),b(t),c(t))\right)\alpha\{1-\eta(t)\}\left\{a(t)-g(q_{\alpha}(a(t),b(t),c(t)))\right\}^{2}}{c(t)+\alpha}\nonumber\\&-\kappa\eta(t)\{g(q_{\alpha}(a(t),b(t),c(t)))-a(t)\}^{2}\nonumber\\
={}&-\frac{1}{c(t)+\alpha}\begin{bmatrix}
\hat{G}(q_{\alpha}(a(t),b(t),c(t)))\\
a(t)-g(q_{\alpha}(a(t),b(t),c(t)))
\end{bmatrix}^{T}\Sigma\begin{bmatrix}
\hat{G}(q_{\alpha}(a(t),b(t),c(t)))\\
a(t)-g(q_{\alpha}(a(t),b(t),c(t)))
\end{bmatrix},\label{V_time_derivative}
\end{align}
where $\Sigma=\Sigma(a(t),b(t),c(t))$ is a $2\times 2$ symmetric matrix, with $\sigma_{i,j}$ indicating the entry in the $i$-th row and $j$-th column for $i,j\in[2]$, given by
\begin{equation}\label{Sigma_defn}
\begin{cases}
&\sigma_{1,1}=1,\quad \sigma_{1,2}=\left[\alpha\{1-\eta(t)\}+\kappa g'\left(q_{\alpha}(a(t),b(t),c(t))\right)\right]/2,\\
& \sigma_{2,2}=\kappa g'\left(q_{\alpha}(a(t),b(t),c(t))\right)\alpha\{1-\eta(t)\}+\kappa \eta(t)\{c(t)+\alpha\}.
\end{cases}
\end{equation}
At this point, our goal is to ensure that, for suitable choices of $\kappa$, the matrix $\Sigma$ is positive definite. This, in turn, would ensure that $\dot{V}(a(t),b(t),c(t))\leqslant 0$, with equality if and only if $\hat{G}(q_{\alpha}(a(t),b(t),c(t)))=0$ and $a(t)=g(q_{\alpha}(a(t),b(t),c(t)))$, which, in turn, is true if and only if $q_{\alpha}(a(t),b(t),c(t))=q^{*}$ and $a(t)=a^{*}=g(q^{*})$, or, in other words, $(a(t),b(t),c(t))\in\hat{\Lambda}$, where $\hat{\Lambda}$ is as defined in \eqref{hat{Lambda}_defn}. Since $\sigma_{1,1}>0$, it suffices to ensure that $\det(\Sigma)>0$. Since $c(t)\geqslant 0$ (because $(a(t),b(t),c(t))\in\hat{\mathcal{S}}$, with $\hat{\mathcal{S}}$ as defined in \eqref{domain_defn_growing_size}), we see that $\det(\Sigma)$ is bounded below by: 
\begin{align}
-\frac{1}{4}\left[\kappa^{2}{g'(q_{\alpha}(a(t),b(t),c(t)))}^{2}-\kappa\{4\eta(t)\alpha+2g'(q_{\alpha}(a(t),b(t),c(t)))\alpha\{1-\eta(t)\}\}+\alpha^{2}\{1-\eta(t)\}^{2}\right].\label{quadratic}
\end{align}
Our task, now, is to choose $\kappa$ such that the expression in \eqref{quadratic} is stritctly positive for all $(a(t),b(t),c(t))$ (and consequently, for the corresponding values of $\eta(t)$). 

To begin with, we note that if $g$ is constant throughout $[0,1]$, then the expression in \eqref{quadratic} boils down to $\kappa\eta(t)\alpha-\alpha^{2}\{1-\eta(t)\}^{2}/4$, so that we must choose $\kappa>\eta(t)^{-1}\alpha\{1-\eta(t)\}^{2}/4$. The function $\eta^{-1}(1-\eta)^{2}$ is strictly decreasing in $\eta\in[\eta_{1},\eta_{2}]$ since $\eta_{2}\leqslant 1$, so that its maximum value is attained at $\eta=\eta_{1}$, and it suffices, therefore, to select $\kappa>\alpha\eta_{1}^{-1}(1-\eta_{1})^{2}/4$.

We now assume that $g$ is \emph{not} constant throughout $[0,1]$. Since the function $g$ has been assumed to be in $\mathcal{C}^{(1)}[0,1]$, as well as monotonically increasing, in the statement of Theorem~\ref{thm:main_5}, there exist some $0\leqslant\delta\leqslant\upsilon<\infty$, with $\upsilon>0$, such that $\delta\leqslant g'(q_{\alpha}(a,b,c))\leqslant \upsilon$ for all $(a,b,c)\in\hat{\mathcal{S}}$. At this point, we set 
\begin{align}
{}&\epsilon=\delta \text{ if } 0<\delta\leqslant\upsilon<\delta+\frac{4}{1-\eta_{1}}\left\{\eta_{1}+\sqrt{\eta_{1}^{2}+\eta_{1}(1-\eta_{1})\delta}\right\},\label{case_1}\\
{}&\text{any } \epsilon \text{ with } 0<\epsilon\leqslant\upsilon<\epsilon+\frac{4}{1-\eta_{1}}\left\{\eta_{1}+\sqrt{\eta_{1}^{2}+\eta_{1}(1-\eta_{1})\epsilon}\right\} \text{ and } \epsilon<\frac{2\eta_{1}}{1-\eta_{1}}, \text{ otherwise}.\label{restrictions_on_epsilon}
\end{align}
That an $\epsilon$ satisfying \eqref{restrictions_on_epsilon} can be chosen while having $\upsilon$ conform to the restrictions relative to $\eta_{1}$ in the statement of Theorem~\ref{thm:main_5}, can be justified as follows. Let us write $\epsilon=r_{1}(1-\eta_{1})^{-1}\eta_{1}$ and $\upsilon=r_{2}(1-\eta_{1})^{-1}\eta_{1}$. According to the statement of Theorem~\ref{thm:main_5}, we must have $r_{2}\leqslant (1+\sqrt{5})$. On the other hand, the first sequence of inequalities in \eqref{restrictions_on_epsilon} boils down to 
\begin{equation}
0<r_{1}\leqslant r_{2}<4+r_{1}+4\sqrt{1+r_{1}},\nonumber
\end{equation}
and since $(1+\sqrt{5})<4$, we can choose \emph{any} $0<r_{1}\leqslant r_{2}$ with $r_{1}<2$. When we are in the scenario given by \eqref{restrictions_on_epsilon}, i.e.\ when $\upsilon\geqslant\delta+4(1-\eta_{1})^{-1}\{\eta_{1}+\sqrt{\eta_{1}^{2}+\eta_{1}(1-\eta_{1})\delta}\}$, we work with two sub-intervals, namely $[0,\epsilon]$ (which covers $[\delta,\epsilon]$) and $[\epsilon,\upsilon]$.

We focus on the expression in \eqref{quadratic}, when $g'(q_{\alpha}(a(t),b(t),c(t)))\in[\epsilon,\upsilon]$. For the sake of brevity, we rewrite the expression in \eqref{quadratic}, replacing $g'(q_{\alpha}(a(t),b(t),c(t)))$ by $y$ and $\eta(t)$ by $\eta$, to obtain:
\begin{align}
\det(\Sigma)\geqslant-\frac{1}{4}\left[\kappa^{2}y^{2}-\kappa\{4\eta\alpha+2y\alpha(1-\eta)\}+\alpha^{2}(1-\eta)^{2}\right],\label{quadratic_rewritten}
\end{align}
and the lower bound in \eqref{quadratic_rewritten} is strictly positive as long as $\kappa$ lies strictly between the roots of this quadratic polynomial, or, in other words, as long as (note that the discriminant is positive since $0<\eta_{1}\leqslant\eta_{2}\leqslant 1$)
\begin{align}
{}&\frac{\alpha\left\{2\eta+y(1-\eta)-2\sqrt{\eta^{2}+\eta(1-\eta)y}\right\}}{y^{2}}<\kappa<\frac{\alpha\left\{2\eta+y(1-\eta)+2\sqrt{\eta^{2}+\eta(1-\eta)y}\right\}}{y^{2}}\nonumber\\
\Longleftrightarrow{}&\alpha\left(\frac{\sqrt{\eta+(1-\eta)y}-\sqrt{\eta}}{y}\right)^{2}<\kappa<\alpha\left(\frac{\sqrt{\eta+(1-\eta)y}+\sqrt{\eta}}{y}\right)^{2} \text{ for all }\eta\in[\eta_{1},\eta_{2}] \text{ and }y\in[\epsilon,\upsilon].\label{between_roots_epsilon_upsilon_interval}
\end{align}
We now examine the behaviour of each of the lower and upper bounds in \eqref{between_roots_epsilon_upsilon_interval}, first as a function of $y$, and then as a function of $\eta$. Note that
\begin{align}
{}&\frac{d}{dy}\frac{\sqrt{\eta+(1-\eta)y}-\sqrt{\eta}}{y}
=\frac{-\{\sqrt{\eta+(1-\eta)y}-\sqrt{\eta}\}^{2}}{2y^{2}\sqrt{\eta+(1-\eta)y}},\nonumber\\
{}&\frac{d}{dy}\frac{\sqrt{\eta+(1-\eta)y}+\sqrt{\eta}}{y}
=\frac{-\{\sqrt{\eta+(1-\eta)y}+\sqrt{\eta}\}^{2}}{2y^{2}\sqrt{\eta+(1-\eta)y}},\nonumber
\end{align}
allowing us to conclude that, for each fixed $\eta\in[\eta_{1},\eta_{2}]$,
\begin{align}
{}&\sup_{y\in[\epsilon,\upsilon]}\alpha\left(\frac{\sqrt{\eta+(1-\eta)y}-\sqrt{\eta}}{y}\right)^{2}=\alpha\left(\frac{\sqrt{\eta+(1-\eta)\epsilon}-\sqrt{\eta}}{\epsilon}\right)^{2},\label{sup_lower_bound_y}\\
{}&\inf_{y\in[\epsilon,\upsilon]}\alpha\left(\frac{\sqrt{\eta+(1-\eta)y}+\sqrt{\eta}}{y}\right)^{2}=\alpha\left(\frac{\sqrt{\eta+(1-\eta)\upsilon}+\sqrt{\eta}}{\upsilon}\right)^{2}.\label{inf_upper_bound_y}
\end{align}
Next, we examine the expression on the right side of each of \eqref{sup_lower_bound_y} and \eqref{inf_upper_bound_y} as a function of $\eta$. Note that
\begin{equation}\label{lower_bound_derivative_eta}
\frac{d}{d\eta}\left\{\sqrt{\eta+(1-\eta)\epsilon}-\sqrt{\eta}\right\}=
\begin{cases}
\dfrac{1-\epsilon}{2\sqrt{\eta+(1-\eta)\epsilon}}-\dfrac{1}{2\sqrt{\eta}} &\text{when } \epsilon> 1,\\
\dfrac{\epsilon\{-\eta(1-\epsilon)-1\}}{2\sqrt{\eta+(1-\eta)\epsilon}\sqrt{\eta}\left\{\sqrt{\eta}(1-\epsilon)+\sqrt{\eta+(1-\eta)\epsilon}\right\}} &\text{when }\epsilon\leqslant 1,
\end{cases}
\end{equation}
showing, right away, that the right side of \eqref{lower_bound_derivative_eta} is always non-positive. On the other hand, we have
\begin{equation}\label{upper_bound_derivative_eta}
\frac{d}{d\eta}\left\{\sqrt{\eta+(1-\eta)\upsilon}+\sqrt{\eta}\right\}=
\begin{cases}
\dfrac{1-\upsilon}{2\sqrt{\eta+(1-\eta)\upsilon}}+\dfrac{1}{2\sqrt{\eta}}&\text{when }\upsilon\leqslant 1,\\
\dfrac{\upsilon\{1-\eta(\upsilon-1)\}}{2\sqrt{\eta}\sqrt{\eta+(1-\eta)\upsilon}\left\{\sqrt{\eta+(1-\eta)\upsilon}+(\upsilon-1)\sqrt{\eta}\right\}}&\text{when }\upsilon>1.
\end{cases}
\end{equation}
In the second scenario described in \eqref{upper_bound_derivative_eta}, the sign of the derivative is determined by the sign of $1-\eta(\upsilon-1)$, which is strictly decreasing in $\upsilon$, and equals $1-\eta\geqslant1-\eta_{2}\geqslant 0$ at $\upsilon=2$ (recall, from the statement of Theorem~\ref{thm:main_5}, that $\upsilon\leqslant 2$). We may, thus, conclude that the derivative in \eqref{upper_bound_derivative_eta} is always non-negative. Combining \eqref{sup_lower_bound_y} and \eqref{inf_upper_bound_y} with the observations made regarding \eqref{lower_bound_derivative_eta} and \eqref{upper_bound_derivative_eta}, we have:
\begin{align}
{}&\xi_{1}:=\sup_{\eta\in[\eta_{1},\eta_{2}]}\sup_{y\in[\epsilon,\upsilon]}\alpha\left(\frac{\sqrt{\eta+(1-\eta)y}-\sqrt{\eta}}{y}\right)^{2}=\alpha\left(\frac{\sqrt{\eta_{1}+(1-\eta_{1})\epsilon}-\sqrt{\eta_{1}}}{\epsilon}\right)^{2},\label{sup_lower_bound_y,eta}\\
{}&\xi_{2}:=\inf_{\eta\in[\eta_{1},\eta_{2}]}\inf_{y\in[\epsilon,\upsilon]}\alpha\left(\frac{\sqrt{\eta+(1-\eta)y}+\sqrt{\eta}}{y}\right)^{2}=\alpha\left(\frac{\sqrt{\eta_{1}+(1-\eta_{1})\upsilon}+\sqrt{\eta_{1}}}{\upsilon}\right)^{2}.\label{inf_upper_bound_y,eta}
\end{align}
From \eqref{between_roots_epsilon_upsilon_interval}, it is evident that $\kappa$ must be chosen to lie strictly between $\xi_{1}$ and $\xi_{2}$. For this to be possible, we must make sure that \eqref{inf_upper_bound_y,eta} strictly exceeds \eqref{sup_lower_bound_y,eta}, or, in other words,
\begin{align}
{}&\frac{\sqrt{\eta_{1}+(1-\eta_{1})\epsilon}-\sqrt{\eta_{1}}}{\epsilon}<\frac{\sqrt{\eta_{1}+(1-\eta_{1})\upsilon}+\sqrt{\eta_{1}}}{\upsilon}\nonumber\\
\Longleftrightarrow{}&\frac{\eta_{1}+(1-\eta_{1})\epsilon-\eta_{1}}{\epsilon\left\{\sqrt{\eta_{1}+(1-\eta_{1})\epsilon}+\sqrt{\eta_{1}}\right\}}<\frac{\eta_{1}+(1-\eta_{1})\upsilon-\eta_{1}}{\upsilon\left\{\sqrt{\eta_{1}+(1-\eta_{1})\upsilon}-\sqrt{\eta_{1}}\right\}}\nonumber\\
\Longleftrightarrow{}&\sqrt{\eta_{1}+(1-\eta_{1})\upsilon}-\sqrt{\eta_{1}}<\sqrt{\eta_{1}+(1-\eta_{1})\epsilon}+\sqrt{\eta_{1},}\nonumber
\end{align}
which is true if and only if the inequality involving $\upsilon$ and $\epsilon$ in \eqref{restrictions_on_epsilon} (which is the same as the inequality in \eqref{case_1}, since $\epsilon=\delta$ in this case) is satisfied. The analysis ends here for the case described in \eqref{case_1}, and any choice of $\kappa$ lying strictly between the expressions in \eqref{sup_lower_bound_y,eta} and \eqref{inf_upper_bound_y,eta} suffices to ensure that $\dot{V}(a(t),b(t),c(t))\leqslant 0$ for all $(a(t),b(t),c(t))\in\hat{\mathcal{S}}$, with equality if and only if $(a(t),b(t),c(t))\in\hat{\Lambda}$.

When we are in the scenario described in \eqref{restrictions_on_epsilon}, our next task is to focus on $y$ lying in the interval $[0,\epsilon]$ (which may equal, or be a superset of, the interval $[\delta,\epsilon]$). Continuing with \eqref{quadratic_rewritten}, we obtain the inequality: $\det(\Sigma)\geqslant-[\kappa^{2}\epsilon^{2}-4\kappa\eta\alpha+\alpha^{2}(1-\eta)^{2}]/4$, and this lower bound is positive if and only if
\begin{align}
\frac{\alpha\left\{2\eta-\sqrt{4\eta^{2}-\epsilon^{2}(1-\eta)^{2}}\right\}}{\epsilon^{2}}<\kappa<\frac{\alpha\left\{2\eta+\sqrt{4\eta^{2}-\epsilon^{2}(1-\eta)^{2}}\right\}}{\epsilon^{2}} \text{ for all }\eta\in[\eta_{1},\eta_{2}].\label{interval_for_kappa_0_epsilon}
\end{align}
Note that the last inequality in \eqref{restrictions_on_epsilon} ensures that the discriminant here is strictly positive. We see that
\begin{align}
{}&\frac{d}{d\eta}\left\{2\eta-\sqrt{4\eta^{2}-\epsilon^{2}(1-\eta)^{2}}\right\}
=\frac{-\epsilon^{2}(1-\eta)\left\{4(1-\eta)+\epsilon^{2}(1-\eta)+8\eta\right\}}{\sqrt{4\eta^{2}-\epsilon^{2}(1-\eta)^{2}}\left\{2\sqrt{4\eta^{2}-\epsilon^{2}(1-\eta)^{2}}+4\eta+\epsilon^{2}(1-\eta)\right\}},\nonumber
\end{align}
showing us that the lower bound in \eqref{interval_for_kappa_0_epsilon} is strictly decreasing in $\eta$. On the other hand,
\begin{align}
{}&\frac{d}{d\eta}\left\{2\eta+\sqrt{4\eta^{2}-\epsilon^{2}(1-\eta)^{2}}\right\}
=2+\frac{4\eta+\epsilon^{2}(1-\eta)}{\sqrt{4\eta^{2}-\epsilon^{2}(1-\eta)^{2}}},\nonumber
\end{align}
showing us that the upper bound in \eqref{interval_for_kappa_0_epsilon} is strictly increasing in $\eta$. We can therefore write
\begin{align}
{}&\xi_{3}:=\sup_{\eta\in[\eta_{1},\eta_{2}]}\frac{\alpha\left\{2\eta-\sqrt{4\eta^{2}-\epsilon^{2}(1-\eta)^{2}}\right\}}{\epsilon^{2}}=\frac{\alpha\left\{2\eta_{1}-\sqrt{4\eta_{1}^{2}-\epsilon^{2}(1-\eta_{1})^{2}}\right\}}{\epsilon^{2}},\label{lower_bound_kappa_0_epsilon}\\
{}&\xi_{4}:=\inf_{\eta\in[\eta_{1},\eta_{2}]}\frac{\alpha\left\{2\eta+\sqrt{4\eta^{2}-\epsilon^{2}(1-\eta)^{2}}\right\}}{\epsilon^{2}}=\frac{\alpha\left\{2\eta_{1}+\sqrt{4\eta_{1}^{2}-\epsilon^{2}(1-\eta_{1})^{2}}\right\}}{\epsilon^{2}}.\label{upper_bound_kappa_0_epsilon}
\end{align}

We now have to ensure that the open interval $(\xi_{1},\xi_{2})$, where $\xi_{1}$ and $\xi_{2}$ are as defined in \eqref{sup_lower_bound_y,eta} and \eqref{inf_upper_bound_y,eta}, has a non-empty intersection with the open interval $(\xi_{3},\xi_{4})$, where $\xi_{3}$ and $\xi_{4}$ are as defined in \eqref{lower_bound_kappa_0_epsilon} and \eqref{upper_bound_kappa_0_epsilon} (so that $\kappa$ can be chosen to lie inside \emph{each} of these two intervals). Note that $\xi_{4}>\xi_{1}$ since
\begin{align}
{}&2\eta_{1}+\sqrt{4\eta_{1}^{2}-\epsilon^{2}(1-\eta_{1})^{2}}-\left(\sqrt{\eta_{1}+(1-\eta_{1})\epsilon}-\sqrt{\eta_{1}}\right)^{2}\nonumber\\
\geqslant{}&\frac{2\left\{4\eta_{1}^{2}-(1-\eta_{1})^{2}\epsilon^{2}\right\}+4\eta_{1}(1-\eta_{1})\epsilon}{\sqrt{4\eta_{1}^{2}-\epsilon^{2}(1-\eta_{1})^{2}}+2\sqrt{\eta_{1}^{2}+\eta_{1}(1-\eta_{1})\epsilon}+(1-\eta_{1})\epsilon},\nonumber
\end{align}
which is positive because of the last inequality in \eqref{restrictions_on_epsilon}. On the other hand, we have
\begin{align} 
{}&\upsilon^{-2}\left\{\sqrt{\eta_{1}+(1-\eta_{1})\upsilon}+\sqrt{\eta_{1}}\right\}^{2}-\epsilon^{-2}\left\{2\eta_{1}-\sqrt{4\eta_{1}^{2}-\epsilon^{2}(1-\eta_{1})^{2}}\right\}\nonumber\\
>{}&2(1-\eta_{1})^{2}\left\{4\eta_{1}^{2}+2\eta_{1}(1-\eta_{1})\upsilon-(1-\eta_{1})^{2}\upsilon^{2}\right\}\left\{\sqrt{\eta_{1}+(1-\eta_{1})\upsilon}-\sqrt{\eta_{1}}\right\}^{-2}\nonumber\\&\left\{2\eta_{1}+\sqrt{4\eta_{1}^{2}-\epsilon^{2}(1-\eta_{1})^{2}}\right\}^{-1}\left[\sqrt{4\eta_{1}^{2}-\epsilon^{2}(1-\eta_{1})^{2}}+2\sqrt{\eta_{1}^{2}+\eta_{1}(1-\eta_{1})\upsilon}+(1-\eta_{1})\upsilon\right]^{-1},\nonumber
\end{align}
where, due to the assumption that $\upsilon\leqslant(1+\sqrt{5})(1-\eta_{1})^{-1}\eta_{1}$ (in the statement of Theorem~\ref{thm:main_5}), we have 
\begin{align}
{}&4\eta_{1}^{2}+2\eta_{1}(1-\eta_{1})\upsilon-(1-\eta_{1})^{2}\upsilon^{2}
=-(1-\eta_{1})^{2}\left\{\upsilon+\frac{(\sqrt{5}-1)\eta_{1}}{1-\eta_{1}}\right\}\left\{\upsilon-\frac{(1+\sqrt{5})\eta_{1}}{1-\eta_{1}}\right\}\geqslant 0,\nonumber
\end{align}
thus ensuring that $\xi_{2}>\xi_{3}$. This finishes the proof of the fact that indeed, $\kappa$ can be chosen to lie inside each of the open intervals $(\xi_{1},\xi_{2})$ and $(\xi_{3},\xi_{4})$. With such a choice of $\kappa$, in the complement of the scenario described in \eqref{case_1}, we have $\dot{V}(a(t),b(t),c(t))\leqslant 0$ for all $(a(t),b(t),c(t))\in\hat{\mathcal{S}}$, with equality if and only if $(a(t),b(t),c(t))\in\hat{\Lambda}$.

Thus far, we have established that the function $V$, defined in \eqref{candidate_Lyapunov}, satisfies (for time $t\geqslant 0$)
\begin{equation}\label{V_main_properties}
\begin{cases}
&V(a,b,c)\geqslant 0 \text{ for all }(a,b,c)\in\hat{\mathcal{S}}, \text{ with equality iff }(a,b,c)\in\hat{\Lambda},\\
&\dot{V}(a(t),b(t),c(t))\leqslant 0 \text{ for }(\dot{a}(t),\dot{b}(t),\dot{c}(t))\in F(a(t),b(t),c(t)), \text{ with equality iff }(a(t),b(t),c(t))\in\hat{\Lambda}.
\end{cases}
\end{equation}

\textbf{Showing, using $V$, that any compact, connected, internally chain transitive invariant set is contained in $\hat{\Lambda}$:} Recall that $\varphi=\{\varphi(t): t\geqslant 0\}$ is a solution to a given differential inclusion $\dot{\mathbf{x}}(t)\in F(\mathbf{x}(t))$ for $t\geqslant 0$, with initial condition $\mathbf{x}(0)=\mathbf{x}_{0}$, if $t\mapsto \varphi(t)$ is absolutely continuous on $[0,\infty)$, $\dot{\varphi}(t)\in F(\varphi(t))$ for almost every $t\geqslant 0$, and $\varphi(0)=\mathbf{x}_{0}$. We let $\Phi(\mathbf{x}_{0})$ indicate the set of all such solutions, for each $\mathbf{x}_{0}$ belonging to the domain of definition of $F$. Given a closed, invariant set $L$, we also define $\Phi^{L}(\mathbf{x}_{0})\subset\Phi(\mathbf{x}_{0})$, for each $\mathbf{x}_{0}\in L$, to be the set of all solutions $\varphi$ that are contained entirely in $L$, i.e.\ 
\begin{equation}
\Phi^{L}(\mathbf{x}_{0})=\left\{\varphi\in\Phi(\mathbf{x}_{0}):\varphi(t)\in L \text{ for all }t\geqslant 0\right\}.\nonumber
\end{equation}
We let $\Phi^{L}_{t}(\mathbf{x}_{0})$ denote the set of $\varphi(t)$ for each $\varphi\in\Phi^{L}(\mathbf{x}_{0})$. Proposition 3.19 of \cite{benaim2005stochastic} may now be restated as follows, with the proof following \emph{mutatis mutandis} the proof of Proposition 3.19 of \cite{benaim2005stochastic} (with $\Phi$ replaced everywhere by $\Phi^{L}$):
\begin{prop}\label{prop:3.19_reformulated}
Let $L$ be a closed, invariant set corresponding to a given differential inclusion, $\dot{\mathbf{x}}(t)\in F(\mathbf{x}(t))$ for $t\geqslant 0$, and let $U$ be an open subset of $L$ (under the topology induced on $L$), such that its closure, $\overline{U}$, is compact. Suppose, for some $T>0$, we have $\varphi(T)\in U$ for each $\varphi\in\Phi^{L}(\mathbf{x}_{0})$, for each $\mathbf{x}_{0}\in\overline{U}$ (in other words, $\Phi^{L}_{T}(\overline{U})=\bigcup_{\mathbf{x}_{0}\in \overline{U}}\Phi^{L}_{t}(\mathbf{x}_{0})\subset U$). Then $U$ is a fundamental neighbourhood of some attractor for $\Phi^{L}$.
\end{prop}

Let $L$ be \emph{any} compact, connected, internally chain transitive invariant set (contained in $\hat{\mathcal{S}}$) correpsonding to the differential inclusion mentioned in the statement of Theorem~\ref{thm:main_5}, and let $\theta=\inf\{V(a,b,c):(a,b,c)\in L\}$. Since $L$ is compact and $V$ continuous on $\hat{\mathcal{S}}$ (this follows from the fact that $q_{\alpha}$ is continuous on $\mathcal{S}$, and each of $\hat{G}$ and $g$ is continuous on $[0,1]$), the value $\theta$ must be attained in $L$. Let $(a_{1},b_{1},c_{1})$ be \emph{any} point in $L$ with $V(a_{1},b_{1},c_{1})=\theta$. If possible, let $(a_{1},b_{1},c_{1})\notin\hat{\Lambda}$. From the definition of $\hat{\Lambda}$ in \eqref{hat{Lambda}_defn}, and the continuity of the function $q_{\alpha}$ on $\hat{\mathcal{S}}$, it is evident that $\hat{\Lambda}$ is a compact set. Since $(a_{1},b_{1},c_{1})\notin\hat{\Lambda}$ and any solution $\varphi\in\Phi(a_{1},b_{1},c_{1})$ is absolutely continuous in time, we can find some $T_{\varphi}>0$ such that $\varphi(t)\notin\hat{\Lambda}$ for all $t\in[0,T_{\varphi}]$. By the second assertion in \eqref{V_main_properties}, $\dot{V}(\varphi(t))<0$ for all $t\in[0,T_{\varphi}]$, or, in other words, $V(\varphi(t))$ is strictly decreasing in $t$ for all $t\in[0,T_{\varphi}]$. Since $L$ is invariant (see, for instance, Lemma 3.5 of \cite{benaim2005stochastic}), there exists at least one $\varphi\in\Phi(a_{1},b_{1},c_{1})$ such that $\varphi(t)\in L$ for each $t\geqslant 0$. For this specific $\varphi$, we have $V(\varphi(T_{\varphi}))<V(\varphi(0))=V(a_{1},b_{1},c_{1})=\theta$ and $\varphi(T_{\varphi})\in L$, contradicting the choice of $\theta$ as the infimum of the values assumed by $V$ on $L$. This leads to the conclusion that $(a_{1},b_{1},c_{1})\in\hat{\Lambda}$ for every $(a_{1},b_{1},c_{1})\in L$ with $V(a_{1},b_{1},c_{1})=\theta$. Moreover, from the first assertion in \eqref{V_main_properties}, we have $\theta=0$.

Let $\{\theta_{n}\}$ be \emph{any} sequence of positive reals with $\theta_{n}\downarrow 0$. We define $L_{n}=\{(a,b,c)\in L:V(a,b,c)<\theta_{n}\}$. Let us fix any $T>0$. As $V$ is continuous on $\hat{\mathcal{S}}$, each $L_{n}$ is open relative to $L$, and since $L_{n}\subset L$ where $L$ is compact, we have $\overline{L}_{n}\subset L$ as well, making $\overline{L}_{n}$ compact too. Our goal, now, is to show that $\varphi(T)\in L_{n}$ for every $\varphi\in\Phi^{L}(a,b,c)$, for every $(a,b,c)\in\overline{L}_{n}$, for each $n\in\mathbb{N}$. This is accomplished in the two paragraphs that follow, always keeping in mind that, for each $(a,b,c)\in\overline{L}_{n}\subset L$, we have $\varphi(T)\in L$ for each $\varphi\in\Phi^{L}(a,b,c)$, by definition of the set $\Phi^{L}$.

Let $(a,b,c)\in\overline{L}_{n}\cap\hat{\Lambda}$. For any $\varphi\in\Phi^{L}(a,b,c)$, we know, combining the two assertions stated in \eqref{V_main_properties}, that $0\leqslant V(\varphi(t))\leqslant V(\varphi(0))=V(a,b,c)=0$ for each $t>0$, implying that $V(\varphi(T))=0$. This, along with the assertion stated in the last sentence of the previous paragraph, implies that $\varphi(T)\in L_{n}$. 

On the other hand, if $(a,b,c)\in\overline{L}_{n}\setminus\hat{\Lambda}$, we have $V(a,b,c)\leqslant\theta_{n}$, by definition of $L_{n}$ and since $(a,b,c)\in\overline{L}_{n}$. Since each $\varphi\in\Phi^{L}(a,b,c)$ is absolutely continuous in time and $\hat{\Lambda}$ is compact, there exists some $T_{\varphi}>0$ such that $\varphi(t)\notin\hat{\Lambda}$ for each $t\in[0,T_{\varphi}]$. By the second assertion stated in \eqref{V_main_properties}, we have $\dot{V}(\varphi(t))<0$ for each $t\in[0,T_{\varphi}]$, so that $V(\varphi(t))$ is strictly decreasing in $t$ for all $t\in[0,T_{\varphi}]$. If $T\in[0,T_{\varphi}]$, we have
\begin{equation}
V(\varphi(T_{\varphi}))\leqslant V(\varphi(T))<V(\varphi(0))=V(a,b,c)\leqslant \theta_{n},\nonumber
\end{equation}
allowing us to conclude that $\varphi(T)\in L_{n}$. Suppose, now, that $T>T_{\varphi}$. Once again, by the second assertion stated in \eqref{V_main_properties}, we have $\dot{V}(\varphi(t))\leqslant 0$ for each $t\in(T_{\varphi},T]$, so that $V(\varphi(t))$ is monotonically decreasing in $t$ for all $t\in(T_{\varphi},T]$. Therefore, we can write
\begin{equation}
V(\varphi(T))\leqslant V(\varphi(T_{\varphi}))<V(\varphi(0))=V(a,b,c)\leqslant \theta_{n},\nonumber
\end{equation}
once again leading to the conclusion that $\varphi(T)\in L_{n}$. 

We have proved that $\varphi(T)\in L_{n}$ for each $\varphi\in\Phi^{L}(a,b,c)$, for each $(a,b,c)\in\overline{L}_{n}$, for every $n\in\mathbb{N}$. By Proposition~\ref{prop:3.19_reformulated}, we conclude that $L_{n}$ must be a fundamental neighbourhood of some attractor for $\Phi^{L}$. However, by Proposition 3.20 of \cite{benaim2005stochastic}, $L$ cannot contain, as a proper subset, any attracting set (and hence, any attractor) for $\Phi^{L}$, since $L$ has been assumed to be internally chain transitive. Therefore, we must have $L=L_{n}$ for each $n$, leading to $L=\bigcap_{n}L_{n}=\{(a,b,c)\in L:V(a,b,c)=0\}$. Since we have already shown that, for each $(a,b,c)\in L$ with $V(a,b,c)=\theta=0$, we have $(a,b,c)\in\hat{\Lambda}$, we conclude that $L\subset\hat{\Lambda}$.
\end{proof}

\begin{proof}[Proof of Theorem~\ref{thm:main_4}]
In this special scenario, instead of \eqref{sa_growing_1}, we write:
\begin{align}
\frac{A_{n+1}}{n+1}={}&\frac{A_{n}}{n}+\frac{k_{n}}{s_{n+1}}\left[\hat{\epsilon}_{n+1,1}+\tilde{\delta}_{n,1}+\eta\hat{h}_{1}\left(\frac{A_{n}}{n},\frac{B_{n}}{s_{n}},\frac{C_{n}}{s_{n}}\right)\right],\label{sa_growing_1_modified}
\end{align}
where $\hat{h}_{1}$ is as defined via \eqref{hat{h}_defn_growing_size} and \eqref{hat{H}_{1}_defn}, $\hat{\epsilon}_{n+1,1}$ is as defined in \eqref{error_defns_sa_growing}, and we define 
\begin{align}\label{new_error_defns_special}
\tilde{\delta}_{n,1}&=\frac{s_{n+1}}{(n+1)k_{n}}\left\{H_{1,n}\left(\frac{A_{n}}{n},\frac{B_{n}}{s_{n}},\frac{C_{n}}{s_{n}}\right)-\hat{H}_{1}\left(\frac{A_{n}}{n},\frac{B_{n}}{s_{n}},\frac{C_{n}}{s_{n}}\right)\right\}+\left\{\frac{s_{n+1}}{(n+1)k_{n}}-\eta\right\}\hat{h}_{1}\left(\frac{A_{n}}{n},\frac{B_{n}}{s_{n}},\frac{C_{n}}{s_{n}}\right).
\end{align}
Right away, we see that $\tilde{\delta}_{n,1}=\hat{\delta}_{n,1}+\{(n+1)^{-1}k_{n}^{-1}s_{n+1}-\eta\}\hat{h}_{1}(n^{-1}A_{n},s_{n}^{-1}B_{n},s_{n}^{-1}C_{n})$, where $\hat{\delta}_{n,1}$ is as defined in \eqref{error_defns_sa_growing}. Thus \eqref{sa_growing_1_modified}, along with \eqref{sa_growing_2}, now boils down to a stochastic approximation corresponding to a typical autonomous ODE, as opposed to a differential inclusion. We can, therefore, apply Theorem~\ref{thm:borkar_a.s.} to this set-up.

Having proved Theorem~\ref{thm:main_3}, we already know that each of $\sum_{n\geqslant N}s_{n+1}^{-1}k_{n}(\hat{\epsilon}_{n+1,1},\hat{\epsilon}_{n+1,2},\hat{\epsilon}_{n+1,3})^{T}$ and $\sum_{n\geqslant N}s_{n+1}^{-1}k_{n}\sum_{j\in[3]}|\hat{\delta}_{n,j}|$ converges almost surely when the second criterion stated in \eqref{k_{n}_sequence_conditions} is true and one of \eqref{A1}, \eqref{A2} and \eqref{A3} is satisfied. We also have (since $\hat{H}_{1}(a,b,c)\in[0,1]$ and $a\in[0,1]$ for $(a,b,c)\in\hat{\mathcal{S}}$):
\begin{align}
\sum_{n\geqslant N}\frac{k_{n}}{s_{n+1}}\left|\left\{\frac{s_{n+1}}{(n+1)k_{n}}-\eta\right\}\hat{h}_{1}\left(\frac{A_{n}}{n},\frac{B_{n}}{s_{n}},\frac{C_{n}}{s_{n}}\right)\right|\leqslant \sum_{n\geqslant N}\frac{k_{n}}{s_{n+1}}\left|\frac{s_{n+1}}{(n+1)k_{n}}-\eta\right|,\nonumber
\end{align}
which converges because of \eqref{k_{n}_eta_conditions}. Assumption~\eqref{borkar_3} is hereby verified.

Since one of \eqref{A1}, \eqref{A2} and \eqref{A3} is assumed to hold, the function $g$ is Lipschitz on $[0,1]$, while the function $q_{\alpha}$, defined via \eqref{q_defn}, is in $\mathcal{C}^{(1)}(\hat{\mathcal{S}})$. Therefore, each of $\hat{h}_{1}$, $\hat{h}_{2}$ and $\hat{h}_{3}$ is Lipschitz on $\hat{\mathcal{S}}$ (as evident from \eqref{hat{h}_defn_growing_size}, \eqref{hat{H}_{1}_defn}, \eqref{hat{H}_{2}_defn} and \eqref{hat{H}_{3}_defn}. Consequently, the drift function appearing in \eqref{ODE_special_case} is Lipschitz on $\hat{\mathcal{S}}$ as well. This verifies Assumption~\eqref{borkar_1}. Finally, Assumption~\eqref{borkar_2} is satisfied since the first two criteria of \eqref{k_{n}_sequence_conditions} are assumed to hold. This completes the proof.
\end{proof}

\begin{proof}[Proof of Proposition~\ref{prop:hat{G}_unique_root_examples}]
When $f(x)=x$, we have $g(x)=(2p-1)x+(1-p)$, which, by \eqref{hat{G}_defn}, leads to: 
\begin{align}
{}&\hat{G}(q)
=2q^{2}(1-q)(2p-1)+(1-2p)q(1-q)+2\alpha(p-1)q+\alpha(1-p) \text{ for all }q\in[0,1]\nonumber\\
\implies{}&\hat{G}'(q)=
(6q-6q^{2}-1)(2p-1)-2\alpha(1-p) \text{ for all }q\in[0,1],\label{hat{G}_derivative_f(x)=x}
\end{align}
so that, as in the proof of Proposition~\ref{prop:examples_g_unique_root_of_G}, the derivative $\hat{G}'(q)$ is symmetric around $q=1/2$. Before we proceed any further, we note that $\hat{G}(0)=\alpha(1-p)>0$ and $\hat{G}(1)=-\alpha(1-p)<0$, so that
\begin{enumerate*}
\item the curve $y=\hat{G}(x)$ intersects the $x$-axis at least once, and
\item the curve $y=\hat{G}(x)$ lies \emph{above} the $x$-axis at $x=0$ and \emph{beneath} the $x$-axis at $x=1$. 
\end{enumerate*}
From the first of these, it suffices for us to prove that the curve $y=\hat{G}(x)$ can intersect the $x$-axis at most once. The second observation comes in handy in some of the arguments that follow.

We now consider $p\in(1/2,1)$, so that the minimum value attained by $\hat{G}'$ equals $\hat{G}'(0)=\hat{G}'(1)=-(2p-1)-2\alpha(1-p)$ and the maximum value attained by $\hat{G}'$ equals $\hat{G}'(1/2)=(2p-1)/2 -2\alpha(1-p)$. If $4\alpha\geqslant(2p-1)/(1-p)$, we have $\hat{G}'(q)\leqslant 0$ for all $q\in[0,1]$ (with possible equality only at $q=1/2$), implying that $\hat{G}$ is strictly decreasing for $q\in[0,1]$ and hence the curve $y=\hat{G}(x)$ can intersect the $x$-axis at most once. If $4\alpha<(2p-1)/(1-p)$, we conclude, from the symmetry of the expression in \eqref{hat{G}_derivative_f(x)=x} around $q=1/2$ and its quadratic nature, that its roots are given by some $0<\gamma_{1}<1/2<\gamma_{2}<1$, with $\gamma_{1}+\gamma_{2}=1$, with
\begin{enumerate*}
\item $\hat{G}'(q)<0$ for all $q\in[0,\gamma_{1})$, so that $\hat{G}$ is strictly decreasing for $q\in[0,\gamma_{1})$, 
\item $\hat{G}'(q)>0$ for all $q\in(\gamma_{1},\gamma_{2})$, so that $\hat{G}$ is strictly increasing for $q\in(\gamma_{1},\gamma_{2})$,
\item and $\hat{G}'(q)<0$ for all $q\in(\gamma_{2},1]$, so that $\hat{G}$ is strictly decreasing for $q\in(\gamma_{2},1]$.
\end{enumerate*}
If $\hat{G}(\gamma_{2})<0$, the behaviour described above ensures that the curve $y=\hat{G}(x)$ intersects the $x$-axis at a unique value of $x$ lying in the interval $(0,\gamma_{1})$, and no further intersection happens. Likewise, if $\hat{G}(\gamma_{1})>0$, the behaviour described above ensures that the curve $y=\hat{G}(x)$ intersects the $x$-axis at a unique value of $x$ lying in the interval $(\gamma_{2},1)$, and no intersection happens sooner than that. These two criteria give rise to \eqref{root_inequalities_hat{G}}.

Finally, we consider $f(x)=e^{x-1}$ for all $x\in[0,1]$, so that we have, from \eqref{hat{G}_defn}:
\begin{align}
\hat{G}(q)
={}&2(2p-1)q(1-q)e^{q-1}+(1-2p)q(1-q)+\alpha (2p-1)e^{q-1}-\alpha q+\alpha(1-p).\label{hat{G}_f_exponential}
\end{align}
We assume $p\in(1/2,1)$ and $\alpha>0.15218(2p-1)/(1-p)$. The function $q\mapsto 1+(1-2p)e^{q-1}$ is strictly decreasing in $q$ for $q\in[0,1]$, with minimum value, attained at $q=1$, equal to $2(1-p)>0$, proving that this function is strictly positive for all $q\in[0,1]$. Next, we consider the function $q \mapsto 1-2q-2(1-q-q^{2})e^{q-1}$, whose derivative equals $-2+2(3q+q^{2})e^{q-1}$, and second derivative equals $2(3+5q+q^{2})e^{q-1}>0$, proving that the function is strictly convex. Moreover, its derivative equals $-2$ at $q=0$, and $+6$ at $q=1$, proving that the function strictly decreases and then strictly increases as $q$ goes from $0$ to $1$. Its minimum, therefore, must be attained at the unique value of $q$, in $(0,1)$, for which its derivative vanishes, i.e.\ $q$ satisfying the equation $(3q+q^{2})e^{q-1}=1$. Numerical estimates reveal that the minimum of $q \mapsto 1-2q-2(1-q-q^{2})e^{q-1}$ is attained at $q\approx 0.48206$, and this minimum value equals $\approx -0.30436$. The maximum value that $(1-2p)\{(1-2q)-2(1-q-q^{2})e^{q-1}\}$ can attain as $p\in(1/2,1)$ remains fixed and $q$ varies over $[0,1]$, is thus equal to $0.30436(2p-1)$. Now, setting $\alpha>0.15218(2p-1)/(1-p)$, and using the observations above, we obtain:
\begin{align}
\hat{G}'(q)
={}&(1-2p)\{(1-2q)-2(1-q-q^{2})e^{q-1}\}-\alpha\{1+(1-2p)e^{q-1}\}\nonumber\\
<{}&(1-2p)\{(1-2q)-2(1-q-q^{2})e^{q-1}\}-\frac{0.15218(2p-1)}{1-p}\{1+(1-2p)e^{q-1}\}<0.\nonumber
\end{align}
This proves that when $p\in(1/2,1)$ and $\alpha>0.15218(2p-1)/(1-p)$, the function $\hat{G}$ is strictly decreasing in $[0,1]$ and hence, may have at most one root in $[0,1]$. Since $\hat{G}(0)=\alpha(2p-1)e^{-1}+\alpha(1-p)>0$ and $\hat{G}(1)=-\alpha(1-p)<0$ from \eqref{hat{G}_f_exponential}, hence $\hat{G}$ must have at least one root in $(0,1)$. Together, these two observations imply that $\hat{G}$ has a unique root in $[0,1]$.
\end{proof}

\section{Acknowledgements}
A.\ Roy acknowledges support from the IIM-K SGRP Research Grant (No. SGRP/2025-26/22) for the accomplishment of this project.

\bibliography{ERW_bib}
\end{document}